\documentclass{amsart}

\usepackage{amssymb}
\usepackage{graphicx}
\usepackage[latin1]{inputenc}
\usepackage{amsmath}
\usepackage{amsthm}
\usepackage{accents}
\usepackage{mathrsfs}
\usepackage{subfigure}
\usepackage{tikz}
\usepackage{mathdots}

\newtheorem{theorem}{Theorem}[section]
\newtheorem{lemma}[theorem]{Lemma}
\newtheorem{corollary}[theorem]{Corollary}
\newtheorem{proposition}[theorem]{Proposition}

\theoremstyle{definition}
\newtheorem{definition}[theorem]{Definition}
\newtheorem{example}[theorem]{Example}

\theoremstyle{remark}
\newtheorem{remark}[theorem]{Remark}

\numberwithin{equation}{section}

\newcommand{\RR}{\mathbb R}
\newcommand{\f}[1]{\mathbf{#1}}

\newcommand{\Quad}{\mathcal{Q}} 
\newcommand{\Mesh}{\mathcal{M}} 
\newcommand{\globalspace}{\mathcal{S}} 
\newcommand{\Facep}{\mathcal{F}} 
\newcommand{\Vertices}{\mathcal{V}} 
\newcommand{\Edge}{\mathcal{E}} 
\newcommand{\Q}{Q} 
\newcommand{\vertex}{{\f v}} 
\newcommand{\edge}{\varepsilon} 
\newcommand{\midp}{{\f m}} 
\newcommand{\facep}{{\f x}} 
\newcommand{\edgenorm}{\mathbf{n}_{\edge}} 
\newcommand{\edgenormi}{\mathbf{n}_{\edge_i}} 
\newcommand{\hatquad}{\widehat{Q}} 
\newcommand{\param}{\mathbf{F}} 
\newcommand{\func}{\varphi} 
\newcommand{\dx}{\partial_1} 
\newcommand{\dy}{\partial_2} 
\newcommand{\dn}{\partial_{\edgenorm}} 
\newcommand{\dni}{\partial_{\edgenormi}} 
\newcommand{\shaperegularityparameter}{\rho}
\newcommand{\Space}{P} 
\newcommand{\Dual}{\Lambda} 
\newcommand{\basisf}{\beta} 
\newcommand{\basisfn}[2]{\basisf_{#1,#2}} 
\newcommand{\basisfhat}{\widehat{\basisf}} 

\renewcommand{\arraystretch}{1.2}

\begin{document}

\title{A family of $C^1$ quadrilateral finite elements}

\author{Mario Kapl}
\address{Johann Radon Institute for Computational and Applied Mathematics, Austrian Academy of Sciences, Austria}
\curraddr{}
\email{mario.kapl@ricam.oeaw.ac.at}
\thanks{}

\author{Giancarlo Sangalli}
\address{Dipartimento di Matematica ``F. Casorati'', Universit\`a degli Studi di Pavia, Italy; Istituto di Matematica Applicata e Tecnologie Informatiche ``E. Magenes'' (CNR), Italy}
\curraddr{}
\email{giancarlo.sangalli@unipv.it}
\thanks{}

\author{Thomas Takacs}
\address{Institute of Applied Geometry, Johannes Kepler University Linz, Austria}
\curraddr{}
\email{thomas.takacs@jku.at}
\thanks{}


\date{\today}

\dedicatory{}

\begin{abstract}
We present a novel family of $C^1$ quadrilateral finite elements, which define global $C^1$ spaces over a general quadrilateral mesh with vertices of arbitrary valency. The elements  extend the construction by Brenner and Sung~\cite{BrSu05}, which is based on polynomial elements of tensor-product degree $p\geq 6$, to all degrees~$p \geq 3$. Thus, we call the family of $C^1$ finite elements \emph{Brenner-Sung quadrilaterals}. The proposed $C^1$ quadrilateral can be seen as a special case of the Argyris isogeometric element of~\cite{KaSaTa19}. The quadrilateral elements possess similar degrees of freedom as the classical Argyris triangles~\cite{ArFrSc68}. Just as for the Argyris triangle, we additionally impose $C^2$ continuity at the vertices. In this paper we focus on the lower degree cases, not covered in~\cite{BrSu05}, that may be desirable  for their lower computational cost and better conditioning of the basis: We  consider  indeed the polynomial quadrilateral of (bi-)degree~$5$, and  the  polynomial degrees~$p=3$ and $p=4$ by employing a splitting into $3\times3$ or $2\times2$ polynomial pieces, respectively. 

The proposed elements reproduce polynomials of total degree $p$. We show that the space provides optimal approximation order. Due to the interpolation properties, the error bounds are local on each element. In addition, we describe the construction of a simple, local basis and give for $p\in\{3,4,5\}$ explicit formulas for the B\'{e}zier or B-spline coefficients of the basis functions. Numerical experiments by solving the biharmonic equation demonstrate the potential of the proposed $C^1$ quadrilateral finite element for the numerical analysis of fourth order problems, also indicating that (for $p=5$) the proposed element  performs comparable or in general even better than the Argyris triangle with respect to the number of degrees of freedom.
\end{abstract}
 
\maketitle

\section{Introduction} \label{sec:introduction}

Using a standard Galerkin approach for the numerical analysis of high order problems, globally smooth function spaces are needed. E.g., for solving fourth order partial differential equations (PDEs) via the finite element method (FEM), $C^1$ finite element spaces are required. In the case of triangular meshes, two well-known examples are the Argyris element~\cite{ArFrSc68} and the Bell element~\cite{Be69}. Both elements require polynomials of degree~$p \geq 5$, and are additionally $C^2$ at the vertices. While the normal derivative along an edge is of degree~$p-1$ for the Argyris element, its degree reduces to $p-2$ for the Bell element. This leads for instance in case of polynomial degree~$p=5$ to the fact that the Argyris triangular space possesses six degrees of freedom for each vertex and one degree of freedom for each edge, while the Bell triangular space just has six degrees of freedom for each vertex and no additional degrees of freedom for the edges. For more details on the Argyris and Bell triangular element as well as on other $C^1$ triangular finite elements, we refer to the books~\cite{BrSc07,Ci02}. $C^1$ finite element spaces of lower polynomial degree are in general based on splines, which are constructed over general triangulations, see ~\cite{LaSc07}.

The design of $C^1$ finite elements over quadrilateral meshes is in general more challenging compared to the case of triangular meshes, in particular with respect to the selection of the degrees of freedom. Examples of $C^1$ quadrilateral elements are \cite{BeMa14,BoFoSc65,BrSu05,Ma01}. The Bogner-Fox-Schmit element~\cite{BoFoSc65} is a simple bivariate Hermite type $C^1$ construction which works for low polynomial degrees such as $p=3$, but is limited to tensor-product meshes. In contrast, the $C^1$ elements~\cite{BeMa14,BrSu05,Ma01} are applicable to more general quadrilateral meshes, but require a polynomial degree $p \geq 6$ in case of~\cite{BrSu05} and a polynomial degree~$p\geq5$ (for some specific settings just $p=4$) in case of~\cite{BeMa14,Ma01}. The degrees of freedom for the finite element space~\cite{BrSu05} are selected similar to the Argyris triangular finite element space~\cite{ArFrSc68} by enforcing additionally $C^2$-continuity at the vertices.

In contrast, the functions in~\cite{BeMa14,Ma01} are just $C^1$ at the vertices and the degrees of freedom are defined by means of the concept of minimal determining sets (cf. \cite{LaSc07}), which is a common strategy for the construction of $C^1$ splines over triangular meshes, see also~\cite{LaSc07}. A different but related problem is the construction of $C^1$ function spaces over general quadrilateral meshes for the design of surfaces, such as in~\cite{GrMa87,Pe90,Pe91,Re95}. The methods are based on the concept of geometric continuity~\cite{Pe02}, which is a well-known tool in computer aided geometric design for generating smooth complex surfaces. 

An alternative to FEM is the use of isogeometric analysis (IgA), which was introduced in~\cite{HuCoBa05}, and employs the same spline function space for describing the physical domain of interest and for representing the solution of the considered PDE, see e.g.~\cite{CoHuBa09,HuCoBa05} for more details. In case of a single patch geometry, this allows the direct discretization of fourth order PDEs~\cite{TaDeQu14}, such as the Kirchhoff-Love shells, e.g.~\cite{KiBlLi09,KiBaHs10}, the Navier-Stokes-Korteweg equation, e.g.~\cite{GoHuNo10}, problems of strain gradient elasticity, e.g.~\cite{gradientElast2011,KhakaloNiiranenC1}, or the Cahn-Hilliard equation, e.g.~\cite{GoCaBa08}, by just using $C^1$ splines. In case of multi-patch geometries with possibly extraordinary vertices, i.e. vertices with a patch valency different to four, the design of smooth spline spaces is challenging and is the topic of current research. 

Depending on the used type of parametrizations for the single patches of the given unstructured quadrilateral mesh, different techniques for the design of a $C^1$ spline space over this mesh have been developed. Possible examples in the case of planar, unstructured quadrilateral meshes are to use $C^1$ multi-patch parametrizations with a singularity at an extraordinary vertex, e.g.~\cite{NgPe16,ToSpHu17}, multi-patch parametrizations which are $C^1$ except in the vicinity of an extraordinary vertex, e.g.~\cite{KaNgPe17,KaPe17,KaPe18,NgKaPe15}, or multi-patch parametrizations which have to be just $C^0$ at all interfaces, e.g.~\cite{BlMoVi17,ChAnRa18,ChAnRa19,CoSaTa16,KaBuBeJu16,KaSaTa17,KaSaTa17b,KaSaTa19,KaViJu15,MoViVi16}. For more details about existing $C^1$ constructions for unstructured quadrilateral meshes, we refer to the recent survey article~\cite{KaSaTa19b}. Beside this, in~\cite{BuJuMa15,SaTaVa16,ScThEv14}, different approaches for the construction of smooth spline functions of degree~$p$ are presented, which are $C^s$ ($1 \leq s \leq p-1$) everywhere, except in the vicinity of an extraordinary vertex, where they are just $C^0$. 

In this work, we present a family of $C^1$ quadrilateral finite elements, that are the low-degree (for $p\in\{3,4,5\}$) counterpart of the quadrilateral finite elements proposed in \cite{BrSu05} (for $p\geq6$) by Brenner and Sung. The interest for the low-degree case  is that the computational cost for the linear system formation (due to numerical quadrature) and solution (that depends on the matrix conditioning) is more favorable. We refer to these elements (the ones in \cite{BrSu05}  and the new ones) as {\em Brenner-Sung (BS) quadrilaterals}. These quadrilateral elements, in turn, are included in the isogeometric family of~\cite{KaSaTa19}, and, indeed, the lower degrees~$p\in\{3,4\}$ construction is based on tensor-product splines.

The BS quadrilateral possesses similar degrees of freedom as the classical $C^1$ Argyris triangle~\cite{ArFrSc68}. An advantage of the quadrilateral construction over the triangular one is the simpler extension to the lower polynomial degrees~$p=3$ and $p=4$ by just using tensor-product spline without the need of special splits for the mesh elements.
    
While in~\cite{KaSaTa19} the optimal approximation properties of the $C^1$ isogeometric spline space is just numerically shown, in this work the optimal approximation order of the BS quadrilateral space is proven. A further extension to~\cite{KaSaTa19} is that for some particular cases the B\'{e}zier or spline coefficients of the basis functions are explicitly given by simple formulas. Several numerical tests of solving the biharmonic equation also show the potential of the BS quadrilateral space for the numerical analysis of fourth order PDEs.

The outline of this paper is as follows. Section~\ref{sec:mesh} introduces the quadrilateral mesh which will be used throughout the paper. In Section~\ref{sec:Argyris}, the construction of the BS quadrilateral is described, focusing first on the  bi-quintic polynomials, and then generalizing to  splines, which allow the use of the lower polynomial degrees~$p=3$ and $p=4$. Section~\ref{sec:Argyris} also discusses the connection of the BS  quadrilateral with two well-known triangular finite elements, namely with the Argyris triangle~\cite{ArFrSc68} and with the Bell triangle~\cite{Be69}. In Section~\ref{sec:gs-appr-prop} we analyze the approximation properties of the BS  quadrilateral space. Then, Sections~\ref{sec:polynomial-representation} and~\ref{sec:Argyris_spline} describe the design of local basis functions of the BS quadrilateral space for the case of polynomials and its extension for the case of splines, respectively, giving for the low-degree $p\in\{3,4,5\}$ cases the  explicit  B\'{e}zier and spline coefficients of the basis functions. The isoparametric extension of the BS quadrilateral, and its relation to the isogeometric element of~\cite{KaSaTa19}, is briefly discussed in Section~\ref{sec:extension}. Finally, we present in Section~\ref{sec:examples} numerical benchmarks on the biharmonic equation with different quadrilateral meshes, and conclude the paper in Section~\ref{sec:conclusion}.

\section{Quadrilateral mesh} \label{sec:mesh}

We consider planar domains that allow meshing by quadrilaterals. Note that a generalization to domains with curved boundaries is possible with some additional care. We refer the reader to~\cite{BeMa14,KaSaTa17b}, where such discretizations were developed, see also Section \ref{sec:extension}.

Let $\Omega\subset\RR^2$ be an open, planar and connected region, which allows a quadrangulation, as defined below. This is the case if the boundary is piecewise linear, including all inner boundaries, if $\Omega$ is not simply connected. The coordinates in physical space are given as $(x_1,x_2)$. A quadrilateral mesh is a tuple
\begin{equation*}
 \Mesh = (\Quad,\Edge,\Vertices),
\end{equation*}
consisting of a set of elements $\Quad$, edges $\Edge$ and vertices $\Vertices$ satisfying the following properties.
\begin{itemize}
 \item Each vertex is a point in the plane, that is $\Vertices \subset \mathbb{R}^2$. 
 \item Each edge $\edge\in \Edge$ is an open segment, and there exist two vertices $\vertex_1,\vertex_2 \in \Vertices$ such that
\begin{equation*}
 \edge = \{ (1-s)\, \vertex_1 + s\,\vertex_2 : s\in\left]0,1\right[ \}.
\end{equation*}
 \item Each element $\Q \in \Quad$ is a convex, non-empty, open
   quadrilateral; there exist four vertices $\vertex_1, \ldots, \vertex_4 \in
   \Vertices$, four edges $\edge_1, \ldots, 
 \edge_4 \in \Edge$, with edge $\edge_i$ connecting $\vertex_{i}$ with
 $\vertex_{i+1}$ (modulo $4$), given in counter-clockwise order; the
 element admits a parametrization by a $\param_{\Q} : \hatquad \rightarrow
 \overline{\Q}$ which is bilinear on
 $\hatquad = \left[0,1\right]^2$, precisely:
 \begin{equation}
   \label{eq:bilinear-parametrization}
   \param_{\Q}(\xi_1,\xi_2) = (1-\xi_1)(1-\xi_2) \, \vertex_1 + \xi_1 (1-\xi_2) \, \vertex_2 + \xi_1 \xi_2 \, \vertex_3 + (1-\xi_1) \xi_2 \, \vertex_4.
 \end{equation}
 See Fig.~\ref{fig:patch} for a visualization.
\item It holds
\begin{equation*}
 \overline\Omega = \bigcup_{\Q\in\Quad} \overline{\Q}
\end{equation*}
and all intersections of different mesh elements are empty, i.e., for all $X,X'\in\Quad\cup\Edge\cup\Vertices$, with $X\neq X'$, we have $X\cap X'=\emptyset$.
\end{itemize}
 
\begin{figure}[htb]
\centering
\resizebox{0.7\textwidth}{!}{
 \begin{tikzpicture}
  \coordinate(A) at (0,0); \coordinate(B) at (2.5,-0.25); \coordinate(C) at (-0.3,2.7); \coordinate(D) at (3,3); 
  \draw[thick] (A) -- (B); \draw[thick] (C) -- (D);  \draw[thick] (A) -- (C); \draw[thick] (B) -- (D);
  
  \draw[thick] (-3,0) -- (-3,2.5) -- (-5.5,2.5) -- (-5.5,0) -- (-3,0);
  \draw[->] (-5.3,0.2) -- (-4.3,0.2);
  \draw[->] (-5.3,0.2) -- (-5.3,1.2);
  \node at (-4.3,0.5) {$\xi_1$};
  \node at (-5,1.2) {$\xi_2$};
  
  \node at (-4.25,1.25) {$\hatquad$};
  \node at (-1.75,1.6) {$\param_{\Q}$};
  \draw[thick,->] (-2.5,1.25) -- (-1,1.25);
  \node at (1.25,1.25) {$\Q$};
  
  \fill (A) circle (2pt); \fill (B) circle (2pt); \fill (C) circle (2pt); \fill (D) circle (2pt);
  \node[left] at (A) {$\vertex_1$};
  \node[right] at (B) {$\vertex_2$};
  \node[right] at (D) {$\vertex_3$};
  \node[left] at (C) {$\vertex_4$};
  
  \draw[->] (0.1,0.1) -- (1.1,0);
  \draw[->] (0.1,0.1) -- (-0.011,1.1);
  \node at (1.0,0.28) {$\xi_1$};
  \node at (0.28,1.0) {$\xi_2$};
  
  \node at (1.25,-0.3) {$\edge_1$};
  \node at (1.35,3) {$\edge_3$};
  \node at (-0.375,1.35) {$\edge_4$};
  \node at (3,1.35) {$\edge_2$};
  \end{tikzpicture}
 }
\caption{Visualization of the mapping $\param_\Q$ for a quadrilateral~$\Q$, with vertices, edges, parameter domain and local coordinates.}
\label{fig:patch}
\end{figure}
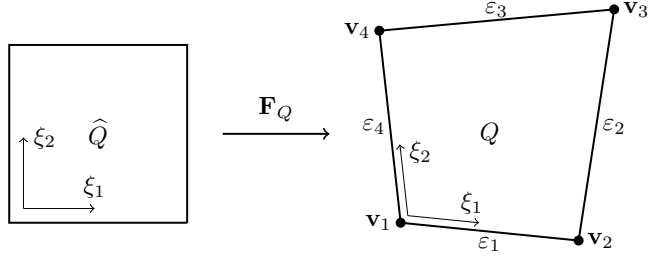

The last condition means that there are no hanging vertices in the quadrilateral mesh, i.e., all neighboring quadrilaterals share an entire edge or a vertex in their closure.

Given a $\Q\in\Quad$, we introduce the following notation: We denote by 
\begin{equation*}
 \f{t}^{(i)} = (t_1^{(i)},t_2^{(i)})^T = \vertex_{i+1} - \vertex_i
\end{equation*}
the vector corresponding to the edge $\edge_i$, and define
\begin{equation*}
 a^{(i)}=\det (\f{t}^{(i-1)}, \f{t}^{(i)} ),
\end{equation*}
with indices modulo $4$. Furthermore $h_{\edge_i} = \| \f{t}^{(i)}\|$ denotes the length of the corresponding edge $\edge_i$, $h_\Q = \max_{i}h_{\edge_i}$, and $\shaperegularityparameter_{\Q} $ its minimum angle defined as follows: If $\Q$ is a degenerate quadrilateral (that is a triangle) then $\shaperegularityparameter_{\Q} =0$, otherwise $\shaperegularityparameter_{\Q} $ is the minimum of the angles of the four triangles that are formed by the edges and the diagonals of
$\Q$. We assume however that each $\Q\in\Quad$ is non-degenerate, indeed the following holds.
\begin{proposition}\label{prop:shape-regularity}
 For each $\Q\in\Quad$, for each $i=1,\ldots,4$,
  \begin{equation}\label{eq:internal-angle}
    2 \shaperegularityparameter_{\Q} \leq \angle(\f t^{(i)},\f
    t^{(i-1)}) \leq \pi -  2 \shaperegularityparameter_{\Q} ,
  \end{equation}
 and 
  \begin{equation}\label{eq:edge-length}
c_1 (\shaperegularityparameter_{\Q})\,     h_\Q  \leq  h_{\edge_i},
  \end{equation}
 furthermore, for all $ (\xi_1,\xi_2) \in \hatquad$,
  \begin{equation}\label{eq:det_jacobian_F}
  c_2 (\shaperegularityparameter_{\Q})\,
h^2_\Q  \leq  \det(\nabla
    \param_\Q  (\xi_1,\xi_2)),
  \end{equation}
  where $c_1 (\shaperegularityparameter_{\Q}),c_2 (\shaperegularityparameter_{\Q}) \in\mathbb{R}$ are constants that depend only on $\shaperegularityparameter_{\Q}$ and are strictly positive for $\shaperegularityparameter_{\Q}>0$.
 \end{proposition}
\begin{proof}
Considering the split of $\Q$ into the four triangles formed by the edges and the diagonals, the bounds \eqref{eq:internal-angle} on $\angle(\f t^{(i)},\f t^{(i-1)})$ are straightforward. We also have, for $i\in\{1,2,3,4\}$ (modulo $4$),
\begin{equation*}
  \sin (\shaperegularityparameter_{\Q})     h_{\edge_{i+1}} \leq
  h_{\edge_i}\text{ and } \sin (\shaperegularityparameter_{\Q})     h_{\edge_{i-1}} \leq  h_{\edge_i},
\end{equation*}
repeating the same argument twice
\begin{equation*}
  \sin^2 (\shaperegularityparameter_{\Q})     h_{\edge_{i+2}} \leq  h_{\edge_i},
\end{equation*}
therefore, for all $j\in\{1,2,3,4\}$
\begin{equation*}
  \sin^2 (\shaperegularityparameter_{\Q})     h_{\edge_{j}} \leq  h_{\edge_i}
\end{equation*}
which gives the lower bound \eqref{eq:edge-length}. We have by direct calculation
\begin{equation}\label{eq:gradient-F}
 \nabla \param_{\Q} (\xi_1,\xi_2)= \left  [ \f{t}^{(1)} - \xi_2(\f{t}^{(1)} +
 \f{t}^{(3)})  \quad  -\f{t}^{(4)} + \xi_1(\f{t}^{(2)} + \f{t}^{(4)}) \right ] 
\end{equation}
and so, $ \det(\nabla \param_\Q) $ being a bilinear polynomial, its extrema are attained at the vertices of $\hatquad $, that is
\begin{equation*}
 \min(  \det(\nabla \param_\Q) ) = \min \{a^{(i)}, i=1,\ldots,4 \}.
\end{equation*}
Since
\begin{equation*}
 a^{(i)} = h_{\edge_i} h_{\edge_{i-1}} \sin(\angle(\f t^{(i)},\f t^{(i-1)})),
\end{equation*}
the lower bound \eqref{eq:det_jacobian_F} follows from \eqref{eq:internal-angle} and \eqref{eq:edge-length}.
\end{proof}
The condition $\shaperegularityparameter_{\Q} > 0 $ also implies that the parametrization
$\param_{\Q} : \hatquad \rightarrow \overline{\Q}$ is regular, that is, its inverse  $\param_{\Q}^{-1} : \overline{\Q} \rightarrow\hatquad$ has bounded derivatives too. This follows from~\eqref{eq:det_jacobian_F}.

Similar to~\cite{BrSu05}, we assume that the quadrilateral mesh $\Mesh$ is \emph{shape regular}, that is
\begin{equation}
  \label{eq:mesh-is-shape-regular}
  \shaperegularityparameter = \inf_{\Q\in\Quad} \shaperegularityparameter_{\Q} > 0.
\end{equation}

\section{$C^1$ BS quadrilateral elements} \label{sec:Argyris}

In the following we recall the definition of BS quadrilaterals from
\cite{BrSu05}, extend to the lower degree  cases, and define the
associated piecewise polynomial $C^1$ space over the
domain of interest $\Omega$, given a quadrilateral mesh $\Mesh$. In
our presentation we loosely follow the style of~\cite{BrSc07,Ci02}. The BS  quadrilateral of degree $p\geq5$ is constructed from  bi-quintic polynomials with normal derivatives across interfaces that are polynomials of degree $p-1$. For $p=5$ the degrees of freedom are given as $C^2$-data at the vertices, 
normal derivatives at the edge midpoints, as well as interior point evaluations. This is in accordance with the degrees of freedom of the Argyris triangle, 
see~\cite{ArFrSc68}. Moreover, one can define piecewise polynomial spaces of degree $p\in\{3,4\}$, where the quadrilaterals have to be considered as macro-elements and subdivided further. This is explained in more detail in Section \ref{sec:Argyris_spline}.

We denote with $\mathbb{P}^{(p,p)}$ the space of bivariate polynomials of bi-degree $(p,p)$ and with $\mathbb{P}^{p}$ the space of polynomials of total degree $p$,  either uni- or bivariate, depending on context.

In the next subsections we introduce the local spaces and degrees of freedom corresponding to a single quadrilateral $\Q$. To do this, we need the following notation.

\begin{definition}[Pre-images of points and edges]
For every point $\vertex\in\mathbb{R}^2$, with $\vertex\in\overline{\Q}$, we define $\hat{\vertex}$ as the pre-image of ${\vertex}$ under $\param_{\Q}$, i.e., 
$\hat{\vertex} = \param_{\Q}^{-1}(\vertex)$. Analogously, we define $\hat{\edge} = \param_{\Q}^{-1}(\edge)$ for $\edge\in \Edge$ with $\edge\subset\overline{\Q}$. 
In Fig.~\ref{fig:patch} we have, e.g., $\hat{\vertex}_1 = (0,0)^T$.
\end{definition}

One set of degrees of freedom is the normal derivative at the edge midpoint, where we use the following notation. For every edge $\edge\in\Edge$ between vertices 
$\vertex_1$ and $\vertex_2$, let $\midp_\edge = \frac{1}{2}\vertex_1 + \frac{1}{2}\vertex_2$ be the edge midpoint. Moreover, let $\edgenorm$ be its unit normal vector and $\dn$ be the normal derivative of a function defined on $\Omega$ across the edge $\edge$. Here we assume that the direction of the normal is fixed for every edge of the mesh $\Mesh$.

\subsection{Local space and degrees of freedom, $p=5$}

Given a quadrilateral $\Q \in\Quad$ we define the local function space and the local degrees of freedom as follows. 
\begin{definition}[BS  quadrilateral for $p=5$]\label{def:local-space-and-dofs}
Given a quadrilateral $\Q$ with vertices $\vertex_1$, $\vertex_2$, $\vertex_3$ and $\vertex_4$ and edges $\edge_1$, $\edge_2$, $\edge_3$ and $\edge_4$ following~\cite{Ci02}, we define the BS  quadrilateral of degree $p=5$ as $(\Q,\Space_\Q^5,\Dual_\Q^5)$, with
\begin{equation}\label{eq:local_space_p5}
 \Space_\Q^5 = \left\{ \func:\overline{\Q}\rightarrow \mathbb{R}, \mbox{ with } (\func \circ \param_\Q) \in\mathbb{P}^{(5,5)}, (\dni \func \circ \param_\Q)|_{\hat{\edge}_i} 
 \in \mathbb{P}^4 ,\; 1 \leq i \leq 4 \right\}
\end{equation}
and 
\begin{equation}\label{eq:local_dofs_p5}
\begin{array}{l}
 \Dual_\Q^5 = \Dual_{0,\Q} \cup \Dual_{1,\Q}^5 \cup \Dual_{2,\Q}^5, \mbox{ with } \\
 \; \Dual_{0,\Q} = \left\{\func(\vertex_i), \dx\func(\vertex_i), \dy\func(\vertex_i), \dx\dx\func(\vertex_i),\dx\dy\func(\vertex_i), \dy\dy\func(\vertex_i),\; 1 \leq i 
 \leq 4 \right\}, \\
 \; \Dual_{1,\Q}^5 = \left\{\dni \varphi(\midp_{\edge_i}),\; 1 \leq i \leq 4 \right\}, \\
 \; \Dual_{2,\Q}^5 = \left\{ \func (\facep),\; \facep \in \Facep_\Q^5 \right\}.
\end{array}
\end{equation}

The set of face points is given as 
\begin{equation*}
 \Facep_{\Q}^5 = \left\{ \param_\Q \left(\eta_1,\eta_2\right),\; \eta_1,\eta_2\in\left\{\frac{2}{5},\frac{3}{5}\right\} \right\}.
\end{equation*}
\end{definition}
See Fig.~\ref{fig:dofs} for a visualization of the local degrees of freedom of the BS  quadrilateral.

\begin{figure}[htb]
\centering
\includegraphics[width=.3\textwidth]{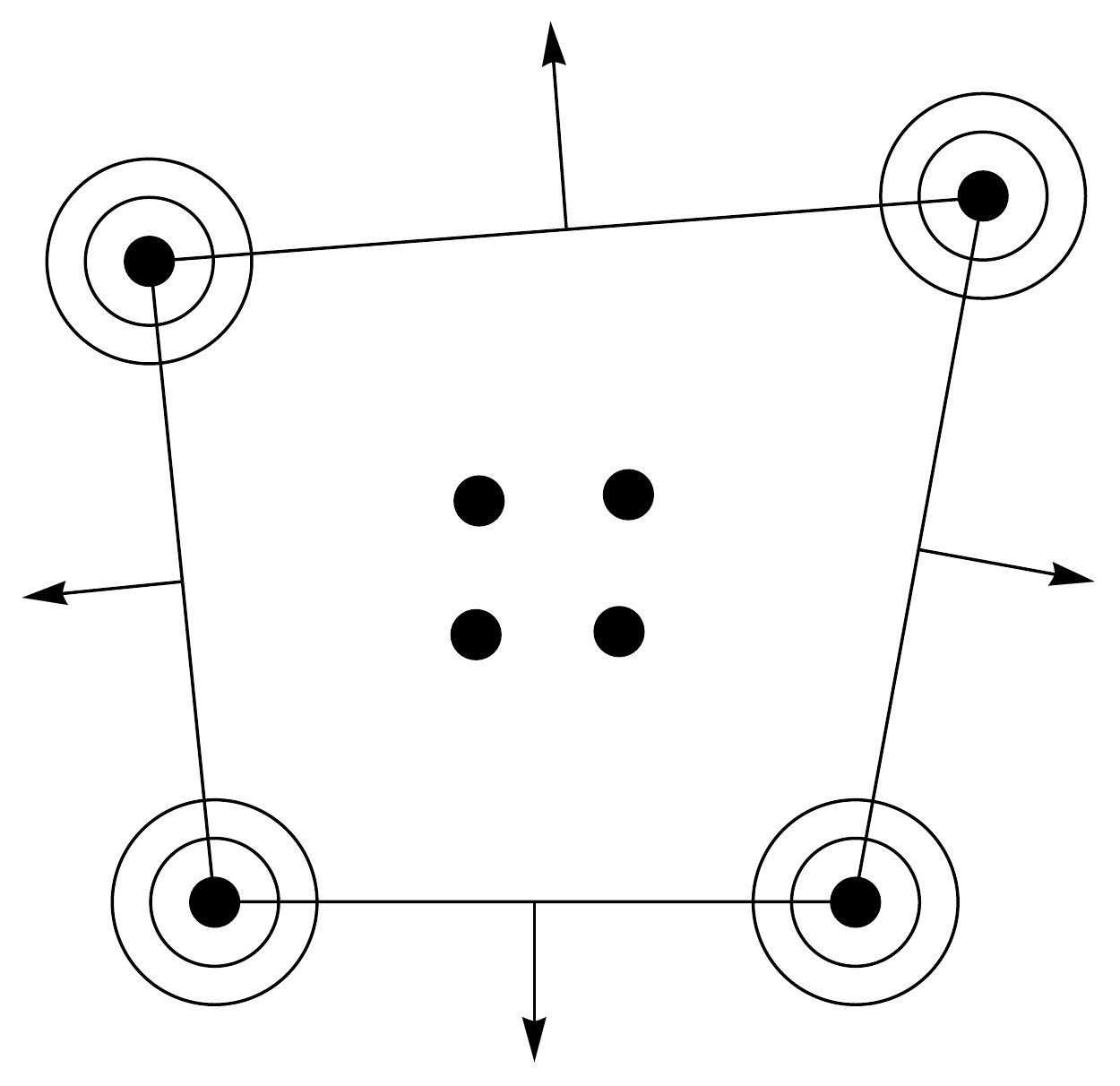}
\caption{The BS  quadrilateral for $p=5$, visualizing the degrees of freedom $\Dual_\Q^5$.}
\label{fig:dofs}
\end{figure}

The unisolvency of the degrees of freedom $\Dual_\Q^5$ for the space
$\Space_\Q^5$ follows from the basis construction in
Section~\ref{sec:polynomial-representation}.

We have already observed the similarity of this construction with the Argyris triangle. Let us recall the definition of the Argyris triangle for $p=5$ as given in~\cite{ArFrSc68,Ci02}.
\begin{definition}[Argyris triangle for $p=5$]
Given a triangle $T$ with vertices $\vertex_1$, $\vertex_2$ and $\vertex_3$ and edges $\edge_1$, $\edge_2$ and $\edge_3$ we define the Argyris triangle as $(T,\Space_T,\Dual_T)$, with
$\Space_T = \mathbb{P}^{5}$
and $\Dual_T = \Dual_{0,T} \cup \Dual_{1,T}$, with
\begin{equation*}
\begin{array}{l}
 \;\Dual_{0,T} = \left\{\func(\vertex_i), \dx\func(\vertex_i), \dy\func(\vertex_i), \dx\dx\func(\vertex_i),\dx\dy\func(\vertex_i), \dy\dy\func(\vertex_i),\; 1 \leq i 
 \leq 3 \right\}, \\
 \;\Dual_{1,T} = \left\{\dni \varphi(\midp_{\edge_i}),\; 1 \leq i \leq 3 \right\}.
\end{array}
\end{equation*}
\end{definition}
Hence, the degrees of freedom for the BS  quadrilateral ($p=5$) and Argyris triangle ($p=5$) are the same, except for the additional point evaluations at face points in the quadrilateral case. In addition, the traces as well as normal derivatives along edges are the same in both elements, i.e., for $p=5$ traces are quintic polynomials and normal derivatives are 
quartic polynomials. The degrees of freedom for the Argyris triangle are visualized in Figure~\ref{fig:dofsAB} (left).

\begin{figure}[htb]
\centering
\includegraphics[width=.3\textwidth]{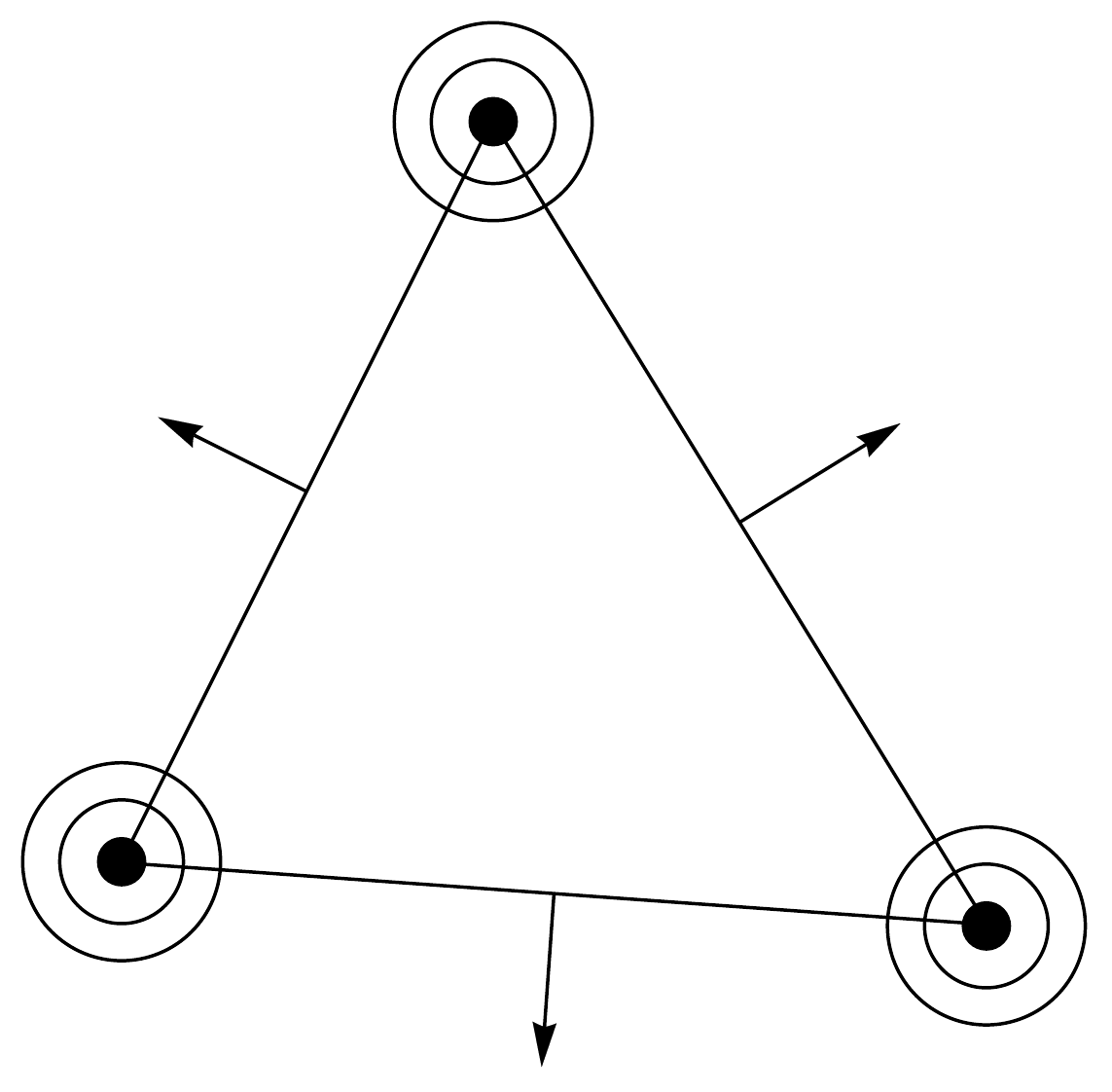}\hspace{.1\textwidth}
\includegraphics[width=.3\textwidth]{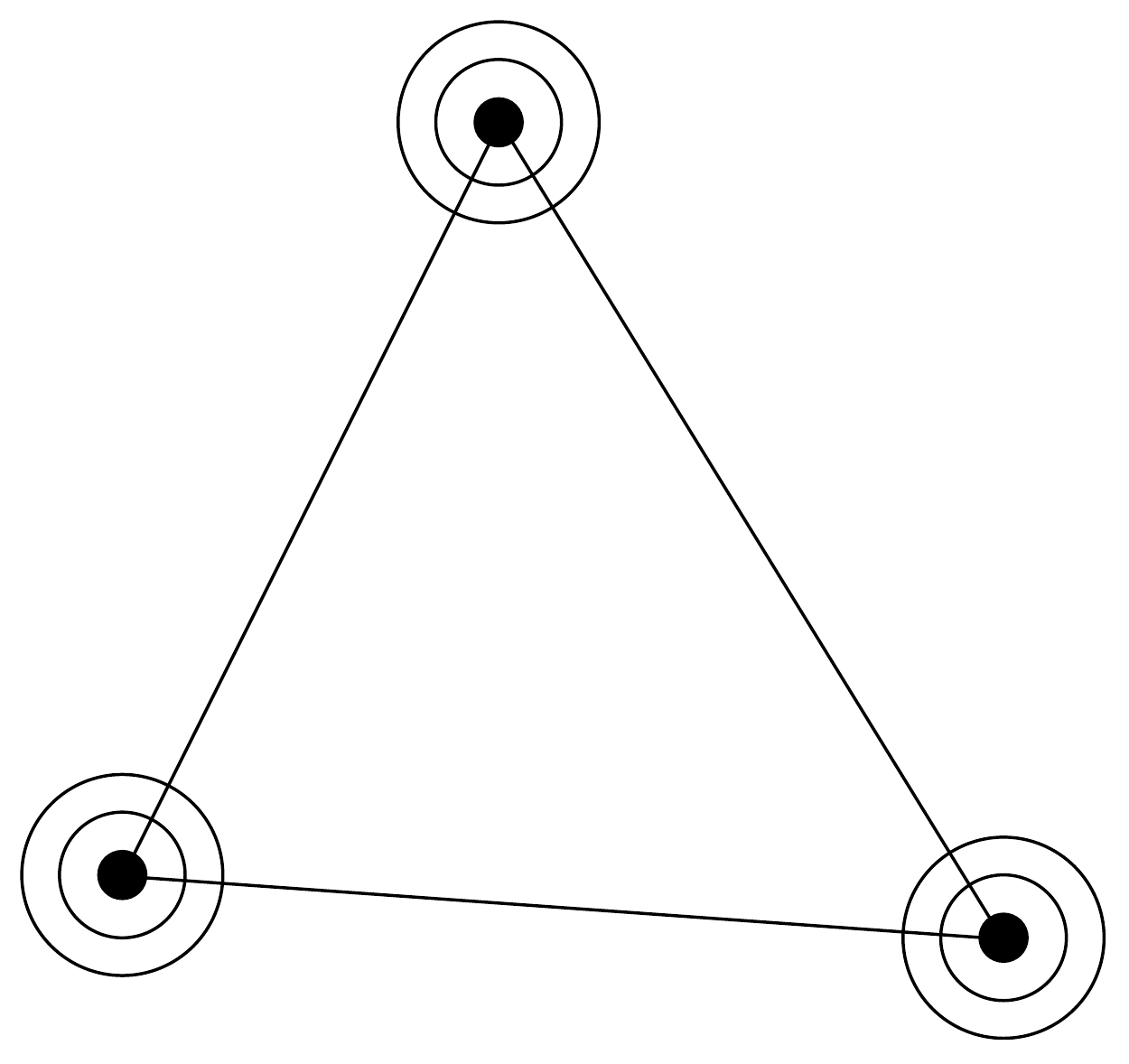}
\caption{The Argyris triangle (left) and Bell triangle (right), visualizing $\Dual_T^A$ and $\Dual_T^B$, respectively.}
\label{fig:dofsAB}
\end{figure}

In addition, the condition that the normal derivative along an edge is of degree $4$, is similar to the condition on the Bell triangular element \cite{Be69}, a quintic element, where normal derivatives are assumed to be polynomials of degree $3$, thus eliminating the normal derivative degrees of freedom and resulting in $18$ degrees of freedom per triangle.
\begin{definition}[Bell triangle for $p=5$]
Given a triangle $T$ with vertices $\vertex_1$, $\vertex_2$ and $\vertex_3$ and edges $\edge_1$, $\edge_2$ and $\edge_3$ we define the Bell triangle as $(T,\Space_T^B, \Dual_T^B)$, with
\begin{equation}
 \Space_T^B = \left\{ \func:\overline{T}\rightarrow \mathbb{R}, \mbox{ with } \func \in\mathbb{P}^{5},\; \dni \func |_{{\edge}_i} \in \mathbb{P}^3 ,\; 1 \leq i \leq 3 \right\}
\end{equation}
and $\Dual_T^B = \Dual_{0,T}$.
\end{definition}
The degrees of freedom for the Bell triangle are visualized in Figure~\ref{fig:dofsAB} (right). Both triangle elements possess variants of higher degree, see~\cite{ArFrSc68,Be69,Ci02,LaSc07}. For triangular elements, constructions of smooth spaces for lower degrees are usually based on special splits, such as the Clough-Tocher or Powell-Sabin $6$- or $12$-splits. Unlike the triangular case, in the quadrilateral case variants of lower degree are relatively straightforward and follow from the spline constructions developed in~\cite{KaSaTa19}.

\subsection{Local space and degrees of freedom,  $p\geq 6$}

Given a quadrilateral $\Q \in\Quad$ we define the local function space and the local degrees of freedom for $p\geq 6$ as follows. 
\begin{definition}[BS  quadrilateral~\cite{BrSu05}]\label{def:local-space-and-dofs-p6}
Given a quadrilateral $\Q$ with vertices $\vertex_1$, $\vertex_2$, $\vertex_3$ and $\vertex_4$ and edges $\edge_1$, $\edge_2$, $\edge_3$ and $\edge_4$ we define the BS  quadrilateral of degree $p\geq 6$ as $(\Q,\Space_\Q^p,\Dual_\Q^p)$, with
\begin{equation}
 \Space_\Q^p = \left\{ \func:\overline{\Q}\rightarrow \mathbb{R}, \mbox{ with } (\func \circ \param_\Q) \in\mathbb{P}^{(p,p)}, (\dni \func \circ \param_\Q)|_{\hat{\edge}_i} 
 \in \mathbb{P}^{p-1} ,\; 1 \leq i \leq 4 \right\}
\end{equation}
and 
\begin{equation}
\begin{array}{l}
 \Dual_\Q^p = \Dual_{0,\Q} \cup \Dual_{1,\Q}^p \cup \Dual_{2,\Q}^p, \mbox{ with } \\
 \; \Dual_{0,\Q} = \left\{\func(\vertex_i), \dx\func(\vertex_i), \dy\func(\vertex_i), \dx\dx\func(\vertex_i),\dx\dy\func(\vertex_i), \dy\dy\func(\vertex_i),\; 1 \leq i 
 \leq 4 \right\}, \\
 \; \Dual^p_{1,\Q} = \left\{\varphi(\param_{\edge_i}(\frac{j}{p})),\; \mbox{ for } 1 \leq i \leq 4,\; 3\leq j\leq p-3 \right\} \\ 
 \hspace{40pt}\cup \left\{\dni \varphi(\param_{\edge_i}(\frac{j}{p-1})),\; \mbox{ for } 1 \leq i \leq 4,\; 2\leq j\leq p-3 \right\}, \\
 \; \Dual_{2,\Q}^p = \left\{ \func (\facep),\; \facep \in \Facep_\Q^p \right\}.
\end{array}
\end{equation}
Here $\param_{\edge_i} = \param_\Q|_{\edge_i}$, and the set of face points is given as 
\begin{equation*}
 \Facep_{\Q}^p = \left\{ \param_\Q \left(\eta_1,\eta_2\right),\; \eta_1,\eta_2\in\left\{\frac{2}{p},\ldots,\frac{p-2}{p}\right\} \right\}.
\end{equation*}
\end{definition}
As one can easily see, Definition~\ref{def:local-space-and-dofs-p6}
covers also the case of
Definition~\ref{def:local-space-and-dofs}. Obviously, we have the
following. The degrees of freedom  $\Dual_\Q^p$ are unisolvent for the space
$\Space_\Q^p$. Indeed, one can show (see   \cite{BrSu05}) that  the
dimension of $\Space_\Q^p$ is  given by $ \dim(\mathbb{P}^{(p,p)})=(p+1)^2$ minus the number of constraints from $(\dni \func \circ \param_\Q)|_{\hat{\edge}_i} \in \mathbb{P}^{p-1}$, which are one per edge, that is, four.  Then  the
dimension of $\Space_\Q^p$ is  $(p+1)^2-4$ and  equals the cardinality
of $\Dual_\Q^p$.

\subsection{Local space and degrees of freedom, $p\in\{3,4\}$}\label{subsec:local_p34}

In the following we extend the construction on quadrilaterals to lower degrees $p=3$ and $p=4$ using a split into sub-elements, as in Fig.~\ref{fig:dofs34}. We assume that the 
parameter domain $\hatquad$ is split into sub-elements $\hat{q}\in s_k(\hatquad)$, with 
\begin{equation}
 s_k(\hatquad) = \left\{ \left[\frac{i}{k},\frac{i+1}{k}\right] \times \left[\frac{j}{k},\frac{j+1}{k}\right],\; 0\leq i\leq k-1,0\leq j\leq k-1\right\}.
\end{equation}
\begin{definition}[$C^1$ quadrilateral macro-element for $p\in\{3,4\}$]\label{def:local-space-and-dofs-34}
Given a quadrilateral $\Q$ with vertices $\vertex_1$, $\vertex_2$, $\vertex_3$ and $\vertex_4$ and edges $\edge_1$, $\edge_2$, $\edge_3$ and $\edge_4$ we define the $C^1$ quadrilateral macro-element of degree $p\in\{3,4\}$ as $(\Q,\Space_\Q^p,\Dual_\Q^p)$, with
\begin{equation}\label{eq:local_space_p34}
 \Space_\Q^p = \left\{ \func:\Q\rightarrow \mathbb{R}, \mbox{ with } 
 \begin{array}{ll} \func &\in C^{p-2}(\Q),\\
 (\func \circ \param_\Q)|_{\hat{q}} &\in \mathbb{P}^{(p,p)},\\
 (\func \circ \param_\Q)|_{\hat{\edge}_i} &\in C^{p-1}(\hat{\edge}_i), \\
 (\dni \func \circ \param_\Q)|_{\hat{\edge}_i\cap\hat{q}} &\in \mathbb{P}^{p-1}
\end{array}
\begin{array}{l} 
\mbox{ for } \hat{q}\in s_{6-p}(\hatquad), \\
\mbox{ for  }1\leq i\leq 4, 
\end{array} \right\}
\end{equation}
and 
\begin{equation}\label{eq:local_dofs_p34}
\begin{array}{l}
 \Dual_\Q^p = \Dual_{0,\Q} \cup \Dual^{p}_{1,\Q} \cup \Dual_{2,\Q}^p, \mbox{ with } \\
 \;\Dual_{0,\Q} = \left\{\func(\vertex_i), \dx\func(\vertex_i), \dy\func(\vertex_i), \dx\dx\func(\vertex_i),\dx\dy\func(\vertex_i), \dy\dy\func(\vertex_i),\; 1 \leq i \leq 4 
 \right\}, \\
 \; \Dual_{1,\Q}^p = \left\{\dni \func(\midp_{\edge_i}),\; 1 \leq i \leq 4 \right\}, \\
 \;\Dual_{2,\Q}^p = \left\{ \func (\facep),\; \facep \in \Facep_\Q^p \right\}.
\end{array}
\end{equation}

For $p=4$ the set of face points is given as 
\begin{equation*}
 \Facep_{\Q}^4 = \left\{ \param_\Q \left(\eta_1,\eta_2\right),\quad \eta_1,\eta_2\in\left\{\frac{1}{4},\frac{2}{4},\frac{3}{4}\right\} \right\},
\end{equation*}
for $p=3$ we have 
\begin{equation*}
 \Facep_{\Q}^3 = \left\{ \param_\Q \left(\eta_1,\eta_2\right),\quad \eta_1,\eta_2\in\left\{\frac{2}{9},\frac{4}{9},\frac{5}{9},\frac{7}{9}\right\} \right\}.
\end{equation*}
\end{definition}
As for $p=5$, the degrees of freedom $\Dual_\Q^p$ completely determine the functions from the space~$\Space_\Q^p$ and the dimension is given by $\dim(\Space_\Q^p)=|\Dual_\Q^p|= 28+(7-p)^2$. This follows as a special case of Lemma \ref{lemma:unisolvency-spline-elements}.

In Fig.~\ref{fig:dofs34} we visualize the polynomial sub-elements from~\eqref{eq:local_space_p34} and local degrees of freedom from~\eqref{eq:local_dofs_p34} for $p\in\{3,4\}$. 
\begin{figure}[htb]
\centering
\includegraphics[width=.3\textwidth]{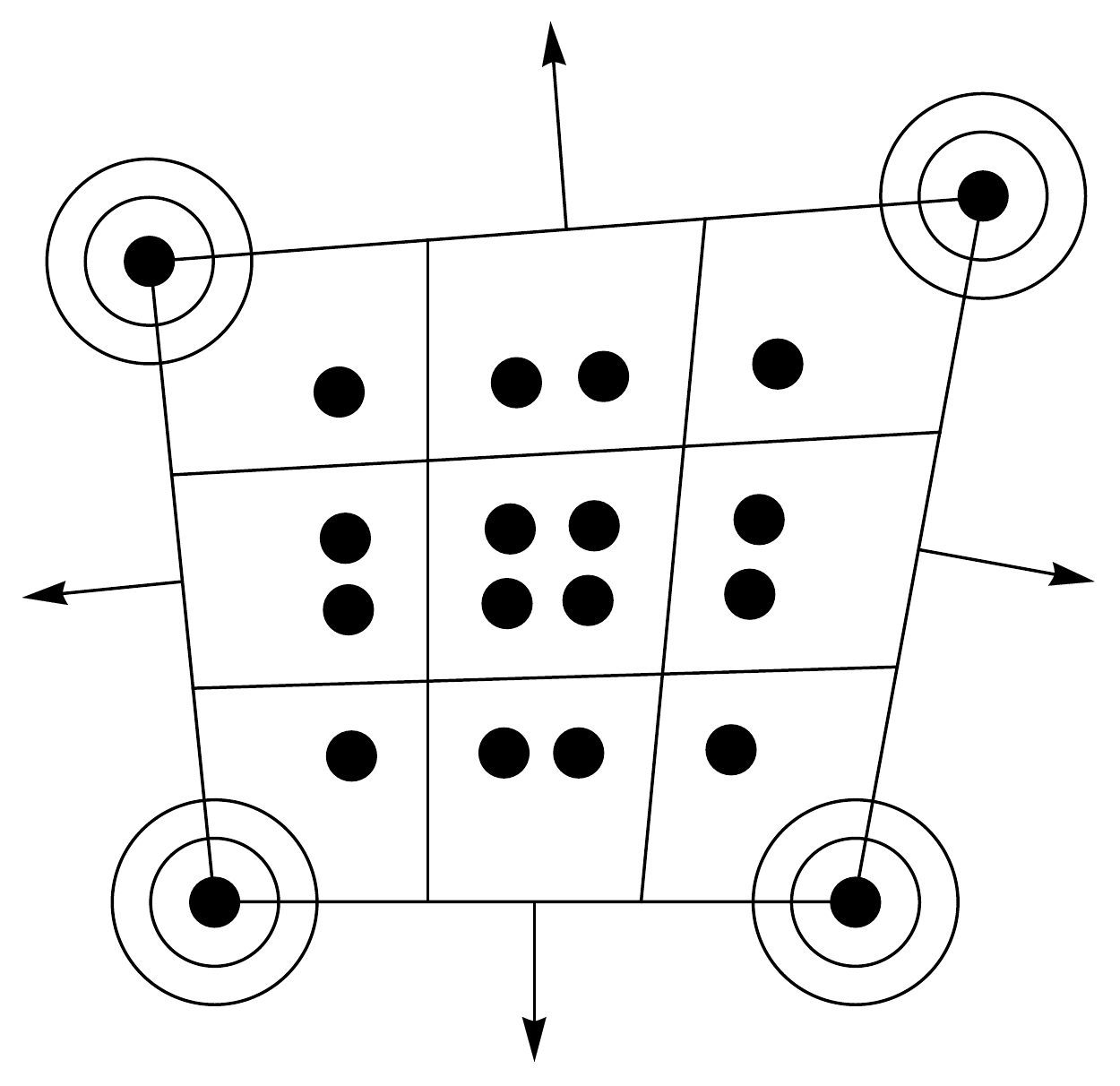}\qquad
\includegraphics[width=.3\textwidth]{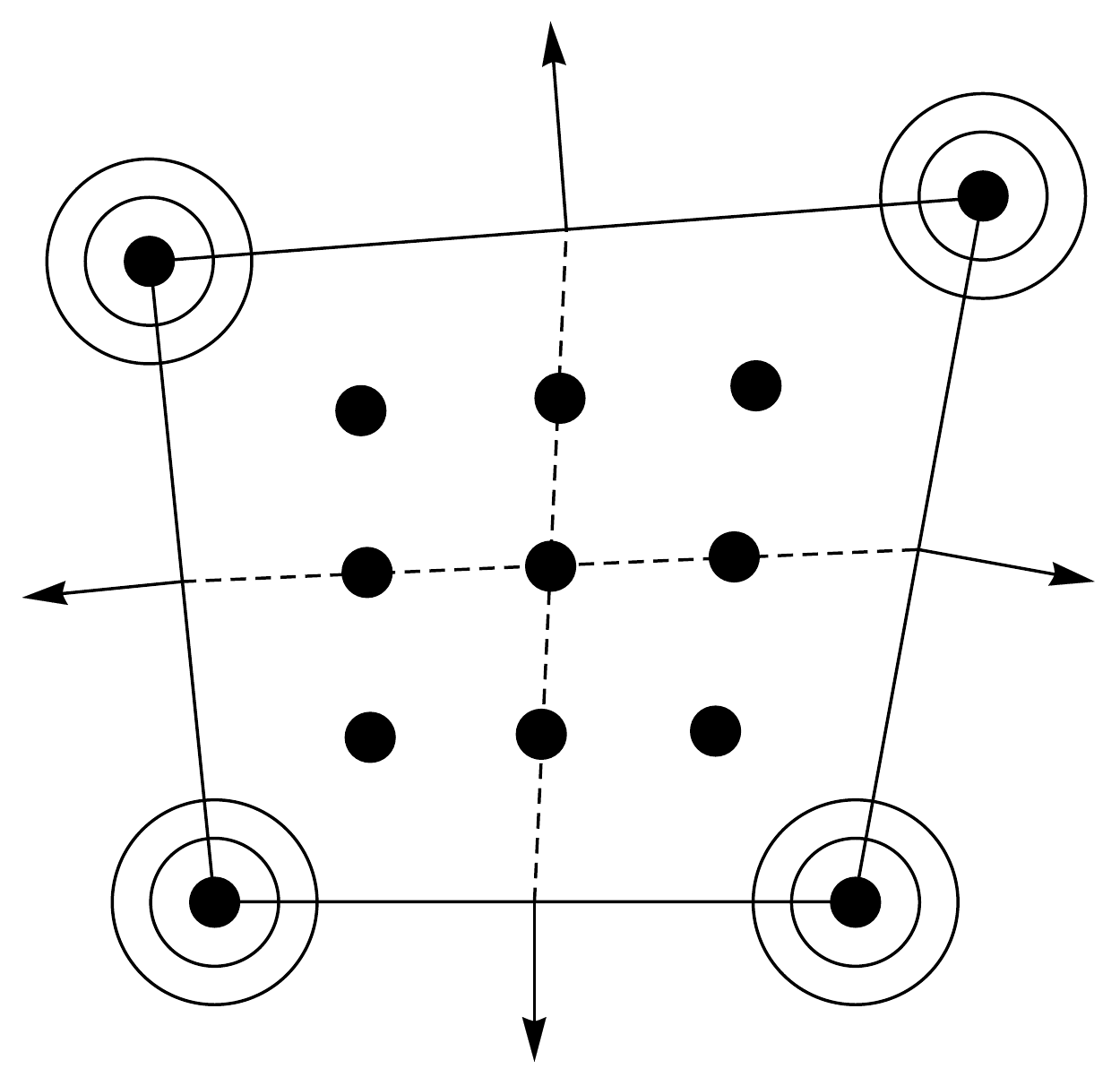}
\caption{The $C^1$ quadrilateral macro-elements for $p=3$ (left) and $p=4$ (right), visualizing $\Dual_\Q^3$ and $\Dual_\Q^4$, respectively. The solid inner lines represent lines 
of $C^1$ continuity, whereas the dashed lines are $C^2$.}
\label{fig:dofs34}
\end{figure}

\subsection{Global space and global degrees of freedom}

In this section we describe the global space and set of degrees of freedom from the local spaces and degrees of freedom defined above, with focus on the low-degree cases $p\in\{3,4,5\}$.
\begin{definition}[Global degrees of freedom]\label{def:global-dofs} Let $p\in\{3,4,5\}$. Given a quadrilateral mesh $\Mesh$ we have the degrees of freedom $\Dual^p$, given as 
\begin{itemize}
 \item $\func(\vertex)$, $\dx\func(\vertex)$, $\dy\func(\vertex)$, $\dx\dx\func(\vertex)$, $\dx\dy\func(\vertex)$ and $\dy\dy\func(\vertex)$ for all vertices $\vertex\in\Vertices$;
 \item $\dn\func(\midp_{\edge})$ for all edge midpoints $\midp_\edge$ with $\edge\in\Edge$; and
 \item $\func(\facep_{\Q})$ for all face points $\facep_{\Q}\in\Facep_{\Q}^p$ for all $\Q\in\Quad$.
\end{itemize}
\end{definition}

The global degrees of freedom in Definition~\ref{def:global-dofs} together with the finite element descriptions in Definitions~\ref{def:local-space-and-dofs} 
and~\ref{def:local-space-and-dofs-34} determine a global space $\globalspace^p(\Mesh)\subset C^1(\Omega)$.
\begin{lemma}[The $C^1$ quadrilateral space]\label{lemma:global-space}
Let $p\in\{3,4,5\}$ and let $\Mesh$ of $\Omega$ and let the space $\globalspace^p(\Mesh)$ be given by the degrees of freedom~$\Dual^p$ as in Definition~\ref{def:global-dofs}, with 
\begin{equation*}
 \globalspace^p(\Mesh) |_{\Q} = \Space^p_\Q \mbox{ for all }\Q\in\Quad,
\end{equation*}
where the local spaces $\Space_\Q^p$ are given as in Definition~\ref{def:local-space-and-dofs} or \ref{def:local-space-and-dofs-34}, respectively.
Then the global space satisfies $\globalspace^p(\Mesh)\subset C^1(\Omega)$ and we have 
\begin{equation*}
 \dim (\globalspace^p(\Mesh)) = |\Dual^p| = (7-p)^2 \cdot \left|\Quad\right| + 1 \cdot \left|\Edge\right| + 6 \cdot \left|\Vertices\right|.
\end{equation*}
\end{lemma}
\begin{proof}
Note that the piecewise polynomial space $\Space_\Q^p$ from Definition~\ref{def:local-space-and-dofs-34} covers also the polynomial case $p=5$ for $k=1$, where $s_{6-p}(\hatquad) = s_{1}(\hatquad) = \{\hatquad\}$. 
To prove $\globalspace^p(\Mesh)\subset C^1(\Omega)$ we consider all
$C^1$-data along a single edge $\edge$ between two elements $\Q$ and
$\Q'$. Let $\func\in \globalspace^p(\Mesh)$ and let $\hat{q}\in
s_{6-p}(\hatquad)$. We have, since $\func |_{\Q} \in \Space^p_\Q$,
that $(\func \circ \param_\Q) |_{\hat{\edge} \cap \hat{q}} \in
\mathbb{P}^p$ and $(\func \circ \param_\Q) |_{\hat{\edge}} \in
C^{p-1}$, and  $(\dn \func \circ
\param_\Q)|_{\hat{\edge}\cap\hat{q}} \in \mathbb{P}^{p-1}$. Consequently, since $\param_\Q |_{\hat{\edge}}$ is a linear function, we have $\func |_{\edge\cap q} \in \mathbb{P}^p$, $\func |_{\edge} \in C^{p-1}$ and $\dn \func |_{{\edge}\cap{q}} \in \mathbb{P}^{p-1}$, where $\edge\cap q = \edge\cap \param_{\Q}(\hat{q}) = \edge\cap \param_{\Q'}(\hat{q}')$. Hence, $\func|_{\edge}$ is a piecewise polynomial of degree $p$, with dimension~$6$. Value, 
first and second derivative (in direction of the edge) of $\func$ at the two vertices of $\edge$ are determined by the $C^2$-data. The function $\func|_{\edge}$ is thus completely determined by the $C^2$-data. This is independent of the element $\Q$, $\Q'$ under consideration. Hence, we have $\func \in C^0 (\Omega)$. Moreover, by definition, the function $\dn \func |_{\edge}$ is a piecewise polynomial of degree $p-1$, with dimension $5$, independent of $\Q$, $\Q'$. Of those $5$ degrees of freedom, the $C^2$-data at the vertices determine two each, whereas one is determined by $\dn\func(\midp_{\edge})$. Hence, $\func|_{\edge}$ and $\dn \func |_{\edge}$ are completely determined by the global degrees of freedom and $\func \in C^1 (\Omega)$. What remains to be shown is that $\dim (\globalspace^p(\Mesh)) = |\Dual^p|$.  Its proof follows directly from a simple counting argument.
\end{proof}
We have presented Lemma~\ref{lemma:global-space} and its proof purely in terms of a finite element setting, considering the local spaces and global degrees of freedom. See~\cite[Section 4]{KaSaTa19} for a more general statement on spline patches.  Note that the space $\globalspace^p(\Mesh)$ is $C^2$ at all vertices by construction.

\begin{remark}
Since both the degrees of freedom $\Dual_\Q^p$ as well as the definition of the local space $\Space_\Q^p$ depend on derivatives in normal direction, the proposed BS quadrilaterals (including the macro-element variants) are not affine invariant, as the Argyris triangle, which possesses an affine invariant space, but no affine invariant degrees of freedom.
\end{remark}

\section{Approximation properties}

\label{sec:gs-appr-prop}

In this section we prove local and global approximation
estimates, where the error is measured only in the norms of interest $\| \cdot \|_{L^{\infty}}$,
$\|\cdot \|_{L^{2}}$,  and  $\|\cdot \|_{H^{\ell}}$, for simplicity.
For the notation concerning Sobolev spaces, we follow \cite{BrSc07}. 

Given a convex quadrilateral $\Q\in \Quad$, the main ingredient to prove the local approximation estimate is the projector 
${ \Pi_{\Space^p_{\Q}}}: C^2(\overline{\Q}) \rightarrow \Space_\Q^p$
defined by 
\begin{equation}\label{eq:projector_local}
  \Pi_{\Space^p_{\Q}}(\varphi ) = \sum_{\lambda_{\Q} \in \Dual^p_{\Q}}
  \lambda_{\Q}(\varphi ) \basisf_{\lambda_{\Q}},
\end{equation}
where $\basisf_{\lambda_{\Q}} \in \Space^p_{\Q}$ are basis functions
that satisfy $\lambda_{\Q}(\basisf_{\lambda_{\Q}}) = 1$ and
$\lambda_{\Q}'(\basisf_{\lambda_{\Q}}) = 0$ for all $\lambda_{\Q} \neq
\lambda_{\Q}' \in \Dual_{\Q}^p$. The existence of such a basis is
a consequence of the unisolvence of the set of degrees of freedom $\Dual_{\Q}^p$.  A key property for the approximation result is the basis stability stated below.

\begin{lemma}\label{lemma:basis-stability}
 Let $\Q\in \Quad$ be a convex quadrilateral. There exists a
 constant $C>0$, dependent on  $h_{\Q} $,   $\shaperegularityparameter_{\Q}$,  and $p$   such that for all  $\lambda_{\Q}  \in \Dual_{\Q}^p$
 \begin{equation*}
   \| \basisf_{\lambda_{\Q}} \| \leq C
 \end{equation*}
 where  $\| \cdot \| $ is any of the norms of interest.
 \end{lemma}
\begin{proof}
 Each basis function $\beta_{\lambda_{\Q}}$ can
 be obtained by imposing the conditions to
 belong to the space
 \begin{equation}
   \label{eq:basis-constraint-1}
   \basisf_{\lambda_{\Q}}  \in \Space^p_{\Q}
 \end{equation}
 and to be in  duality  to the degrees of freedom
\begin{equation}
   \label{eq:basis-constraint-2}
 \lambda_{\Q}(\basisf_{\lambda_{\Q}}) = 1 \text{ and }
\lambda_{\Q}'(\basisf_{\lambda_{\Q}}) = 0, \quad \forall \lambda_{\Q} \neq
\lambda_{\Q}' \in \Dual_{\Q}^p.
 \end{equation}
In parametric coordinates, it means that $\basisf_{\lambda_{\Q}} \circ  \param_{\Q}$ defined on $ \hatquad $ is a polynomial (for $p\geq 5$, it belongs to $\mathbb{P}^{(p,p)}$) or piecewise polynomial (for $p\in\{3,4\}$, its restriction to each subelement $\hat{q}\in s_{6-p}(\hatquad)$ belongs to $\mathbb{P}^{(p,p)}$)  that fulfills \eqref{eq:basis-constraint-1}--\eqref{eq:basis-constraint-2}. These conditions above involve the first and second derivatives of the inverse parametrization $ \param_{\Q}^{-1}$, that are well defined and  bounded on $\overline \Q$ thanks to \eqref{eq:det_jacobian_F}. Recalling the expression \eqref{eq:gradient-F} of $  \nabla \param_{\Q}$, the first and second derivatives of $ \param_{\Q}^{-1}$ are rational polynomials in  $x_1$ and $x_2$ and depend continuously on the parameters $\f{t}^{(1)}, \ldots,
\f{t}^{(4)}$. Therefore $  \|  \basisf_{\lambda_{\Q}}   \| $ only depends on  $\f{t}^{(1)}, \ldots, \f{t}^{(4)}$, the dependence is continuous and the parameters belong to the compact set
 \begin{equation*}
  \left \{ \f{t}^{(1)}, \ldots,
 \f{t}^{(4)} :  2 \shaperegularityparameter_{\Q} \leq \angle(\f t^{(i)},\f
    t^{(i-1)})) \leq \pi -  2 \shaperegularityparameter_{\Q} \text{
      and } \| \f t^{(i)} \| \leq h_\Q\right \},
\end{equation*}
thanks to Proposition~\ref{prop:shape-regularity}. Continuity and
compactness give the existence of a maximum of  
$\|\basisf_{\lambda_{\Q}}\| $ which only depends on $h_\Q$,
$\shaperegularityparameter_{\Q} $ and on $p$.
\end{proof}
Lemma \ref{lemma:basis-stability} yields the local stability of the projector.
\begin{lemma}\label{lemma:local-projector}
Let $\Q\in \Quad$ be any convex quadrilateral. There exists a
 constant $C>0$, dependent on  $h_{\Q} $,
 $\shaperegularityparameter_{\Q}$,  and $p$   such that  for all $\psi
 \in C^2(\overline{\Q})$, 
 \begin{equation*}
   \| \Pi_{\Space^p_{\Q}} \psi \| \leq C \|\psi \|_{C^2(\overline{\Q})},
 \end{equation*}
 where  $\| \cdot \| $ is one of the norms of interest.
\end{lemma}
\begin{proof}
Thanks to Lemma \ref{lemma:basis-stability}, we have that 
\begin{equation*}
  \| \Pi_{\Space^p_{\Q}} \psi \| \leq C \max_{\lambda\in
    \Dual_{\Q}^p}\left|\lambda(\psi)\right| ,
\end{equation*}
and then we use the obvious continuity $ \left|\lambda(\psi)\right|
\leq \|\psi
  \|_{C^2(\overline{\Q})}.$ 
\end{proof}
The next two Lemmata, from \cite{BrSc07}, concern  standard Sobolev
inequalities and standard  polynomial approximation over $\Q\in
\Quad$.
\begin{lemma}[{\cite[Lemma 4.3.4]{BrSc07}}]\label{lemma:sobolev-inequality}
Let $\Q\in \Quad$ be any convex quadrilateral. There exists a constant $C_{SI}>0$, dependent on  $h_{\Q} $ and $\shaperegularityparameter_{\Q}$  such that  for all  $\psi \in H^4({\Q})$  we have $\psi \in C^2({\Q})$ and  
\begin{equation*}
  \| \psi \|_{C^2(\overline{\Q})} \leq C_{SI} \| \psi \|_{H^4(\Q)} .
\end{equation*}
\end{lemma}
\begin{lemma}[{\cite[Lemma 4.3.8]{BrSc07}}]\label{lemma:bramble-hilbert}
Let $\Q\in \Quad$ be any convex quadrilateral and  
$B$ a maximal ball   inscribed in $\Q$. Let $m\leq p+1$.
There exists a  constant $C_{BH}>0$, dependent on  $h_{\Q} $,
 $\shaperegularityparameter_{\Q}$,  and $p$   such that  for all
 $\func\in H^m(\Q)$   \begin{equation*} \| \func - \Pi_{\mathbb{P}^p}\func \|_{H^m(\Q)} \leq C_{BH} | \func |_{H^m(\Q)},
\end{equation*}
where $\Pi_{\mathbb{P}^p} \func$ is the averaged Taylor polynomial of
degree $p$ of $\func$ over $B$.
\end{lemma}
The last property we need is that the  BS  quadrilateral element space contains the polynomials of total degree~$p$.
\begin{lemma}
Let $\Q\in \Quad$ be any convex quadrilateral, then  $\mathbb{P}^{p} \subset \Space^p_{\Q}$.
\end{lemma}
\begin{proof}
Let $\psi \in\mathbb{P}^{p}$, we need to show that  $\psi \circ \param_\Q \in
\mathbb{P}^{(p,p)}$ and $(\dni \psi \circ \param_\Q)|_{\hat{\edge}_i}
\in \mathbb{P}^{p-1}$ for all $\edge_i$, according
to~\eqref{eq:local_space_p5} and~\eqref{eq:local_space_p34}. Note that
we do not need to consider the sub-elements separately, as $\psi$ is a
global polynomial. The composition of a polynomial of total degree $p$
with a bilinear function always results in a polynomial of bi-degree
$(p,p)$, hence we have $\psi \circ \param_\Q \in
\mathbb{P}^{(p,p)}$. Moreover, the directional derivative $\dni \psi$
is a polynomial of total degree $p-1$, restricted to an edge yields a
univariate polynomial of degree $p-1$, which gives $(\dni \psi \circ
\param_\Q)|_{\hat{\edge}_i} \in \mathbb{P}^{p-1}$ since $
\param_\Q|_{\hat{\edge}_i} $ is a linear parametrization.  
\end{proof}
We can now state and prove the local approximation estimate.
\begin{theorem}\label{thm:local-estimate}
Let $\Q\in \Quad$ be a convex quadrilateral. There exists a constant $C>0$, dependent on  $\shaperegularityparameter_\Q$ and on $p$, such that for $0\leq \ell\leq 2$, $4 \leq m \leq p+1$ and for all $\func \in H^m(\Q)$ we have
\begin{equation*}
  \left| \func - \Pi_{P^p_{\Q}}\func \right|_{H^\ell(\Q)} \leq C \,  {h_\Q}^{m-\ell} \left| \func \right|_{H^m(\Q)};
\end{equation*}
moreover, for $3\leq m\leq p+1$, $3\leq p$ and for all $\func \in W^m_\infty(\Q)$,
\begin{equation*}
  \left\| \func - \Pi_{P^p_{\Q}}\func \right\|_{L^\infty(\Q)} \leq C \,  {h_\Q}^{m} \left| \func \right|_{W^m_\infty(\Q)}.
\end{equation*}
\end{theorem}
\begin{proof}
 The proof follows the proof of~\cite[Theorem 4.4.4]{BrSc07}. We can
 assume $h_\Q = 1$, since the general case and the role of $h_\Q $ in
 the estimates follow by an homogeneity argument. We have
 \begin{equation*}
 \begin{array}{ll}
  \| \func - \Pi_{P^p_{\Q}}\func \|_{H^\ell(\Q)} & \leq 
  \| \func - \Pi_{\mathbb{P}^p}\func \|_{H^\ell(\Q)} + 
  \| \Pi_{\mathbb{P}^p}\func - \Pi_{P^p_{\Q}}\func \|_{H^\ell(\Q)} \\
  & = 
  \| \func - \Pi_{\mathbb{P}^p}\func \|_{H^\ell(\Q)} + 
  \| \Pi_{P^p_{\Q}}(\Pi_{\mathbb{P}^p}\func - \func) \|_{H^\ell(\Q)}
 \end{array}
 \end{equation*}
 Applying the bound from Lemma~\ref{lemma:local-projector}, Lemma
 \ref{lemma:sobolev-inequality} and \ref{lemma:bramble-hilbert},
 we obtain 
 \begin{equation*}
 \begin{array}{ll}
  \| \func - \Pi_{P^p_{\Q}}\func \|_{H^\ell(\Q)} & \leq 
  \| \func - \Pi_{\mathbb{P}^p}\func \|_{H^\ell(\Q)} + 
  C \| \func - \Pi_{\mathbb{P}^p}\func \|_{C^2(\overline{\Q})} \\
  & \leq 
  (1 + C  C_{SI} )\| \func - \Pi_{\mathbb{P}^p}\func \|_{H^m(\Q)} \\
  & \leq 
  (1 +C C_{SI} ) C_{BH} | \func |_{H^m(\Q)},
 \end{array}
 \end{equation*}
 The $L^\infty$-estimate follows the same 
 idea as the $H^\ell$-estimates, where a bound of the form
 \begin{equation*}
  \| \Pi_{\globalspace(\Mesh)}(\func) \|_{L^\infty(\Q)} \leq \sigma(\shaperegularityparameter,p) \|\func \|_{C^2(\overline{\Q})}
 \end{equation*}
 is needed together with estimates similar to
 Lemma~\ref{lemma:sobolev-inequality}
 and~\ref{lemma:bramble-hilbert}. Note that in case of the $L^\infty$
 estimate we only need $m\geq 3$, see again~\cite[Theorem 4.4.4]{BrSc07}.  
\end{proof}

From this local error estimate, a global estimate follows
straightforwardly. Let 
$\Pi_{\globalspace^p(\Mesh)}: C^2(\overline{\Omega}) \rightarrow
\globalspace^p(\Mesh)$ be the global projector defined as  
\begin{equation}\label{eq:projector_global}
 \Pi_{\globalspace^p(\Mesh)}\varphi = \sum_{\lambda \in \Dual^p} \lambda(\varphi) \basisf_\lambda,
\end{equation}
where each $\basisf_\lambda \in \globalspace^p(\Mesh)$  satisfies $\lambda(\basisf_\lambda) = 1$ and $\lambda'(\basisf_\lambda) = 0$ for all $\lambda \neq \lambda' \in \Dual^p$. By definition of the local and global spaces and degrees of freedom, the  global projector and the local projector fulfill, for any $\Q\in \Quad$, 
\begin{equation}
 \left(\Pi_{\globalspace^p(\Mesh)}\varphi\right)|_{\Q} = \Pi_{\Space^p_{\Q}}(\varphi|_{\Q}) = \sum_{\lambda_{\Q} \in \Dual^p_{\Q}} \lambda_{\Q}(\varphi|_{\Q}) \basisf_{\lambda_{\Q}}.
\end{equation}
For a given local functional $\lambda_{\Q}(\cdot) = \lambda(\cdot|_\Q)$ we have $\basisf_{\lambda_\Q} = \basisf_\lambda |_\Q$. Hence, the support of 
$\basisf_\lambda$ is given by all elements on which $\lambda$ is defined, i.e., one element for all face point evaluations, two neighboring elements for all edge midpoint evaluations and, in case of vertex degrees of freedom, all elements around the vertex.
\begin{corollary}
 Let $\Mesh$ be a quadrilateral mesh of $\Omega$, that fulfills
 the requirements of Section \ref{sec:mesh}, with $h=\max_{\Q\in
   \Quad}(h_\Q)$ and $\shaperegularityparameter$ from \eqref{eq:mesh-is-shape-regular}. Let $0\leq \ell\leq 2$ and $4 \leq m \leq p+1$. 
 There exists a constant $C>0$, depending on  $\shaperegularityparameter$ and $p$, such that we have for all  $\func \in H^m(\Omega)$
 \begin{equation*}
  \left| \func - \Pi_{\globalspace^p(\Mesh)}\func \right|_{H^\ell(\Omega)} \leq C \,  {h}^{m-\ell} \left| \func \right|_{H^m(\Omega)},
 \end{equation*}
 as well as for $3\leq m\leq p+1$, $3\leq p$ and for all $\func \in W^m_\infty(\Omega)$
  \begin{equation*}
  \left\| \func - \Pi_{\globalspace^p(\Mesh)}\func \right\|_{L^\infty(\Omega)} \leq C \,  {h}^{m} \left| \func \right|_{W^m_\infty(\Omega)}.
 \end{equation*}
\end{corollary}

\section{Basis construction, $p=5$}\label{sec:polynomial-representation}

In the following we describe how to compute the basis functions corresponding to one quadrilateral $\Q$ in the mesh. We define for every vertex six basis functions to interpolate the $C^2$ data, for every edge we define one basis function to interpolate the normal derivative at the edge midpoint. The remainder basis functions inside the element (with vanishing traces and derivatives on the element boundary) are selected to be standard Bernstein polynomials (for $p=5$) or standard B-splines (for $p\in\{3,4\}$). See~\cite{PrBoPa02,Sc07} for basics on B-splines.

To simplify the construction, we build a basis with respect to a slightly modified dual basis. Instead of point evaluations at the interior, we use integral-based functionals that are dual to the Bernstein polynomials (or B-splines).

Before we go into the details, we discuss the Bernstein-B\'ezier representation.
Let $\hat{b}_j$ be the Bernstein polynomials of degree $5$, i.e., for $0\leq j\leq 5$ and $\xi \in[0,1]$,
\begin{equation*}
 \hat{b}_j(\xi) = \binom{5}{j} \xi^j (1-\xi)^{5-j}
\end{equation*}
and let $\hat{\mu}_i$ be the corresponding dual functionals, as in~\cite{Ju98}, i.e., $\hat{\mu}_i(\hat{b}_j)=\delta_i^j$.
Let moreover
\begin{equation*}
\f B= 
\left(
\begin{array}{ccc} 
\hat{b}_{0,5}(\xi_1,\xi_2) & & \hat{b}_{5,5}(\xi_1,\xi_2) \\
\vdots & \ddots &\vdots \\
\hat{b}_{0,0}(\xi_1,\xi_2) & \cdots & \hat{b}_{5,0}(\xi_1,\xi_2)
\end{array}
\right)
=
\left(
\begin{array}{c} 
\hat{b}_5(\xi_2) \\
\vdots \\
\hat{b}_0(\xi_2)
\end{array}
\right)
\left(
\begin{array}{ccc} 
\hat{b}_0(\xi_1) &
\ldots &
\hat{b}_5(\xi_1)
\end{array}
\right)
\end{equation*}
be the matrix of tensor-product Bernstein basis functions spanning $\mathbb{P}^{(5,5)}$.

For each basis function $\basisf\in P^5_{\Q}$, the pull-back $\basisfhat = \basisf \circ \param_\Q$ possesses a biquintic tensor-product
Bernstein-B\'{e}zier representation, having the coefficients~$d_{j_1,j_2}\in \RR$,
\begin{equation*}
 \basisfhat(\xi_1,\xi_2) = \basisf \circ \param_\Q(\xi_1,\xi_2) = \sum_{j_1=0}^{5} \sum_{j_2=0}^{5} d_{j_1,j_2} \, \hat{b}_{j_1,j_2}(\xi_1,\xi_2).
\end{equation*}
By means of a table of the form
\begin{equation*}
\f D{[\basisf]} = 
\begin{array}{|c|c|c|c|}
\hline
 d_{0,5} & d_{1,5} & \cdots & d_{5,5}\\ \hline
 \vdots & \vdots &  & \vdots \\ \hline
 d_{0,1} & d_{1,1} & \cdots & d_{5,1}\\ \hline
 d_{0,0} & d_{1,0} & \cdots & d_{5,0} \\ \hline
\end{array}
\end{equation*}
we can represent the basis function as $\basisfhat = \f B : \f D[{\basisf}]$, the Frobenius product of 
the matrix of basis functions with the coefficient matrix. Given the basis $b_{i_1,i_2} = \hat{b}_{i_1,i_2}\circ {\param_\Q}^{-1}$ we can define 
a dual basis $\mu_{j_1,j_2}$ as $ \mu_{j_1,j_2}(\func) = \hat{\mu}_{j_1}\otimes\hat{\mu}_{j_2}(\func\circ\param_\Q)$, satisfying $ \mu_{j_1,j_2}(b_{i_1,i_2}) = 
\delta_{i_1}^{j_1}\delta_{i_2}^{j_2}$.

We now turn on defining the basis functions for $\Space^5_{\Q}$ and dual
functionals~$\Dual^5_{\Q}$.  On each quadrilateral~$\Q$, we define
$24$ vertex basis functions (six for each vertex)
\begin{equation*}
 \mathrm{B}^5_{0,\Q} = \{ \basisfn{0}{k,i},\quad\mbox{ for } k=1,\ldots,4\mbox{ and }i=0,\ldots,5 \},
\end{equation*}
determined by $\Dual_{0,\Q}$, four edge basis functions (one for each edge) 
\begin{equation*}
 \mathrm{B}^5_{1,\Q} = \{ \basisfn{1}{i},\quad\mbox{ for } i=1,\ldots,4 \},
\end{equation*}
determined by $\Dual^5_{1,\Q}$, and four patch-interior basis functions 
\begin{equation*}
 \mathrm{B}^5_{2,\Q} = \{ \basisfn{2}{i},\quad\mbox{ for } i=1,\ldots,4 \}.
\end{equation*}
To simplify the construction, we replace the point evaluation functionals $\Dual^5_{2,\Q}$ by the dual functionals of mapped tensor-product Bernstein polynomials
\begin{equation*}
 \mathrm{M}^5_{2,\Q} = \{ \mu_{j_1,j_2}(\func) = \hat{\mu}_{j_1}\otimes\hat{\mu}_{j_2}(\func\circ\param_\Q): j_1,j_2 \in \{2,3\}\}. 
\end{equation*}
We define the basis 
\begin{equation*}
 \mathrm{B}^5_{\Q} = \mathrm{B}^5_{0,\Q} \cup \mathrm{B}^5_{1,\Q} \cup \mathrm{B}^5_{2,\Q}
\end{equation*}
in such a way that it is dual to 
\begin{equation*}
 \mathrm{M}^5_\Q = \Dual_{0,\Q} \cup \Dual_{1,\Q} \cup \mathrm{M}^5_{2,\Q}.
\end{equation*}

\subsection{Patch interior basis functions} 

It is clear that we have, by definition, 
\begin{equation*}
 \mathrm{B}^5_{2,\Q} = \{ \basisfn{2}{1},\basisfn{2}{2},\basisfn{2}{3},\basisfn{2}{4} \} = \{ b_{2,2},b_{2,3},b_{3,2},b_{3,3} \}.
\end{equation*}
In terms of their B\'ezier coefficients we have e.g.:
\begin{small}
\begin{equation*}
\f D [b_{2,2}] = 
\begin{array}{|c|c|c|c|c|c|} 
\hline
 0 & 0 & 0 & 0 & 0 & 0\\ \hline
 0 & 0 & 0 & 0 & 0 & 0\\ \hline
 0 & 0 & 0 & 0 & 0 & 0\\ \hline
 0 & 0 & 1 & 0 & 0 & 0\\ \hline
 0 & 0 & 0 & 0 & 0 & 0\\ \hline
 0 & 0 & 0 & 0 & 0 & 0 \\ \hline
\end{array}
\end{equation*}
\end{small}
We trivially have $\mbox{span}(\mathrm{B}^5_{2,\Q}) = \ker( \Dual_{0,\Q} \cup \Dual_{1,\Q})$.

\subsection{Edge basis functions}

We recall the notation introduced in Section \ref{sec:mesh}: Let 
\begin{equation*}
 \f{t}^{(k)} = (t_1^{(k)},t_2^{(k)})^T = \vertex_{k+1} - \vertex_k
\end{equation*}
be the vector corresponding to the edge $\edge_k$ and let $a^{(k)}=\det (\f{t}^{(k-1)}, \f{t}^{(k)} )$. Then the edge basis function $\basisfn{1}{1}$, corresponding to edge $\edge_1$, is given by
\begin{small}
\begin{equation*}
\f D [\basisfn{1}{1}] = \frac{8}{25 \|\f t^{(1)} \|} \;
\begin{array}{|c|c|c|c|c|c|} 
\hline
 0 & 0 & 0 & 0 & 0 & 0\\ \hline
 0 & 0 & 0 & 0 & 0 & 0\\ \hline
 0 & 0 & 0 & 0 & 0 & 0\\ \hline
 0 & 0 & 0 & 0 & 0 & 0\\ \hline
 0 & 0 & a^{(1)} & a^{(2)} & 0 & 0\\ \hline
 0 & 0 & 0 & 0 & 0 & 0 \\ \hline
\end{array}
\end{equation*}
\end{small}
and analogously for $\basisfn{1}{2}$, $\basisfn{1}{3}$ and $\basisfn{1}{4}$. We have $\basisfn{1}{j} \in \ker( \Dual^0_{\Q} \cup \mathrm{M}^2_\Q)$ and 
$\dni \basisfn{1}{j}(\midp_{\edge_i}) = \delta_i^j$, if the unit normal vector $\f{n}_i$ is assumed to point inwards.

\subsection{Vertex basis functions}

Before we define the coefficient matrices for the basis functions, we need to define some precomputable coefficients. We assume that all normal vectors point inwards and have 
\begin{equation*}
 {\f n}_{\edge_k} = (n^{(k)}_1,n^{(k)}_2)^T = \frac{1}{\|\f{t}^{(k)}\|}(-t_2^{(k)},t_1^{(k)})^T.
\end{equation*}
Let 
\begin{equation*}
\f{q}^{(k)} = (q_1^{(k)},q_2^{(k)})^T = \vertex_{k} - \vertex_{k+1}+\vertex_{k+2}-\vertex_{k+3}
\end{equation*}
and moreover
\begin{equation*}\renewcommand{\arraystretch}{1.5}
\begin{array}{lll}
 b_{0}^{(k)} &=& \frac{\f{t}^{(k-1)}\f{t}^{(k)}}{\| \f{t}^{(k)} \|^2} , \\
 b_{1}^{(k)} &=& \frac{\f{t}^{(k+1)}\f{t}^{(k)}}{\| \f{t}^{(k)} \|^2}, \\
 T^{(k)}_{i,j} &=& t^{(k)}_i t^{(k)}_j, \\
 Q^{(k)}_{i,j} &=& t^{(k-1)}_i t^{(k)}_j + t^{(k-1)}_j t^{(k)}_i, \\
 N^{(k)}_{i,j} &=& n^{(k)}_i t^{(k)}_j + n^{(k)}_j t^{(k)}_i,
\end{array}
\end{equation*}
for $i,j\in\{1,2\}$ and $k\in\{1,2,3,4\}$. Here $k$ is considered modulo $4$. We define
\begin{small}
\begin{equation*}
{\f M}^L_k = 
\begin{array}{|c|c|c|c|c|c|} 
\hline
 0 & 0 & 0 & 0 & 0 & 0\\ \hline
 0 & 0 & 0 & 0 & 0 & 0\\ \hline
 0 & -\frac{3}{5}b^{(k-1)}_0 & 0 & 0 & 0 & 0\\ \hline
 1 & 1+\frac{3}{5}b^{(k-1)}_1 & 0 & 0 & 0 & 0\\ \hline
 \frac{1}{2} & \frac{1}{2} & 0 & 0 & 0 & 0\\ \hline
 0 & 0 & 0 & 0 & 0 & 0\\ \hline
\end{array} \;,
\end{equation*}
\end{small}
\begin{small}
\begin{equation*}
{\f M}^B_k = 
\begin{array}{|c|c|c|c|c|c|} 
\hline
 0 & 0 & 0 & 0 & 0 & 0\\ \hline
 0 & 0 & 0 & 0 & 0 & 0\\ \hline
 0 & 0 & 0 & 0 & 0 & 0\\ \hline
 0 & 0 & 0 & 0 & 0 & 0\\ \hline
 0 & \frac{1}{2} & 1+\frac{3}{5}b^{(k)}_0 & -\frac{3}{5}b^{(k)}_1 & 0 & 0\\ \hline
 0 & \frac{1}{2} & 1 & 0 & 0 & 0 \\ \hline
\end{array}
\end{equation*}
\end{small}
and
\begin{small}
\begin{equation*}
\f X = 
\begin{array}{|c|c|c|c|c|c|} 
\hline
 0 & 0 & 0 & 0 & 0 & 0\\ \hline
 0 & 0 & 0 & 0 & 0 & 0\\ \hline
 0 & 0 & 0 & 0 & 0 & 0\\ \hline
 0 & 0 & 0 & 0 & 0 & 0\\ \hline
 \frac{1}{2} & 0 & 0 & 0 & 0 & 0\\ \hline
 1 & \frac{1}{2} & 0 & 0 & 0 & 0 \\ \hline
\end{array} \;.
\end{equation*}
\end{small}
The vertex basis function $\basisfn{0}{1,0}$ is then given by 
\begin{equation*}
\f D [\basisfn{0}{1,0}] = {\f M}^L_1 + {\f M}^B_1 + \f X.
\end{equation*}
In general, the basis functions $\basisfn{0}{k,0}$ are given by 
\begin{equation*}
\f D [\basisfn{0}{k,0}] = R_k({\f M}^L_k + {\f M}^B_k + \f X)
\end{equation*}
where $R_k$ is a suitable operator $R_k:\mathbb{R}^{6\times 6} \rightarrow \mathbb{R}^{6\times 6}$ taking care of the local reparametrization, rotating the positions of the vertices.
Let
\begin{small}
\begin{equation*}
 \f Y_{k,i} =
 \begin{array}{|c|c|c|c|c|c|}
\hline
 0 & 0 & 0 & 0 & 0 & 0\\ \hline
 0 & 0 & 0 & 0 & 0 & 0\\ \hline
 0 & 0 & 0 & 0 & 0 & 0\\ \hline
 0 & -\frac{1}{5} t^{(k-2)}_i& 0 & 0 & 0 & 0\\ \hline
 0 & \frac{1}{10}q^{(k)}_i & \frac{1}{5} t^{(k+1)}_i & 
 0 & 0 & 0\\ \hline
 0 & 0 & 0 & 0 & 0& 0 \\ \hline
\end{array}
\end{equation*}
\end{small}
then the vertex basis functions $\basisfn{0}{k,1}$ and $\basisfn{0}{k,2}$ interpolating the derivatives in $x_1$- and $x_2$-direction, respectively, are given by
\begin{equation*}
\begin{array}{ll}
\f D [\basisfn{0}{k,i}] = & \frac{2}{5} R_k\left(-t^{(k-1)}_i {\f M}^L_k + t^{(k)}_i {\f M}^B_k + \f Y_{k,i}\right)\\
&-\frac{5}{16}n^{(k)}_i \f D[\basisfn{1}{k}]-\frac{5}{16}n^{(k-1)}_i \f D[\basisfn{1}{k-1}],
\end{array}
\end{equation*}
for $i=1,2$. Finally we define the vertex basis functions $\basisfn{0}{k,3}$, $\basisfn{0}{k,4}$ and $\basisfn{0}{k,5}$, interpolating the second derivatives. Let 
\begin{small}
\begin{equation*}
 \f Z_{k,(i,j)} =
 \begin{array}{|c|c|c|c|c|c|}
\hline
 0 & 0 & 0 & 0 & 0 & 0\\ \hline
 0 & 0 & 0 & 0 & 0 & 0\\ \hline
 0 & 0 & 0 & 0 & 0 & 0\\ \hline
 0 & \frac{1}{5} Q^{(k-1)}_{i,j} & 0 & 0 & 0 & 0\\ \hline
 -\frac{1}{2} T^{(k-1)}_{i,j} & -\frac{2}{5}Q^{(k)}_{i,j} - \frac{1}{2} T^{(k-1)}_{i,j} -\frac{1}{2} T^{(k)}_{i,j} & 
 \frac{1}{5} Q^{(k+1)}_{i,j} & 0 & 0 & 0\\ \hline
 0 & -\frac{1}{2} T^{(k)}_{i,j} & 0 & 0 & 0 & 0 \\ \hline
\end{array} \;,
\end{equation*}
\end{small}
then we have 
\begin{equation*}
\begin{array}{ll}
 \f D [\basisfn{0}{k,i+j+1}] = & \frac{\lambda}{20} R_k\left( T^{(k-1)}_{i,j} {\f M}^L_k +  T^{(k)}_{i,j} {\f M}^B_k + \f Z_{k,(i,j)} \right) \\ 
 &- \frac{\lambda}{32}N^{(k)}_{i,j}\f D[\basisfn{1}{k}]+\frac{\lambda}{32}N^{(k-1)}_{i,j}\f D[\basisfn{1}{k-1}]
\end{array}
\end{equation*}
for ${i,j}\in\{1,2\}$, where $\lambda = 2-\delta_i^j$. All representations of basis functions are verifiable via symbolic computation, e.g. by using Mathematica. We have $\basisfn{0}{k,j} \in \ker( \Dual_{1,\Q} \cup \mathrm{M}^5_{2,\Q})$ and 
\begin{equation*}
 (\basisfn{0}{k,0},\basisfn{0}{k,1},\basisfn{0}{k,2},\basisfn{0}{k,3},\basisfn{0}{k,4},\basisfn{0}{k,5})
\end{equation*}
being dual to 
\begin{equation*}
 (\func(\vertex_k),\dx\func(\vertex_k),\dy\func(\vertex_k),\dx\dx\func(\vertex_k),\dx\dy\func(\vertex_k),\dy\dy\func(\vertex_k)), 
\end{equation*}
with vanishing $C^2$-data at all other vertices.

\section{Quadrilateral macro-element: Definition and basis construction} \label{sec:Argyris_spline}

We can extend the definition of polynomial BS  quadrilaterals of degree $p\geq 5$ to certain B-spline based macro-elements of any degree $p\geq 3$. In that case the degrees of freedom are given as $C^2$-data in the 
vertices, normal derivative and point data at certain points along the edges, as well as suitably many interior functions that have vanishing values and gradients at all element boundaries. In such a setting, refinement can be performed either by splitting the macro-elements or by knot insertion within every macro-element. Note that, in the construction below, the continuity within the macro-element is of order $p-2$ for all degrees. In general, any order of continuity $r$, with $1\leq r\leq p-2$, can be achieved.

We assume that every quadrilateral $\Q$ is split into $k\times k$ elements by mapping a regular split of the parameter domain $\widehat{\Q} = [0,1]^2$ using $\param_\Q$. Let $\mathcal{S}^{p,r}_k$ be the univariate B-spline space of degree $p$ and regularity $r$ over the interval $[0,1]$ split into $k$ polynomial segments of the same length, i.e., having the knot vector 
\begin{equation*}
 \left(0,\ldots,0,\frac{1}{k},\frac{1}{k},\frac{2}{k},\frac{2}{k},\ldots,\frac{k-1}{k},\frac{k-1}{k},1,\ldots,1\right)
\end{equation*}
for $r=p-2$ and 
\begin{equation*}
 \left(0,\ldots,0,\frac{1}{k},\frac{2}{k},\ldots,\frac{k-1}{k},1,\ldots,1\right)
\end{equation*}
for $r=p-1$, where the first and last knots are repeated $p+1$ times. These knot vectors define piecewise polynomials on the split $s_k(\widehat{\Q})$ in the tensor-product case. Let $\theta^p_i$, for $i=0,\ldots,2k+p-2$ and $\eta^p_i$, for $i=0,\ldots,k+p-1$, be the corresponding Greville abscissae for the first and second knot vector, respectively. Note that the Greville abscissae $\gamma_i$ corresponding to a given knot vector $(\xi_0,\ldots,\xi_N)$ are defined as knot averages $\gamma_i = (\xi_{i+1}+\cdots+\xi_{i+p})/p$ for $i=0,\ldots,N-p-1$.

As in Definition~\ref{def:local-space-and-dofs} we can define the local function space and the local degrees of freedom, where we need to assume~$k\geq \max(1,6-p)$ in order to be able to split the vertex degrees of freedom.
\begin{definition}[Local space and degrees of freedom]\label{def:local-space-and-dofs-spline}
Given a quadrilateral $\Q$ with vertices $\vertex_1$, $\vertex_2$, $\vertex_3$ and $\vertex_4$ we define the $C^1$ quadrilateral spline macro-element of degree 
$p$ as 
$(\Q,\Space_\Q,\Dual_\Q)$, with
\begin{equation*}
 \Space_\Q = \left\{ \func:\Q\rightarrow \mathbb{R}, \mbox{ with} 
 \begin{array}{ll} (\func \circ \param_\Q) &\in \mathcal{S}^{p,p-2}_k \otimes \mathcal{S}^{p,p-2}_k,\\
 (\func \circ \param_\Q)|_{\hat{\edge}} &\in \mathcal{S}^{p,p-1}_k, \\
 (\dn \func \circ \param_\Q)|_{\hat{\edge}} &\in \mathcal{S}^{p-1,p-2}_k,
\end{array} \mbox{for each }\hat{\edge} \mbox{ of }\hat{\Q} \right\}
\end{equation*}
and 
\begin{equation*}
\begin{array}{l}
 \Dual_\Q = \Dual_{0,\Q} \cup \Dual^\ast_{1,\Q} \cup  \Dual^\ast_{2,\Q}, \mbox{ with } \\
 \;\Dual_{0,\Q} = \left\{\func(\vertex_i), \dx\func(\vertex_i), \dy\func(\vertex_i), \dx\dx\func(\vertex_i),\dx\dy\func(\vertex_i), \dy\dy\func(\vertex_i),\; 1 \leq i \leq 4 
 \right\}, \\
 \; \Dual^\ast_{1,\Q} = \left\{\varphi(r_{i,j_0}),\; \mbox{ for } 1 \leq i \leq 4,\; 1\leq j_0\leq k+p-6 \right\} \\ 
 \hspace{40pt}\cup \left\{\dni \varphi(q_{i,j_1}),\; \mbox{ for } 1 \leq i \leq 4,\; 1\leq j_1\leq k+p-5 \right\}, \\
 \; \Dual^\ast_{2,\Q} = \left\{ \func (\facep), \; \facep \in \Facep^\ast_\Q \right\}.
\end{array}
\end{equation*}
\end{definition}
Here $r_{i,j_0} = \param_{\edge_i}(\eta^p_{j_0+2})$, $q_{i,j_1} = \param_{\edge_i}(\eta^{p-1}_{j_1+1})$, with $\param_{\edge_i} = \param_\Q|_{\edge_i}$, and the 
set of face points is given as 
\begin{equation*}
 \Facep^\ast_{\Q} = \left\{ \param_\Q \left(\theta^p_{j_1},\theta^p_{j_2}\right), 2 \leq j_1,j_2\leq 2k+p-4 \right\}.
\end{equation*}
We have the following.
\begin{lemma}\label{lemma:unisolvency-spline-elements}
 Let $(\Q,\Space_\Q,\Dual_\Q)$ be the element defined in Definition~\ref{def:local-space-and-dofs-spline}. The degrees of freedom $\Dual_\Q$ completely determine the space $\Space_\Q$.
\end{lemma}
\begin{proof}
This lemma is a direct consequence of the results developed in~\cite[Section 4]{KaSaTa19}. For the sake of completeness, we present the main steps of the proof in the following. Let $\widehat{\varphi} = \varphi \circ \param_\Q$. Let us consider the conditions on ${\varphi}$ (and consequently on $\widehat{\varphi}$) along one edge $\edge$, w.l.o.g. $\hat{\edge}=\{0\}\times]0,1[$. For the unit normal vector $\f n$ along $\edge$, we then have 
\begin{equation*}
 \f n = \lambda(\xi_2) \partial_1 \param_\Q(0,\xi_2) + \mu(\xi_2) \partial_2 \param_\Q(0,\xi_2) = \lambda(\xi_2) \f v(\xi_2) - \mu(\xi_2) \f t^{(4)},
\end{equation*}
with $\f v(\xi_2) = \f t^{(1)} (1-\xi_2) -\f t^{(3)}\xi_2$,
where 
\begin{equation*}
\lambda(\xi_2) = \frac{\|\f t^{(4)}\|}{\det(\f v(\xi_2),\f t^{(4)})} = \frac{1}{\alpha(\xi_2)} \quad\mbox{ and }\quad
\mu(\xi_2) = \frac{\f v(\xi_2)\cdot\f t^{(4)}}{\det(\f v(\xi_2),\f t^{(4)}) \|\f t^{(4)}\|}= \frac{\beta(\xi_2)}{\alpha(\xi_2)}.
\end{equation*}
It is easy to check, that $\f n \cdot \f t^{(4)} = 0$ and 
\begin{equation*}
 \f n\cdot\f n = \frac{1}{\det(\f v(\xi_2),\f t^{(4)})^2} \left(\|\f v(\xi_2)\|^2 \|\f t^{(4)}\|^2 - (\f v(\xi_2)\cdot\f t^{(4)})^2 \right) =1.
\end{equation*}
Then, the chain rule yields
\begin{equation*}
 (\partial_{\f n} \func \circ \param_\Q)|_{\hat{\edge}} = \nabla \widehat\varphi \; (\nabla \param_\Q)^{-1} \cdot \f n|_{\hat{\edge}} = \lambda(\xi_2) \partial_1 \widehat{\varphi}(0,\xi_2) + \mu(\xi_2) \partial_2 \widehat{\varphi}(0,\xi_2).
\end{equation*}
Let $f_0 = \widehat{\varphi}|_{\hat{\edge}} \in \mathcal{S}^{p,p-1}_k$ and $f_1 = (\partial_{\f n} \func \circ \param_\Q)|_{\hat{\edge}} \in \mathcal{S}^{p-1,p-2}_k$. Then 
\begin{equation*}
 \widehat{\varphi}(0,\xi_2) = f_0(\xi_2)
\end{equation*}
and
\begin{equation*}
 \partial_1 \widehat{\varphi}(0,\xi_2) = \alpha(\xi_2)f_1(\xi_2) - \beta(\xi_2) f_0'(\xi_2) \in \mathcal{S}^{p,p-2}_k,
\end{equation*}
since $\alpha(\xi_2)$ and $\beta(\xi_2)$ are linear functions.

The trace $f_0 \in \mathcal{S}^{p,p-1}_k$, with $\dim(\mathcal{S}^{p,p-1}_k)=k+p$, is completely determined by the $C^2$-data at the vertices ($3$ degrees of freedom at each vertex) together with the $k+p-6$ point evaluations
$\varphi(r_{i,j_0})$, since the points are selected as suitable mapped Greville points. Similarly, the normal derivative along an edge $f_1\in \mathcal{S}^{p-1,p-2}_k$, a spline space of dimension $k+p-1$, is completely determined by the $C^2$-data at the vertices (determining $2$ degrees of freedom each) as well as by the $k+p-5$ normal derivative evaluations at Greville points, i.e., $\dni \varphi(q_{i,j_1})$. The space $\mathcal{S}^{p,p-2}_k \otimes \mathcal{S}^{p,p-2}_k$ has a standard tensor-product basis $\{ \hat{b}_{i_1}(\xi_1) \hat{b}_{i_2}(\xi_2) , 0\leq i_1,i_2 \leq 2k+p-2 \}$. Hence, every function $\widehat{\varphi} \in \mathcal{S}^{p,p-2}_k \otimes \mathcal{S}^{p,p-2}_k$ is represented by coefficients $c_{i_1,i_2}$, such that
\begin{equation*}
 \widehat{\varphi}(\xi_1,\xi_2) = \sum_{i_1=0}^{n-1}\sum_{i_2=0}^{n-1} c_{i_1,i_2} \hat{b}_{i_1}(\xi_1) \hat{b}_{i_2}(\xi_2).
\end{equation*}
with $n=2k+p-1$. All coefficients $c_{0,i_2}$, $c_{1,i_2}$, for $0\leq i_2\leq n$, are determined by $f_0$ and $f_1$. Analogously, the coefficients $c_{n-1,i_2}$, $c_{n-2,i_2}$, $c_{i_1,0}$, $c_{i_1,1}$, $c_{i_1,n-1}$ and $c_{i_1,n-2}$ corresponding to the remaining edges are also determined by $\Dual_{0,\Q}$ and $\Dual^\ast_{1,\Q}$. All remaining coefficients $c_{i_1,i_2}$, for $2\leq i_1,i_2\leq n-3$, corresponding to the space
\begin{equation*}
 \{ \widehat{\varphi} \in \mathcal{S}^{p,p-2}_k \otimes \mathcal{S}^{p,p-2}_k : \widehat{\varphi}|_{\partial\widehat{\Omega}} = 0, \nabla\widehat{\varphi}|_{\partial\widehat{\Omega}} = \f 0 \}
\end{equation*}
are determined by the $(n-4)^2$ evaluations at mapped Greville points $\param_\Q \left(\theta^p_{j_1},\theta^p_{j_2}\right)$. Consequently, the full space $\Space_\Q$ is completely determined by the dual functionals $\Dual_\Q$ and the proof is complete.
\end{proof}
It follows immediately from
Lemma~\ref{lemma:unisolvency-spline-elements} that the dimension of
the space can be determined completely by counting, having six degrees
of freedom per vertex, $2k+2p-11$ degrees of freedom per edge and
$(2k+p-5)^2$ degrees of freedom inside the element. Thus, we need
$2k+2p-11 \geq 0$, yielding the constraint $k\geq 6-p$. 

In the remainder of this section, we present in more detail the two special cases of $C^1$ quadrilateral macro-elements presented in Definition~\ref{def:local-space-and-dofs-34}. Since we need $k\geq \max(1,6-p)$, they represent the spline elements with the least number of inner knots, allowing a separation of degrees of freedom at the vertices. For $p=4$ we consider the spline space with one inner knot at $\frac{1}{2}$ with multiplicity two in each direction $\mathcal{S}^{4,2}_{2} \otimes \mathcal{S}^{4,2}_{2}$, having the basis $\hat{b}^4_{j_1,j_2}$ and corresponding dual basis $\hat{\mu}^4_{j_1,j_2}$ for $0\leq j_1,j_2\leq 6$. For $p=3$ we consider the spline space $\mathcal{S}^{3,1}_{3} \otimes \mathcal{S}^{3,1}_{3}$, with basis $\hat{b}^3_{j_1,j_2}$ and dual basis $\hat{\mu}^3_{j_1,j_2}$ for $0\leq j_1,j_2\leq 7$, 
see~\cite{PrBoPa02,Sc07}. 

Hence, for smaller degrees that patches $\Q$ are macro-elements with $2\times2$ (for $p=4$) or $3\times3$ (for $p=3$) polynomial sub-elements. Let $n=11-p$. We write, as for $p=5$, all tensor-product basis functions in a matrix
\begin{equation*}
\f B=
\left(
\begin{array}{ccc} 
\hat{b}^p_{0,n-1} & \ldots & \hat{b}^p_{n-1,n-1} \\
\vdots & & \vdots \\
\hat{b}^p_{0,0} & \ldots & \hat{b}^p_{n-1,0}
\end{array}
\right)
\end{equation*}
and denote again with $\f D[\basisf]$ the $(n\times n)$-matrix of coefficients. As for $p=5$, let ${b}^p_{j_1,j_2} = \hat{b}^p_{j_1,j_2} \circ {\param_{\Q}}^{-1}$ denote the basis 
functions on the element $\Q$.

\subsection{Patch interior basis functions}

We have $(n-4)^2$ basis functions
\begin{equation*}
 \mathrm{B}^p_{2,\Q} = \{ b^p_{j_1,j_2}, \; 2\leq j_1,j_2 \leq n-3 \},
\end{equation*}
which satisfy $\mbox{span}(\mathrm{B}^p_{2,\Q}) = \ker(  \Dual_{0,\Q} \cup \Dual_{1,\Q})$.

\subsection{Edge basis functions}

The edge basis function $\basisfn{1}{1}$, corresponding to edge $\edge_1$, is given for $p=4$ by
\begin{small}
\begin{equation*}
\f D [\basisfn{1}{1}] = \frac{1}{32 \|\f t^{(1)} \|} \;
\begin{array}{|c|c|c|c|c|c|c|} 
\hline
 0 & 0 & 0 & 0 & 0 & 0 & 0\\ \hline
 0 & 0 & 0 & 0 & 0 & 0 & 0\\ \hline
 0 & 0 & 0 & 0 & 0 & 0 & 0\\ \hline
 0 & 0 & 0 & 0 & 0 & 0 & 0\\ \hline
 0 & 0 & 0 & 0 & 0 & 0 & 0\\ \hline
 0 & 0 & 2a^{(1)} & 3a^{(1)}+3a^{(2)} & 2a^{(2)} & 0 & 0\\ \hline
 0 & 0 & 0 & 0 & 0 & 0 & 0\\ \hline
\end{array}\;,
\end{equation*}
\end{small}
and for $p=3$ by 
\begin{small}
\begin{equation*}
\f D [\basisfn{1}{1}] = \frac{2}{81 \|\f t^{(1)} \|} \;
\begin{array}{|c|c|c|c|c|c|c|c|} 
\hline
 0 & 0 & 0 & 0 & 0 & 0 & 0 & 0\\ \hline
 0 & 0 & 0 & 0 & 0 & 0 & 0 & 0\\ \hline
 0 & 0 & 0 & 0 & 0 & 0 & 0 & 0\\ \hline
 0 & 0 & 0 & 0 & 0 & 0 & 0 & 0\\ \hline
 0 & 0 & 0 & 0 & 0 & 0 & 0 & 0\\ \hline
 0 & 0 & 0 & 0 & 0 & 0 & 0 & 0\\ \hline
 0 & 0 & a^{(1)} & 3a^{(1)}+2a^{(2)} & 2a^{(1)}+3a^{(2)} & a^{(2)} & 0 & 0\\ \hline
 0 & 0 & 0 & 0 & 0 & 0 & 0 & 0\\ \hline
\end{array} \; .
\end{equation*}
\end{small}
The functions $\basisfn{1}{2}$, $\basisfn{1}{3}$ and $\basisfn{1}{4}$ are defined analogously. Analogously to the polynomial case, we have $\basisfn{1}{j} \in \ker( \Dual_{0,\Q} \cup \mathrm{M}^p_{2,\Q})$ and 
$\dni \basisfn{1}{j}(\midp_{\edge_i}) = \delta_i^j$, if the unit normal vector $\f{n}_i$ is assumed to point inwards.

\subsection{Vertex basis functions, $p=4$}

We define
\begin{small}
\begin{equation*}
{\f M}^L_k = 
\begin{array}{|c|c|c|c|c|c|c|} 
\hline
 0 & 0 & 0 & 0 & 0 & 0 & 0\\ \hline
 0 & 0 & 0 & 0 & 0 & 0 & 0\\ \hline
 0 & -\frac{1}{8} b^{(k-1)}_0 & 0 & 0 & 0 & 0 & 0\\ \hline
 \frac{1}{2} & \frac{1}{2}+ \frac{3}{16} ( b^{(k-1)}_1- b^{(k-1)}_0) & 0 & 0 & 0 & 0 & 0\\ \hline
 \frac{2}{3} & \frac{2}{3}+ \frac{1}{8} b^{(k-1)}_1 & 0 & 0 & 0 & 0 & 0\\ \hline
 \frac{1}{3} & \frac{1}{3} & 0 & 0 & 0 & 0 & 0 \\ \hline
 0 & 0 & 0 & 0 & 0 & 0 & 0 \\ \hline
\end{array}\;,
\end{equation*}
\end{small}
\begin{small}
\begin{equation*}
{\f M}^B_k = 
\begin{array}{|c|c|c|c|c|c|c|} 
\hline
 0 & 0 & 0 & 0 & 0 & 0 & 0\\ \hline
 0 & 0 & 0 & 0 & 0 & 0 & 0\\ \hline
 0 & 0 & 0 & 0 & 0 & 0 & 0\\ \hline
 0 & 0 & 0 & 0 & 0 & 0 & 0\\ \hline
 0 & 0 & 0 & 0 & 0 & 0 & 0\\ \hline
 0 & \frac{1}{3} & \frac{2}{3} + \frac{1}{8} b^{(k)}_0& \frac{1}{2}+\frac{3}{16} (b^{(k)}_0- b^{(k)}_1) & -\frac{1}{8} b^{(k)}_1 & 0 & 0\\ \hline
 0 & \frac{1}{3} & \frac{2}{3} & \frac{1}{2} & 0 & 0 & 0 \\ \hline
\end{array}
\end{equation*}
\end{small}
and 
\begin{small}
\begin{equation*}
\f X = 
\begin{array}{|c|c|c|c|c|c|c|} 
\hline
 0 & 0 & 0 & 0 & 0 & 0 & 0 \\ \hline
 0 & 0 & 0 & 0 & 0 & 0 & 0 \\ \hline
 0 & 0 & 0 & 0 & 0 & 0 & 0 \\ \hline
 0 & 0 & 0 & 0 & 0 & 0 & 0 \\ \hline
 \frac{1}{3} & \frac{1}{3} & 0 & 0 & 0 & 0 & 0 \\ \hline
 \frac{2}{3} & \frac{1}{3} & \frac{1}{3} & 0 & 0 & 0 & 0 \\ \hline
 1 & \frac{2}{3} & \frac{1}{3} & 0 & 0 & 0 & 0 \\ \hline
\end{array} \;.
\end{equation*}
\end{small}
The basis functions $\basisfn{0}{k,0}$ are given by 
\begin{equation*}
\f D [\basisfn{0}{k,0}] = R_k({\f M}^L_k + {\f M}^B_k + \f X).
\end{equation*}
Let 
\begin{small}
\begin{equation*}
 \f Y_{k,i} =
  \begin{array}{|c|c|c|c|c|c|c|}
\hline
 0 & 0 & 0 & 0 & 0 & 0 & 0\\ \hline
 0 & 0 & 0 & 0 & 0 & 0 & 0\\ \hline
 0 & 0 & 0 & 0 & 0 & 0 & 0\\ \hline
 0 & -\frac{1}{24}t^{(k-2)}_i & 0 & 0 & 0 & 0 & 0\\ \hline
 0 & -\frac{1}{4}t^{(k-2)}_i -\frac{1}{6}q^{(k)}_i & 0 & 0 & 0 & 0 & 0\\ \hline
 0 & \frac{1}{24}q^{(k)}_i & \frac{1}{4}t^{(k+1)}_i - \frac{1}{6}q^{(k)}_i & \frac{1}{24}t^{(k+1)}_i & 0 & 0 & 0\\ \hline
 0 & 0 & 0 & 0 & 0 & 0 & 0 \\ \hline
\end{array} \;,
\end{equation*}
\end{small}
then the vertex basis functions $\basisfn{0}{k,1}$ and $\basisfn{0}{k,2}$ are given by
\begin{equation*}
\begin{array}{ll}
\f D [\basisfn{0}{k,i}] =& \frac{3}{8} R_k\left(-t^{(k-1)}_i {\f M}^L_k + t^{(k)}_i {\f M}^B_k + \f Y_{k,i}\right)\\
&-\frac{1}{4}n^{(k)}_i \f D[\basisfn{1}{k}]-\frac{1}{4}n^{(k-1)}_i \f D[\basisfn{1}{k-1}],
\end{array}
\end{equation*}
for $i=1,2$. Let 
\begin{small}
\begin{equation*}
\begin{array}{l}
 \f Z_{k,(i,j)} = \\
 \begin{array}{|c|c|c|c|c|c|}
\hline
 0 & 0 & 0 & 0 & 0 & \ldots \\ \hline
 0 & 0 & 0 & 0 & 0 &  \\ \hline
 0 & 0 & 0 & 0 & 0 &  \\ \hline
 0 & \frac{1}{32}Q^{(k-1)}_{i,j} & 0 & 0 & 0 &  \\ \hline
 -\frac{1}{6}T^{(k-1)}_{i,j} & -\frac{1}{6}T^{(k-1)}_{i,j}  & 0 & 0 & 0 &  \\ 
  &  - \frac{1}{8}Q^{(k)}_{i,j} + \frac{1}{16}Q^{(k-1)}_{i,j} &  &  &  &  \\ \hline
 -\frac{1}{3} T^{(k-1)}_{i,j} & -\frac{1}{3} T^{(k-1)}_{i,j} & 
  - \frac{1}{8}Q^{(k)}_{i,j} + \frac{1}{16}Q^{(k+1)}_{i,j} & \frac{1}{32}Q^{(k+1)}_{i,j} & 0 &  \\ 
  & -\frac{1}{3} T^{(k)}_{i,j} - \frac{3}{16}Q^{(k)}_{i,j} & 
 -\frac{1}{6}T^{(k)}_{i,j} &  &  &  \\ \hline
 0 & -\frac{1}{3} T^{(k)}_{i,j} & -\frac{1}{6} T^{(k)}_{i,j} & 0 & 0 & \ldots  \\ \hline
\end{array} \;,
\end{array}
\end{equation*}
\end{small}
then we have 
\begin{equation*}
\begin{array}{ll}
 \f D [\basisfn{0}{k,i+j+1}] = & \frac{\lambda}{24} R_k\left(T^{(k-1)}_{i,j} {\f M}^L_k + T^{(k)}_{i,j} {\f M}^B_k + \f Z_{k,(i,j)} \right) \\
 &- \frac{\lambda}{48}N^{(k)}_{i,j}\f D[\basisfn{1}{k}]+\frac{\lambda}{48}N^{(k-1)}_{i,j}\f D[\basisfn{1}{k-1}]
\end{array}
\end{equation*}
for ${i,j}\in\{1,2\}$, where $\lambda = 2-\delta_i^j$. We have $\basisfn{0}{k,j} \in \ker(  \Dual_{1,\Q} \cup \mathrm{M}^4_{2,\Q})$ and $\{\basisfn{0}{k,j}\}_{k=1,\ldots,4,j=0,\ldots,5}$ being dual to $\Dual_{0,\Q}$.

\subsection{Vertex basis functions, $p=3$}

We define
\begin{small}
\begin{equation*}
\begin{array}{l}
{\f M}^L_k = \\
\begin{array}{|c|c|c|c|c|c|c|c|} 
\hline
 0 & 0 & 0 & 0 & 0 & 0 & 0 & 0\\ \hline
 0 & 0 & 0 & 0 & 0 & 0 & 0 & 0\\ \hline
 0 & -\frac{1}{18} b^{(k-1)}_0 & 0 & 0 & 0 & 0 & 0 & 0\\ \hline
 \frac{1}{3} & \frac{1}{3}+ \frac{1}{9} b^{(k-1)}_1- \frac{1}{6} b^{(k-1)}_0 & 0 & 0 & 0 & 0 & 0 & 0\\ \hline
 \frac{2}{3} & \frac{2}{3}+ \frac{1}{6} b^{(k-1)}_1- \frac{1}{9} b^{(k-1)}_0 & 0 & 0 & 0 & 0 & 0 & 0\\ \hline
 \frac{2}{3} & \frac{2}{3}+ \frac{1}{18} b^{(k-1)}_1 & 0 & 0 & 0 & 0 & 0 & 0\\ \hline
 \frac{1}{3} & \frac{1}{3} & 0 & 0 & 0 & 0 & 0 & 0 \\ \hline
 0 & 0 & 0 & 0 & 0 & 0 & 0 & 0 \\ \hline
\end{array}\;,
\end{array}
\end{equation*}
\end{small}
\begin{small}
\begin{equation*}
\begin{array}{l}
{\f M}^B_k = \\
\begin{array}{|c|c|c|c|c|c|c|c|} 
\hline
 0 & 0 & 0 & 0 & 0 & 0 & 0 & 0\\ \hline
 0 & 0 & 0 & 0 & 0 & 0 & 0 & 0\\ \hline
 0 & 0 & 0 & 0 & 0 & 0 & 0 & 0\\ \hline
 0 & 0 & 0 & 0 & 0 & 0 & 0 & 0\\ \hline
 0 & 0 & 0 & 0 & 0 & 0 & 0 & 0\\ \hline
 0 & 0 & 0 & 0 & 0 & 0 & 0 & 0\\ \hline
 0 & \frac{1}{3} & \frac{2}{3} + \frac{1}{18} b^{(k)}_0 & \frac{2}{3}+\frac{1}{6} b^{(k)}_0 -\frac{1}{9} b^{(k)}_1 & \frac{1}{3}+ \frac{1}{9} b^{(k)}_0- \frac{1}{6} b^{(k)}_1 & -\frac{1}{18} b^{(k)}_1 & 0 & 0\\ \hline
 0 & \frac{1}{3} & \frac{2}{3} & \frac{2}{3} & \frac{1}{3} & 0 & 0 & 0 \\ \hline
\end{array}
\end{array}
\end{equation*}
\end{small}
and 
\begin{small}
\begin{equation*}
\f X = 
\begin{array}{|c|c|c|c|c|c|c|c|} 
\hline
 0 & 0 & 0 & 0 & 0 & 0 & 0 & 0 \\ \hline
 0 & 0 & 0 & 0 & 0 & 0 & 0 & 0 \\ \hline
 0 & 0 & 0 & 0 & 0 & 0 & 0 & 0 \\ \hline
 0 & 0 & 0 & 0 & 0 & 0 & 0 & 0 \\ \hline
 0 & 0 & 0 & 0 & 0 & 0 & 0 & 0 \\ \hline
 \frac{1}{3} & \frac{1}{3} & 0 & 0 & 0 & 0 & 0 & 0 \\ \hline
 \frac{2}{3} & \frac{1}{3} & \frac{1}{3} & 0 & 0 & 0 & 0 & 0 \\ \hline
 1 & \frac{2}{3} & \frac{1}{3} & 0 & 0 & 0 & 0 & 0 \\ \hline
\end{array} \;.
\end{equation*}
\end{small}
The basis functions $\basisfn{0}{k,0}$ are given by 
\begin{equation*}
\f D [\basisfn{0}{k,0}] = R_k({\f M}^L_k + {\f M}^B_k + \f X).
\end{equation*}
Let 
\begin{small}
\begin{equation*}
\begin{array}{l}
 \f Y_{k,i} =\\
   \begin{array}{|c|c|c|c|c|c|c|c|}
\hline
 0 & 0 & 0 & 0 & 0 & 0 & 0 & 0\\ \hline
 0 & 0 & 0 & 0 & 0 & 0 & 0 & 0\\ \hline
 0 & 0 & 0 & 0 & 0 & 0 & 0 & 0\\ \hline
 0 & 0 & 0 & 0 & 0 & 0 & 0 & 0\\ \hline
 0 & -\frac{1}{18}t^{(k-2)}_i - \frac{1}{54}q^{(k)}_i & 0 & 0 & 0 & 0 & 0 & 0\\ \hline
 0 & -\frac{5}{18}t^{(k-2)}_i - \frac{11}{54}q^{(k)}_i & 0 & 0 & 0 & 0 & 0 & 0\\ \hline
 0 & \frac{1}{27}q^{(k)}_i & \frac{5}{18}t^{(k+1)}_i - \frac{11}{54}q^{(k)}_i & \frac{1}{18}t^{(k+1)}_i  -\frac{1}{54}q^{(k)}_i & 0 & 0 & 0 & 0\\ \hline
 0 & 0 & 0 & 0 & 0 & 0 & 0 & 0 \\ \hline
\end{array}  
\end{array}
\end{equation*}
\end{small}
then the vertex basis functions $\basisfn{0}{k,1}$ and $\basisfn{0}{k,2}$ are given by
\begin{equation*}
\begin{array}{ll}
\f D [\basisfn{0}{k,i}] =& \frac{1}{3} R_k\left(-t^{(k-1)}_i {\f M}^L_k + t^{(k)}_i {\f M}^B_k + \f Y_{k,i}\right)\\
&-\frac{1}{8}n^{(k)}_i \f D[\basisfn{1}{k}]-\frac{1}{8}n^{(k-1)}_i \f D[\basisfn{1}{k-1}],
\end{array}
\end{equation*}
for $i=1,2$. Let 
\begin{small}
\begin{equation*}
\begin{array}{l}
 \f Z_{k,(i,j)} =\\
 \begin{array}{|c|c|c|c|c|c|}
\hline
 0 & 0 & 0 & 0 & 0 & \ldots \\ \hline
 0 & 0 & 0 & 0 & 0 &  \\ \hline
 0 & 0 & 0 & 0 & 0 &  \\ \hline
 0 & 0 & 0 & 0 & 0 &  \\ \hline
 0 & - \frac{1}{72}Q^{(k)}_{i,j} + \frac{1}{36}Q^{(k-1)}_{i,j} & 0 & 0 & 0 &  \\ \hline
 -\frac{1}{6}T^{(k-1)}_{i,j} & - \frac{1}{6} T^{(k-1)}_{i,j} & 0 & 0 & 0 &  \\ 
  &  - \frac{11}{72}Q^{(k)}_{i,j} + \frac{1}{18}Q^{(k-1)}_{i,j} & & & &  \\ \hline
 -\frac{1}{3}T^{(k-1)}_{i,j} & -\frac{1}{3}T^{(k-1)}_{i,j} & 
 - \frac{11}{72}Q^{(k)}_{i,j} + \frac{1}{18}Q^{(k+1)}_{i,j} & \frac{1}{36}Q^{(k+1)}_{i,j} & 0 &  \\ 
  & -\frac{1}{3}T^{(k)}_{i,j}  - \frac{1}{6}Q^{(k)}_{i,j} & 
 -\frac{1}{6} T^{(k)}_{i,j}  & -\frac{1}{72}Q^{(k)}_{i,j} &  &  \\ \hline
 0 & -\frac{1}{3}T^{(k)}_{i,j} & -\frac{1}{6}T^{(k)}_{i,j} & 0 & 0 & \ldots \\ \hline
\end{array} \;,
\end{array}
\end{equation*}
\end{small}
then we have 
\begin{equation*}
\begin{array}{ll}
 \f D [\basisfn{0}{k,i+j+1}] =&\frac{\lambda}{27} R_k\left(T^{(k-1)}_{i,j} {\f M}^L_k + T^{(k)}_{i,j} {\f M}^B_k + \f Z_{k,(i,j)} \right) \\
 &- \frac{\lambda}{96}N^{(k)}_{i,j}\f D[\basisfn{1}{k}]+\frac{\lambda}{96}N^{(k-1)}_{i,j}\f D[\basisfn{1}{k-1}]
\end{array}
\end{equation*}
for ${i,j}\in\{1,2\}$, where $\lambda = 2-\delta_i^j$.  We have $\basisfn{0}{k,j} \in \ker( \Dual_{1,\Q} \cup \mathrm{M}^3_{2,\Q})$ and $\{\basisfn{0}{k,j}\}_{k=1,\ldots,4,j=0,\ldots,5}$ being dual to $\Dual_{0,\Q}$.

As for $p=5$, all representations of basis functions for $p\in\{3,4\}$ can be verified using simple symbolic computations.

\section{Extension to isoparametric/isogeometric elements} \label{sec:extension}

As pointed out before, the BS quadrilaterals and related spline macro-elements are not affine invariant. Hence their definition depends on the underlying geometry. It is possible to extend the construction from bilinearly mapped quadrilaterals $\Q$, with $\param_{\Q} \in(\mathbb{P}^{(1,1)})^2$, to fully isoparametric elements, with $\param_{\Q} \in(\Space_\Q^p)^2$. This naturally leads to the $C^1$ multi-patch isogeometric space proposed in \cite{KaSaTa19b}, which is based on the previous works \cite{KaViJu15,CoSaTa16,KaSaTa17,KaSaTa17b}. The isoparametric/isogeometric extension, however, needs some additional care, in order to guarantee optimal approximation properties. First, the definition of the space $\Space_\Q^p$ and the associated degrees of freedom need to be generalized, mainly replacing the normal derivative $(\dn \func \circ \param_\Q)|_{\hat{\edge}}$  with a suitable directional derivative $(\boldsymbol{d} \cdot \nabla  \func \circ \param_\Q)|_{\hat{\edge}}$ (in the bilinear case, $\boldsymbol{d}$ is a constant normal vector which is then rescaled to the unitary normal $\boldsymbol{n}$). Secondly, and most importantly, the element  parametrizations need to fulfill a condition (named analysis-suitable $G^1$ in the papers mentioned above) that holds for all  bilinear parametrizations but requires a suitable refitting for higher degree parametrizations, see~\cite{KaSaTa17b}.

Modifications of elements near curved boundaries were also discussed and 
resolved successfully in~\cite{BeMa14} for a $C^1$ space of degree $4$ and $5$ over a quadrilateral mesh.  There, the authors presented the construction of a minimal determining set (similar to a dual basis), without giving explicit degrees of freedom or a basis representation.

A complete analysis of the convergence in case of local modifications near the boundary is not known and beyond the scope of the current paper. It is important to note, that a suitable splitting of elements can increase the flexibility of the resulting space, such as in~\cite{HaBo00}, where using a regular $4$-split on degree $p=5$ triangular elements allows for the construction of surfaces of arbitrary topology.

\section{Numerical examples} \label{sec:examples}

The goal is to demonstrate the potential of using the proposed $C^1$ spaces over quadrilateral meshes~$\Mesh$ for solving fourth order PDEs over domains~$\Omega$ with piecewise linear boundary. This is done on the basis of a particular example, namely for the biharmonic equation 
\begin{equation} \label{eq:problem_biharmonic}
\left\{
 \begin{array}{rll} 
 \triangle^{2} u(\f{x}) & = g(\f{x}) & \f{x} \in \Omega \\
  u(\f{x}) & = g_1(\f{x})  & \f{x} \in \partial \Omega \\
  \frac{\partial u}{\partial \f{n}}(\f{x})  & = g_2(\f{x}) &  \f{x} \in \partial \Omega .
        \end{array} \right. 
\end{equation}
More precisely, we solve problem~\eqref{eq:problem_biharmonic} via a standard Galerkin discretization by employing the family of  $C^1$ quadrilateral 
spaces~$\globalspace^p(\Mesh_h)$, where the mesh size~$h$ denotes the length of the longest edge in $\Mesh_h$, with $h=h_0\frac{1}{2^L}$, $L=0,1,\ldots,5$. Here $L$ denotes the level of refinement, $h_0$ is the mesh size of the initial mesh~$\Mesh$, and $\Mesh_h=(\Quad_h,\Edge_h,\Vertices_h)$ is the resulting 
refined quadrilateral mesh obtained from~$\Mesh$ with corresponding sets of quadrilaterals~$\Quad_h$, edges~$\Edge_h$ and vertices~$\Vertices_h$. 
Note that in the refinement process, each quadrilateral of the current mesh is split regularly into four sub-quadrilaterals. Moreover, in all examples below, 
the functions~$g$, $g_1$ and $g_2$ from problem~\eqref{eq:problem_biharmonic} are computed from an exact solution~$u$, and the resulting Dirichlet boundary 
data~$g_1$ and $g_2$ are $L^2$ projected and strongly imposed to the numerical solution~$u_h \in \globalspace^p(\Mesh_h)$. 

\begin{example} \label{ex:example_multiholedomain}
For the two meshes~$\Mesh$ in Fig.~\ref{fig:example_multidomain} and Fig.~\ref{fig:example_holedomain}, which are visualized in the top left of each figure, 
we solve the biharmonic 
equation~\eqref{eq:problem_biharmonic} over the corresponding bilinear multi-patch domains by using the BS  quadrilateral and macro-element spaces~$\globalspace^p(\Mesh_h)$ for polynomial degrees~$p=3,4,5$. For both cases, the considered exact solution is given by 
\begin{equation} \label{eq:exact_multiholedomain}
 u(x_1,x_2)= - 4 \cos \left(\frac{x_1}{2}\right) \sin \left(\frac{x_2}{2}\right),
\end{equation}
and is shown in Fig.~\ref{fig:example_multidomain} (top row, right) and Fig.~\ref{fig:example_holedomain} (top row, right), respectively. The resulting $L^{\infty}$-error as well 
as the relative $L^2$, $H^1$ and $H^2$-errors with respect to the number of degrees of freedom (NDOF) are shown in the middle and bottom rows of Fig.~\ref{fig:example_multidomain} and 
Fig.~\ref{fig:example_holedomain}, and decrease for both examples with optimal order of $\mathcal{O}(h^{p+1})$, $\mathcal{O}(h^{p+1})$, $\mathcal{O}(h^p)$ and $\mathcal{O}(h^{p-1})$, respectively.

\begin{figure}
\centering\footnotesize
\begin{tabular}{cc}
\includegraphics[width=.35\textwidth,clip]{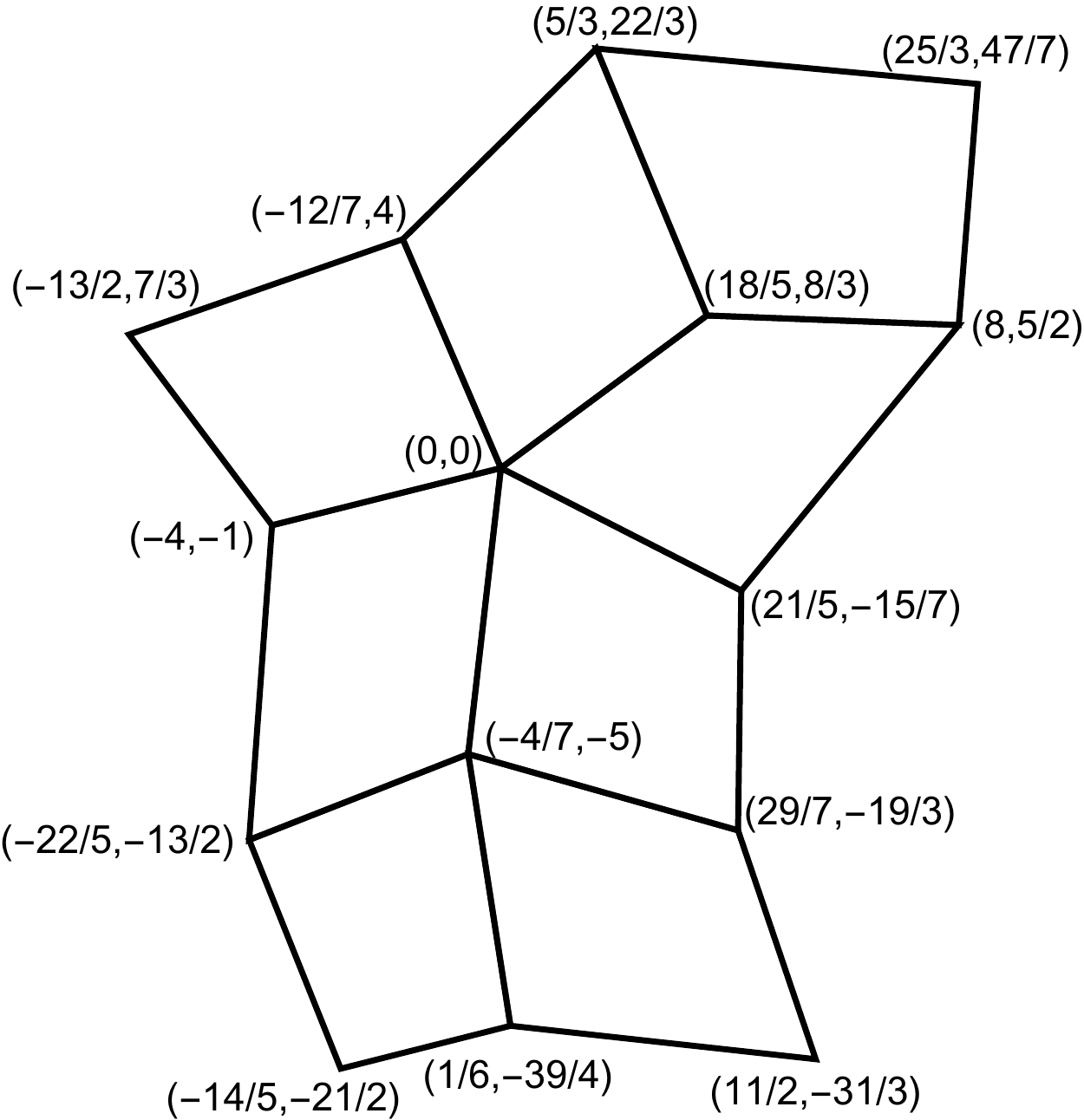} &
\includegraphics[width=.35\textwidth,clip]{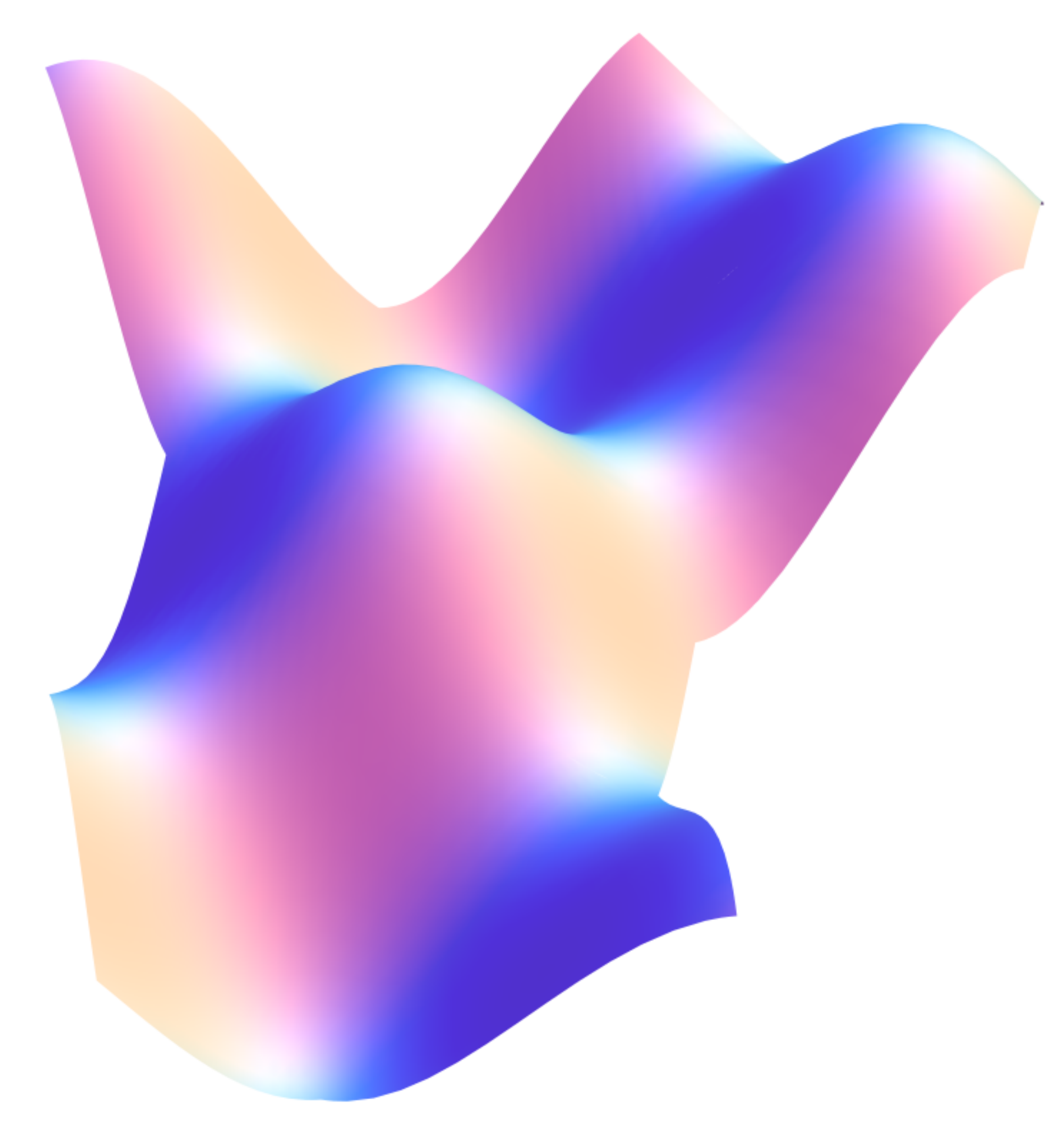} \\
Quadrilateral mesh~$\Mesh$ & Exact solution \\
\includegraphics[width=.45\textwidth,clip]{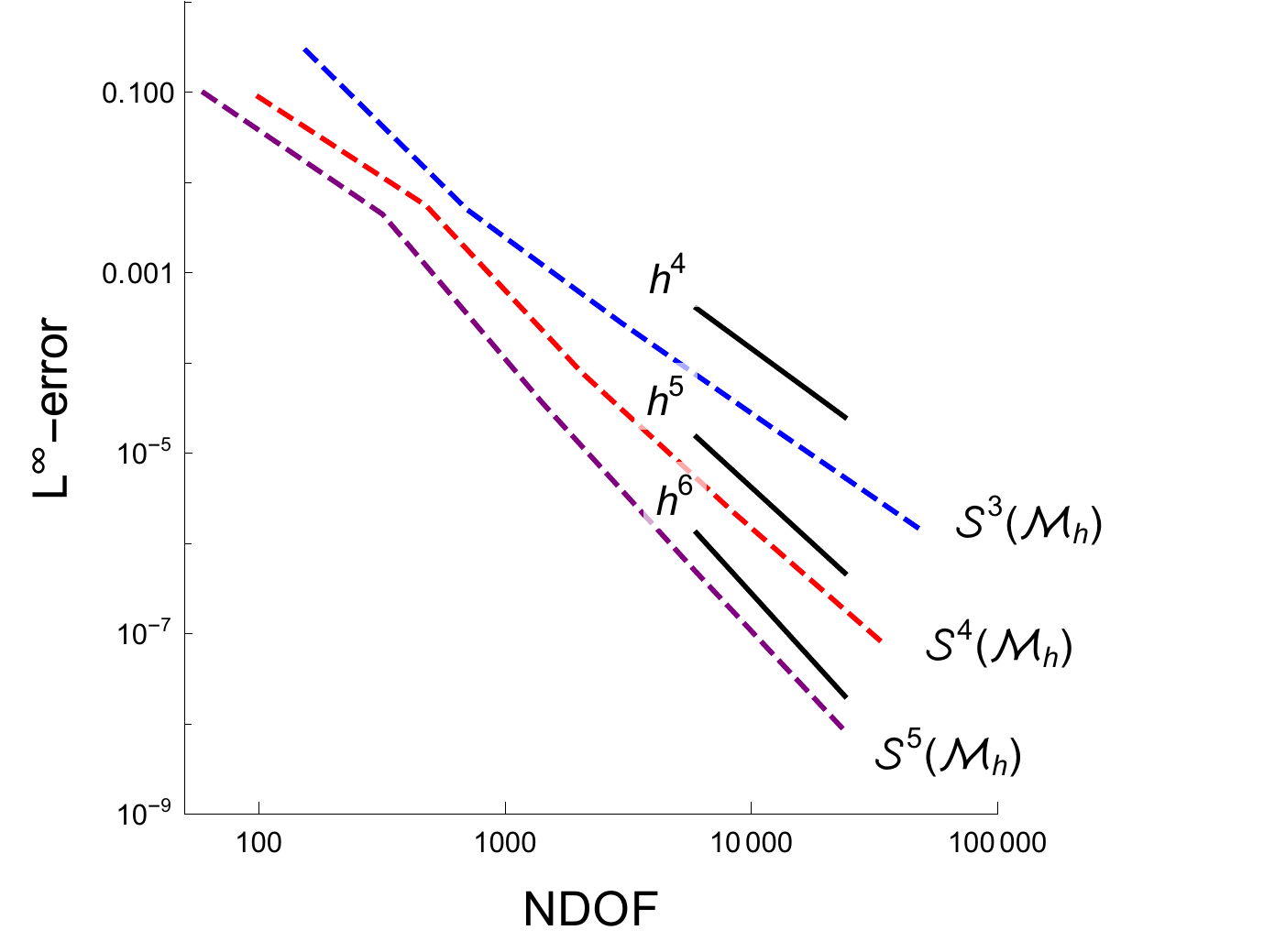} &
\includegraphics[width=.45\textwidth,clip]{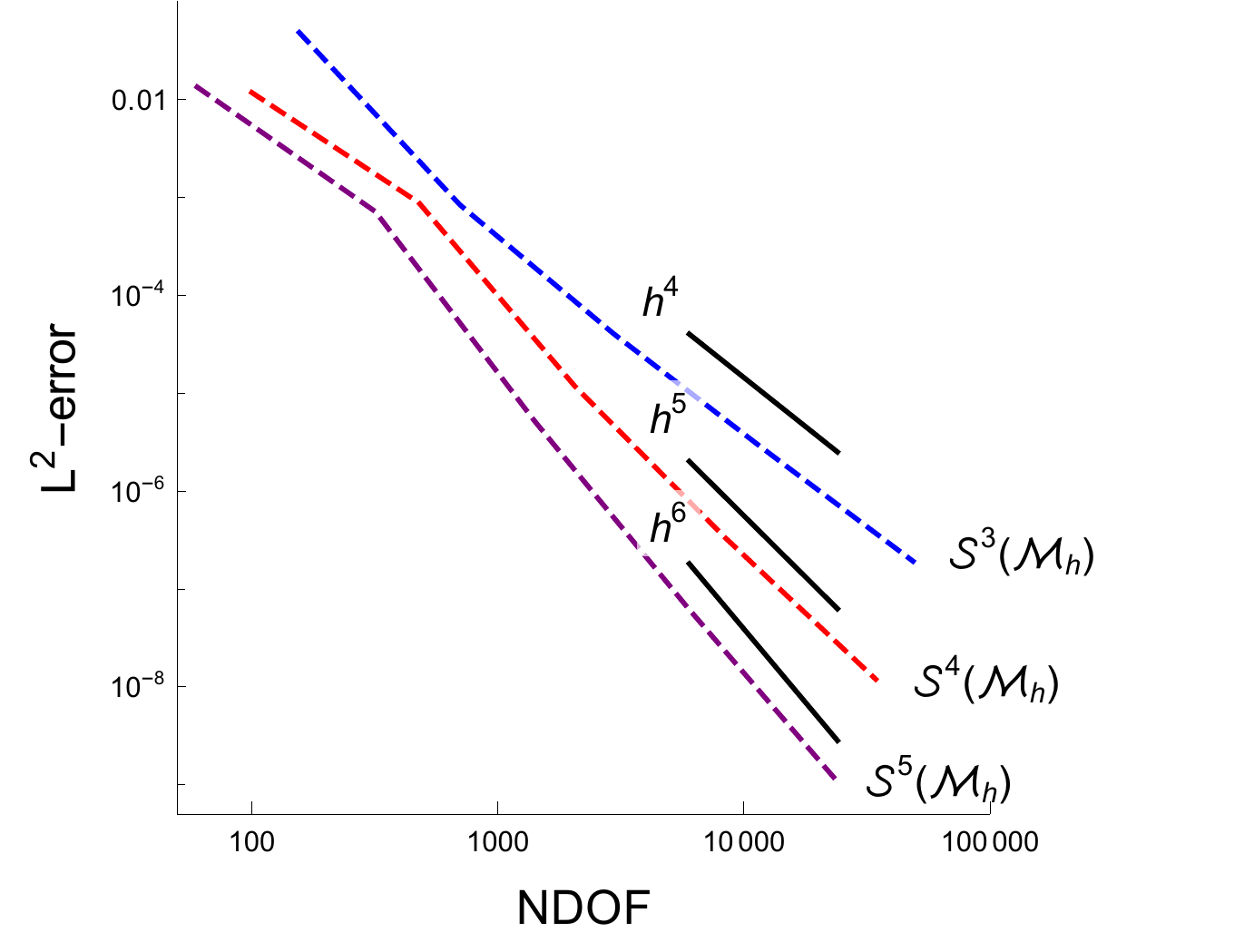} \\
$L^{\infty}$ error & Rel. $L^2$ error \\
\includegraphics[width=.45\textwidth,clip]{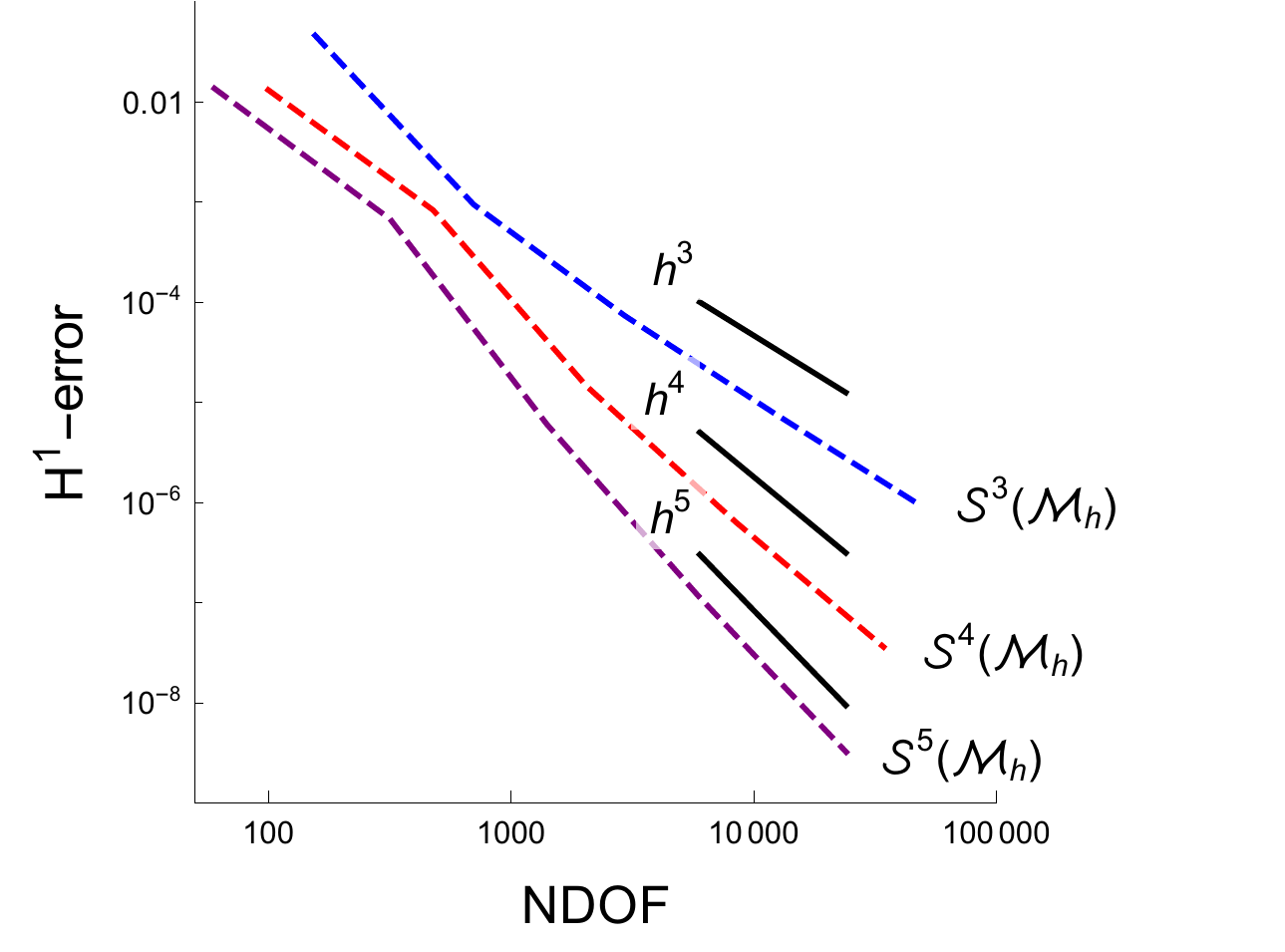} &
\includegraphics[width=.45\textwidth,clip]{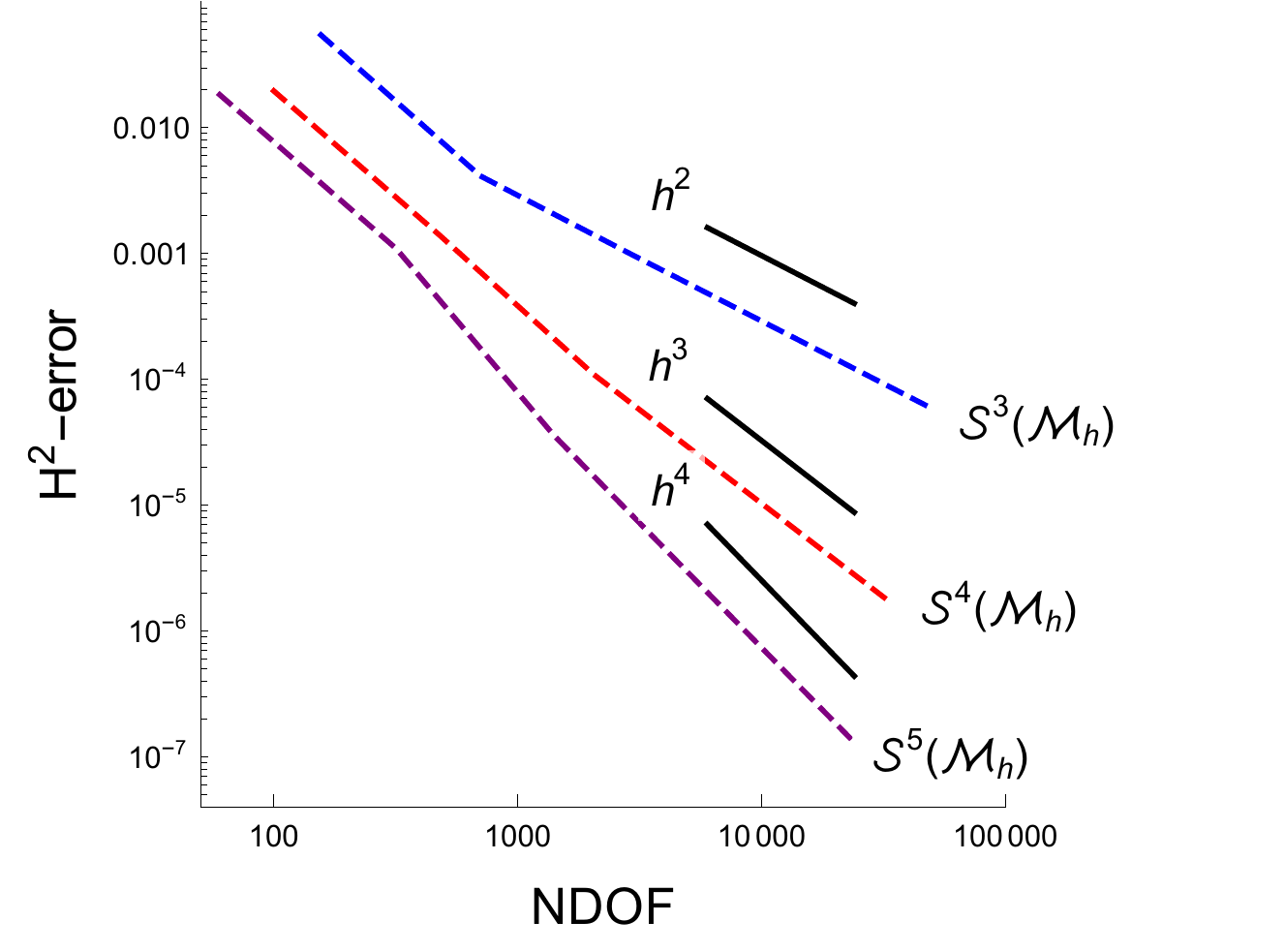} \\
Rel. $H^1$ error & Rel. $H^2$ error
\end{tabular}
\caption{Solving the biharmonic equation~\eqref{eq:problem_biharmonic} on the given quadrilateral mesh~$\Mesh$ (top row, left) for the exact 
solution~\eqref{eq:exact_multiholedomain} (top row, right) with the resulting $L^{\infty}$ and relative $L^{2}$, $H^1$, $H^2$-errors (middle and bottom row). 
See Example~\ref{ex:example_multiholedomain}.}
\label{fig:example_multidomain}
\end{figure}

\begin{figure}
\centering\footnotesize
\begin{tabular}{cc}
\includegraphics[width=.45\textwidth,clip]{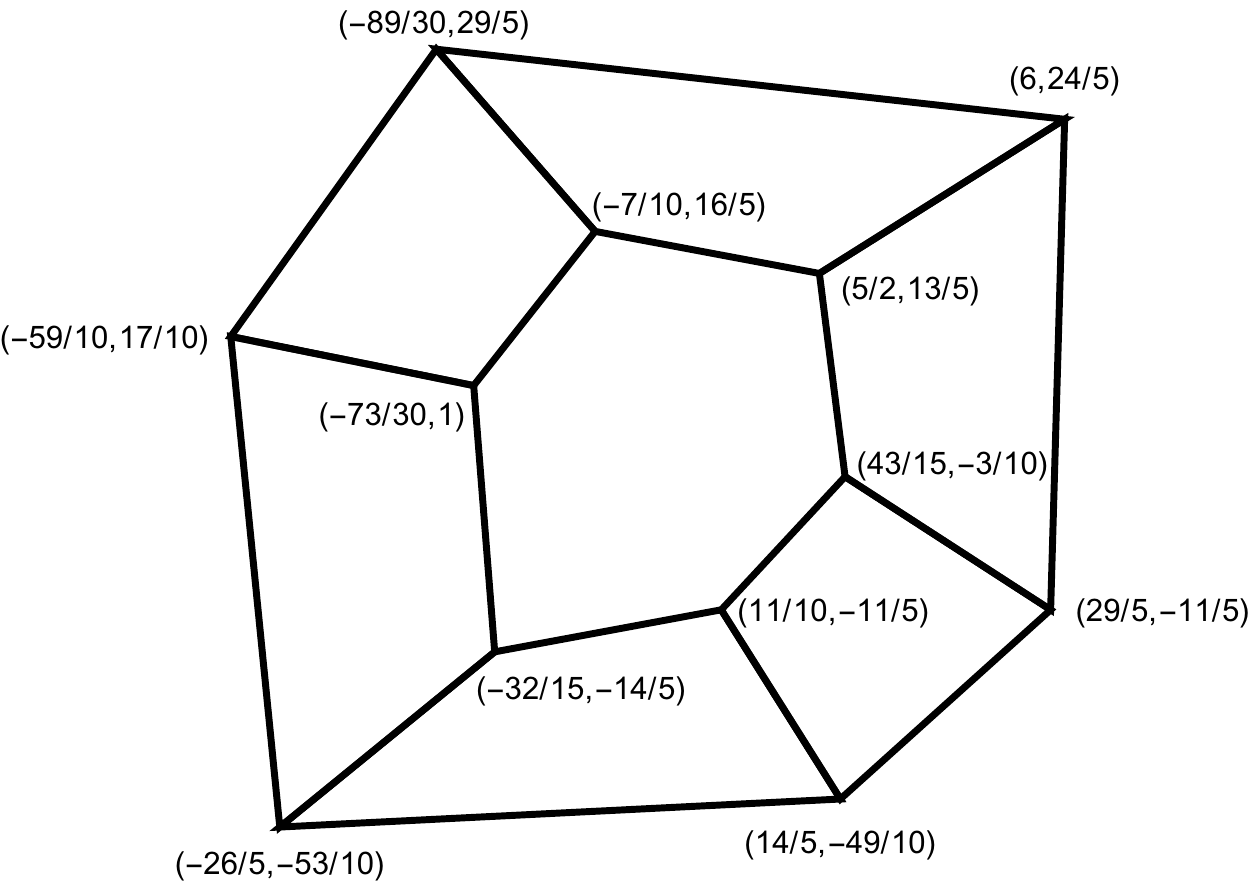} &
\includegraphics[width=.35\textwidth,clip]{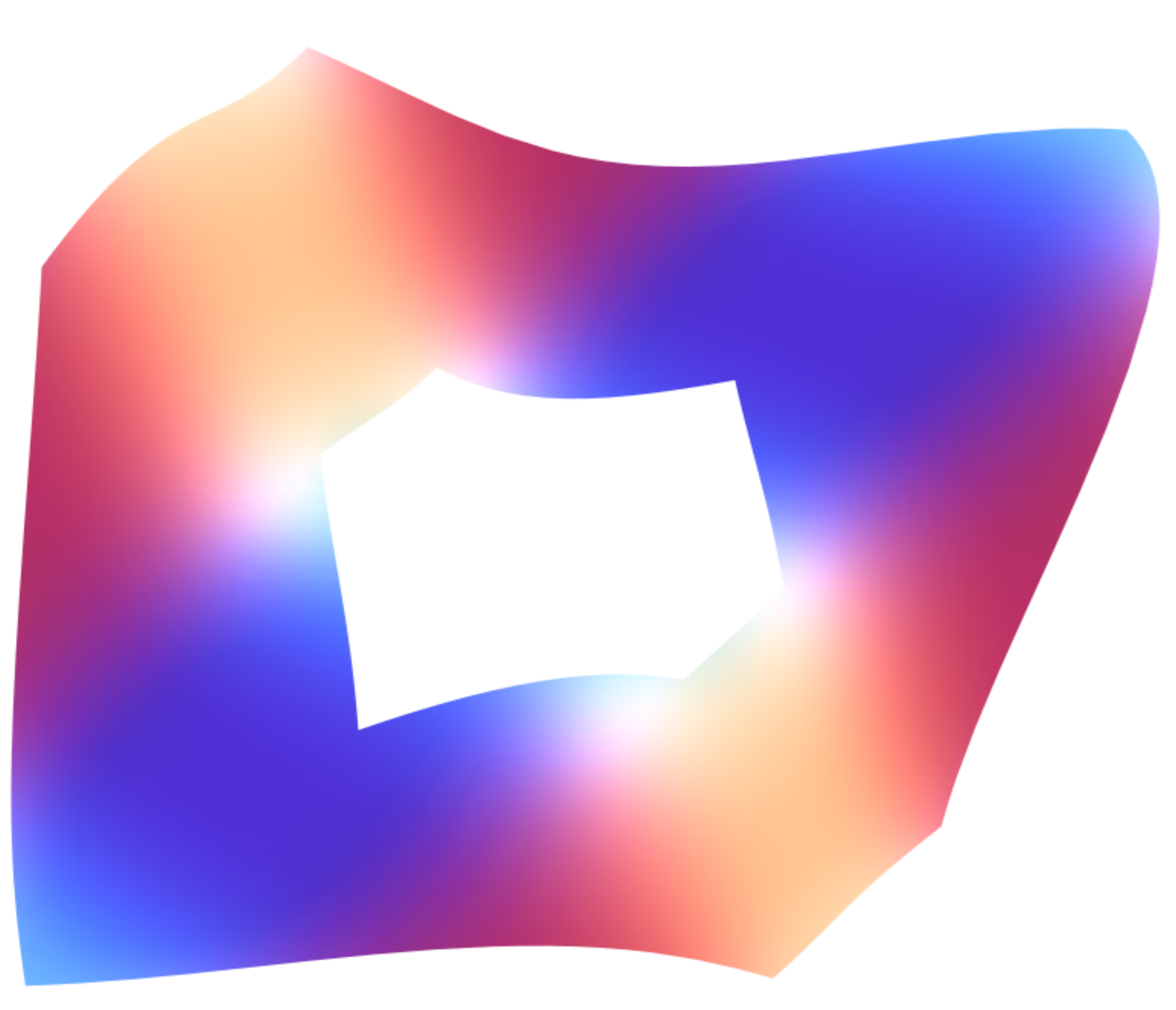} \\
Quadrilateral mesh~$\Mesh$ & Exact solution \\
\includegraphics[width=.45\textwidth,clip]{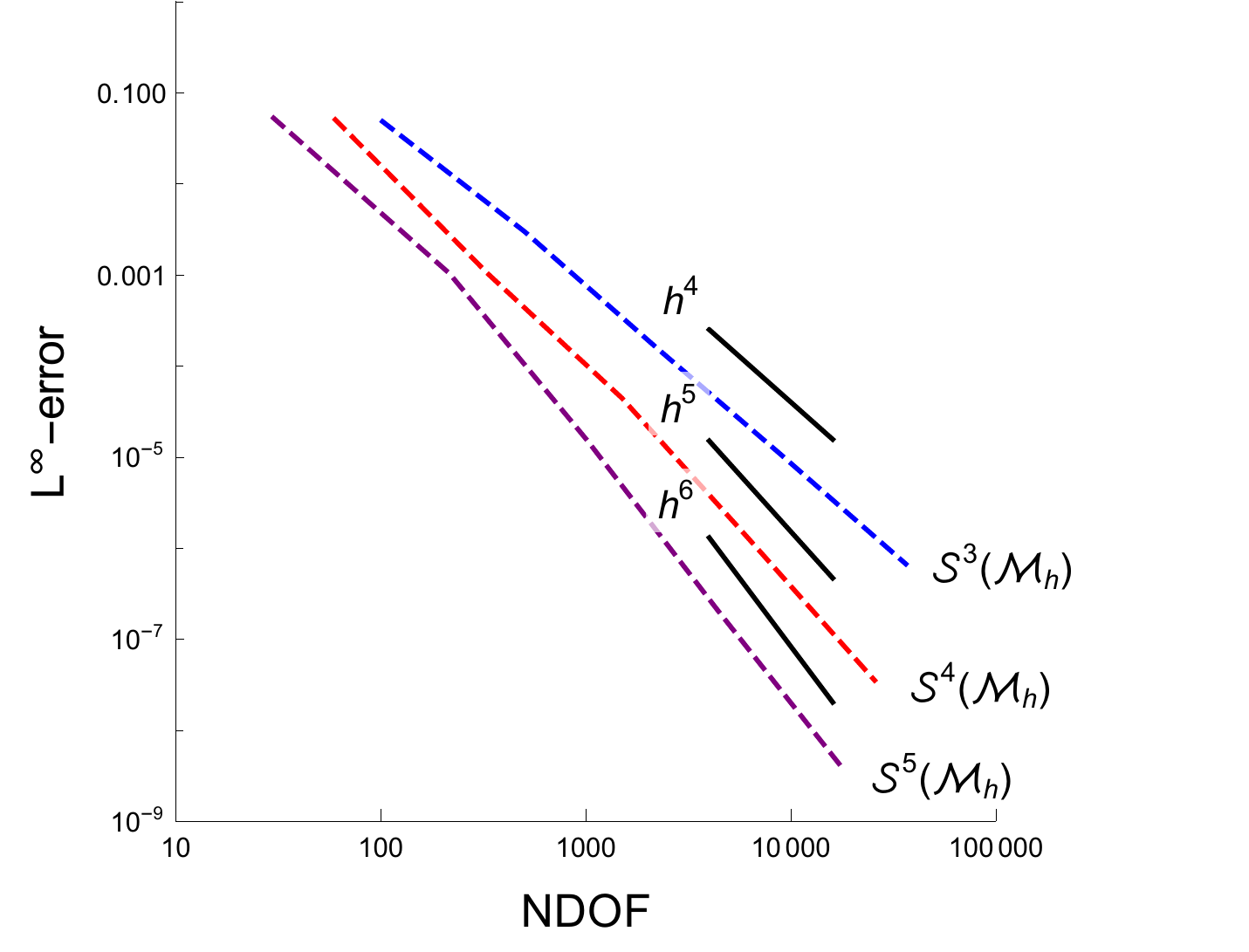} &
\includegraphics[width=.45\textwidth,clip]{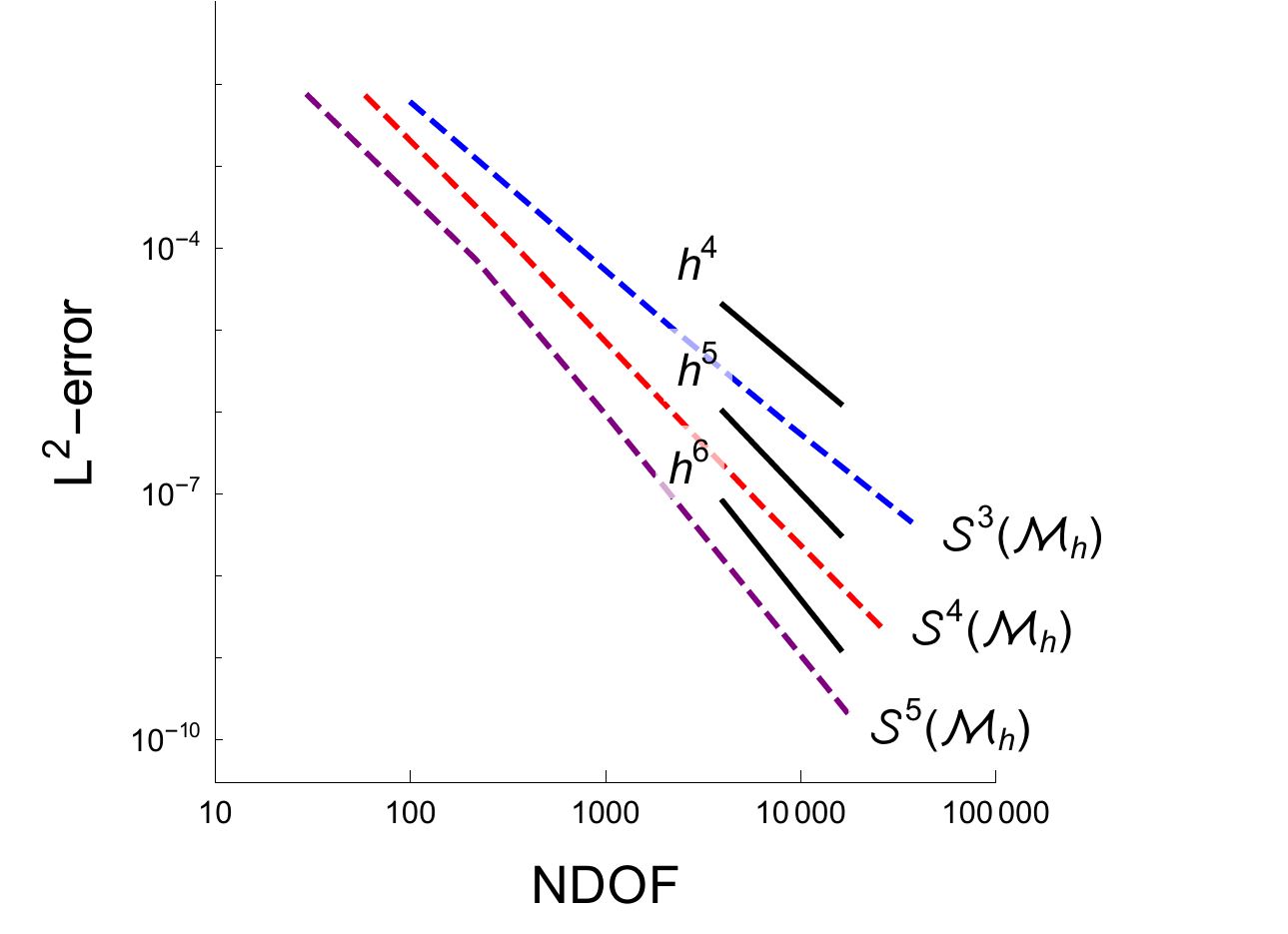} \\
$L^{\infty}$ error & Rel. $L^2$ error \\
\includegraphics[width=.45\textwidth,clip]{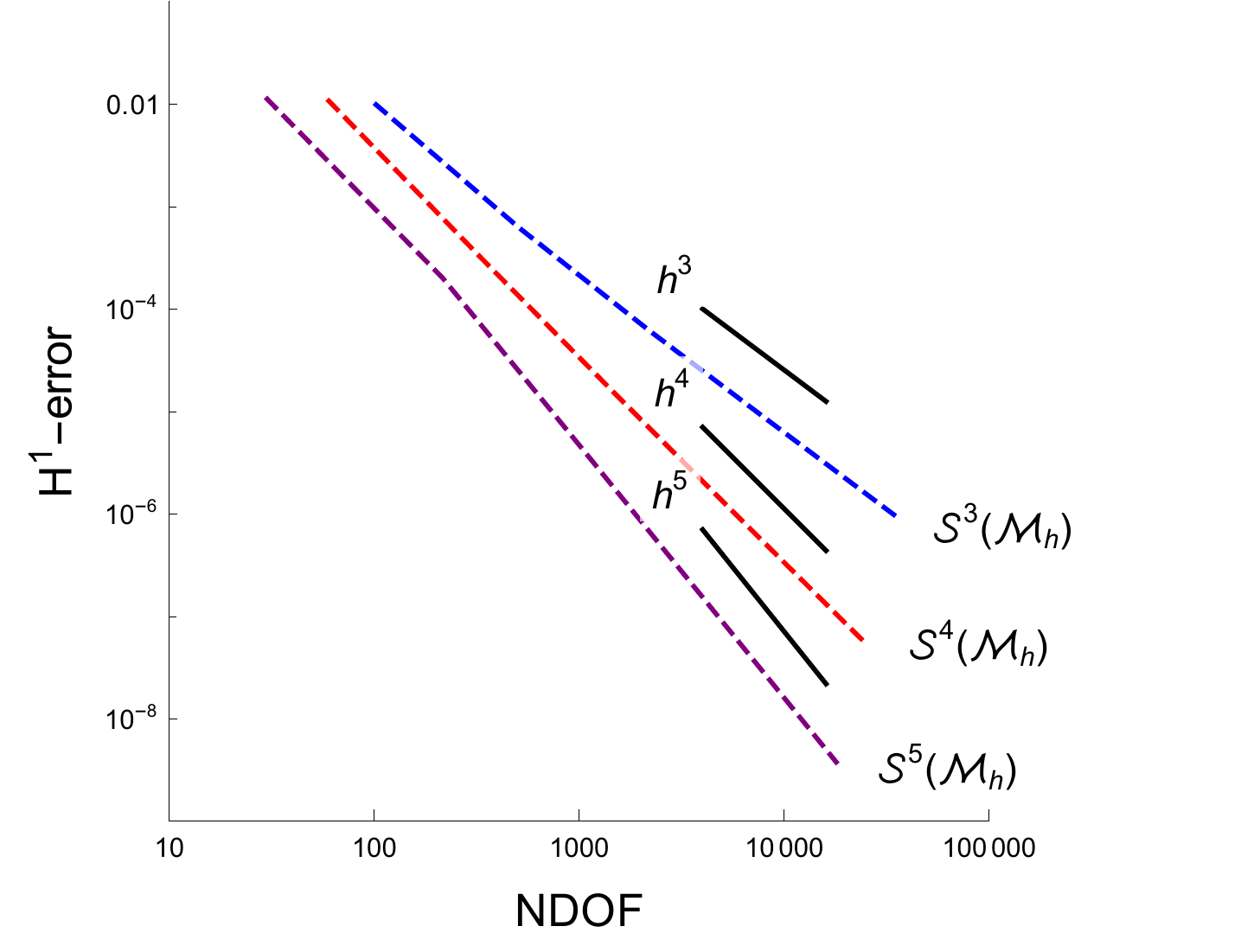} &
\includegraphics[width=.45\textwidth,clip]{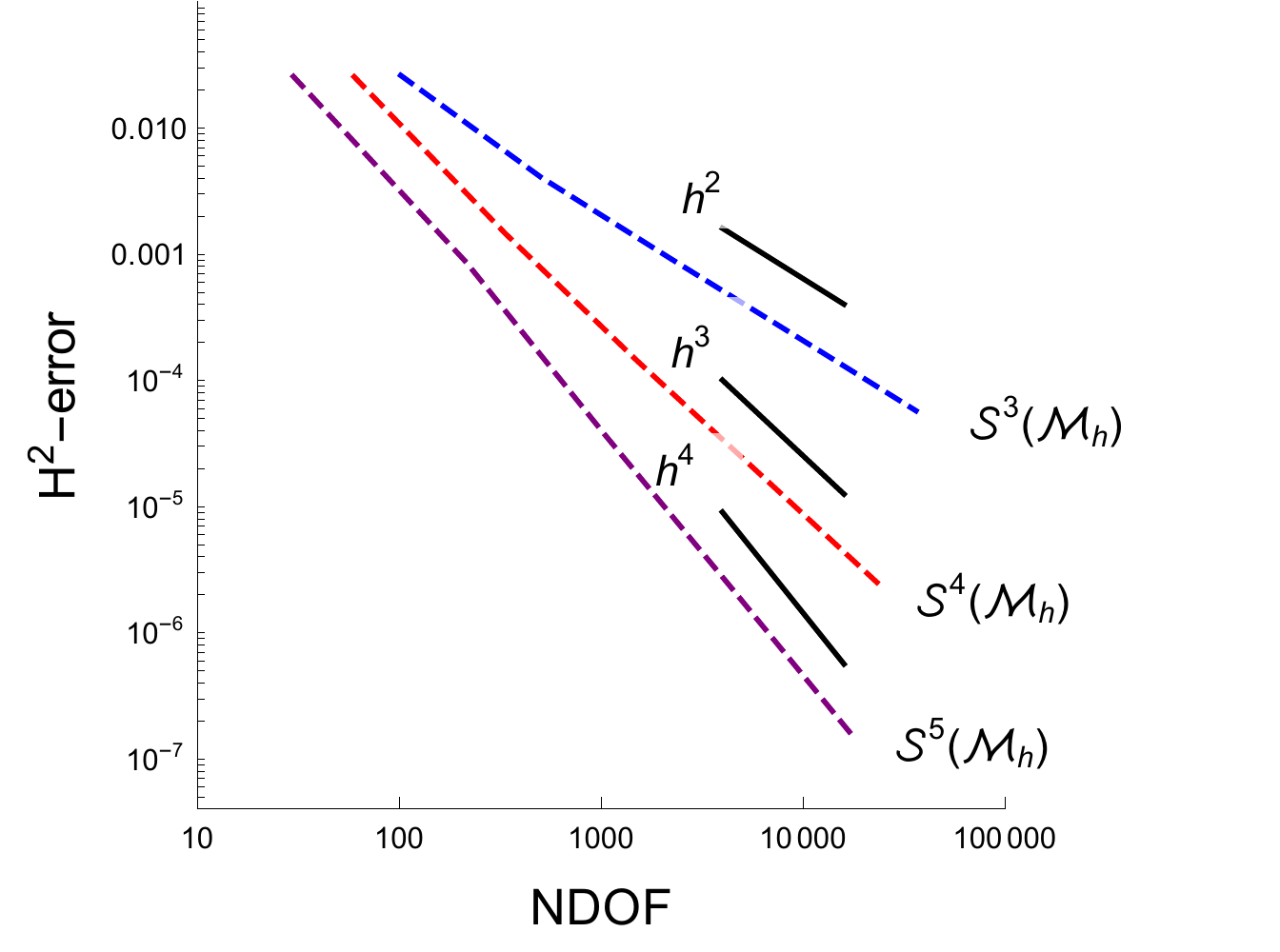} \\
Rel. $H^1$ error & Rel. $H^2$ error
\end{tabular}
\caption{Solving the biharmonic equation~\eqref{eq:problem_biharmonic} on the given quadrilateral mesh~$\Mesh$ (top row, left) for the exact 
solution~\eqref{eq:exact_multiholedomain} (top row, right) with the resulting $L^{\infty}$ and relative $L^{2}$, $H^1$, $H^2$-errors (middle and bottom row). 
See Example~\ref{ex:example_multiholedomain}.}
\label{fig:example_holedomain}
\end{figure}
\end{example}

\begin{example} \label{ex:example_comparison}
We compare the $C^1$ quadrilateral spaces~$\globalspace^p(\Mesh_h)$ for polynomial degrees~$p=3,4,5$ as constructed in this paper with 
the $C^1$ isogeometric spaces~$\mathcal{A}_h$ for the cases $(p,r)=(3,1)$, $(p,r)=(4,2)$ and $(p,r)=(5,3)$ as generated in~\cite{KaSaTa19} by means of standard $h$-refinement. 
For this purpose, we solve the biharmonic equation~\eqref{eq:problem_biharmonic} for the exact solution
\begin{equation} \label{eq:exact_comparison}
 u(x_1,x_2)= - 4 \cos \left(\frac{x_1}{2}\right) \sin \left(\frac{x_2}{2}\right),
\end{equation}
see Fig.~\ref{fig:example_comparison} (top row, right), on the bilinearly parametrized multi-patch domain~$\Omega$ determined by the 
mesh~$\Mesh$ shown in Fig.~\ref{fig:example_comparison} (top row, left). The resulting $L^{\infty}$-error as well as the relative $L^2$, $H^1$ and $H^2$-errors, which are reported in 
Fig.~\ref{fig:example_comparison} (middle and bottom row) with respect to the number of degrees of freedom (NDOF), indicate for all considered degrees $p=3,4,5$ and for both spaces 
$\globalspace^p(\Mesh_h)$ and $\mathcal{A}_h$ convergence rates of optimal order of $\mathcal{O}(h^{p+1})$, $\mathcal{O}(h^{p+1})$, $\mathcal{O}(h^{p})$ and $\mathcal{O}(h^{p-1})$, 
respectively. While the spaces $\globalspace^p(\Mesh_h)$ perform slightly better than the spaces~$\mathcal{A}_h$ for the case~$p=3$, it is in the opposite way around for the 
case~$p=5$. This is not really surprising, since for the case $p=3$, the resulting spaces $\globalspace^p(\Mesh_h)$ are $C^2$ at all vertices $\vertex \in \Vertices_h$, while the 
spaces~$\mathcal{A}_h$ are in general just $C^1$ at the vertices $\vertex \in \Vertices_h \setminus \Vertices$, and since for the case~$p=5$, e.g., the spaces~$\mathcal{A}_h$ are 
$C^3$ at all edges~$\edge \in \Edge_h$, while the spaces~$\globalspace^p(\Mesh_h)$ are in general just $C^1$ there.

 \begin{figure}
\centering\footnotesize
\begin{tabular}{cc}
\includegraphics[width=.4\textwidth,clip]{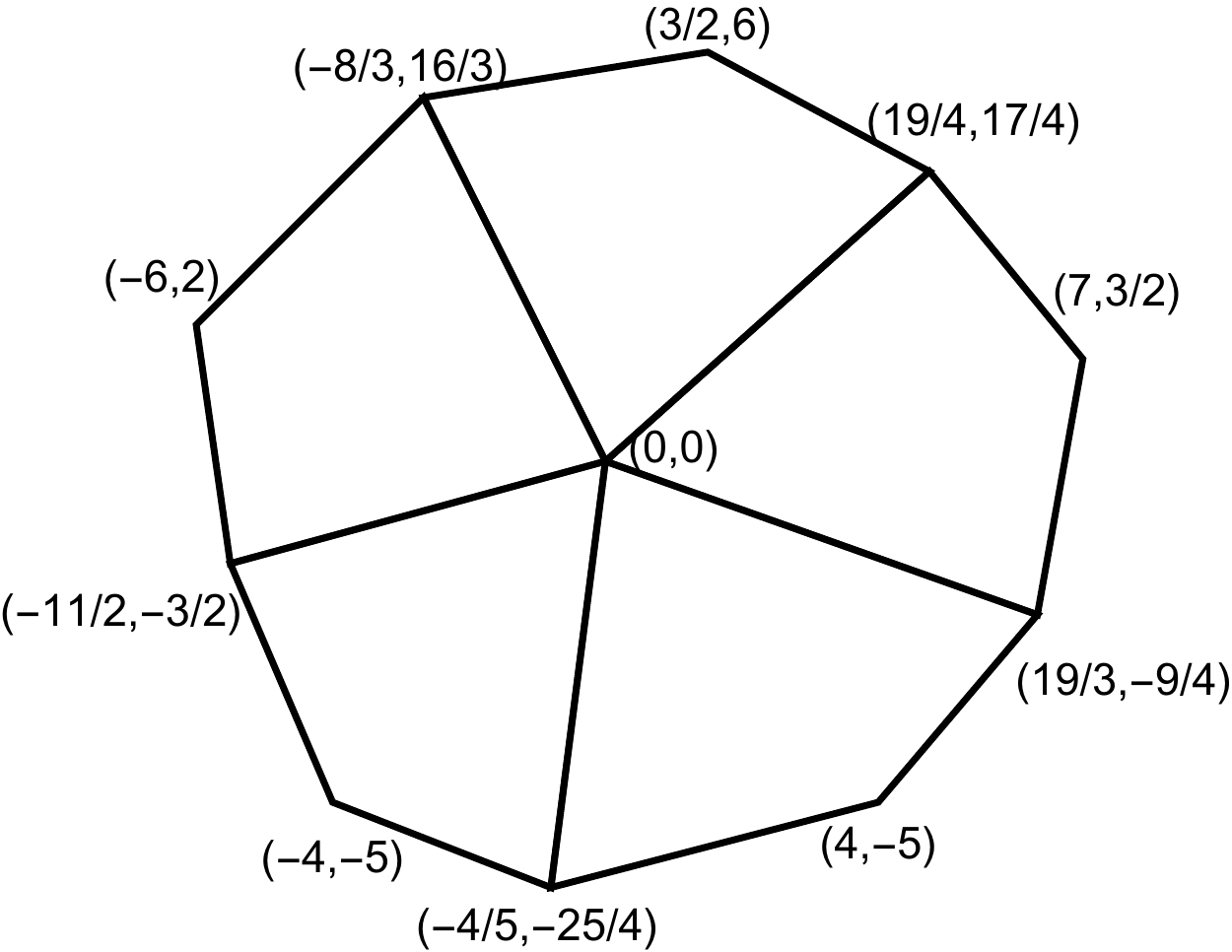} &
\includegraphics[width=.35\textwidth,clip]{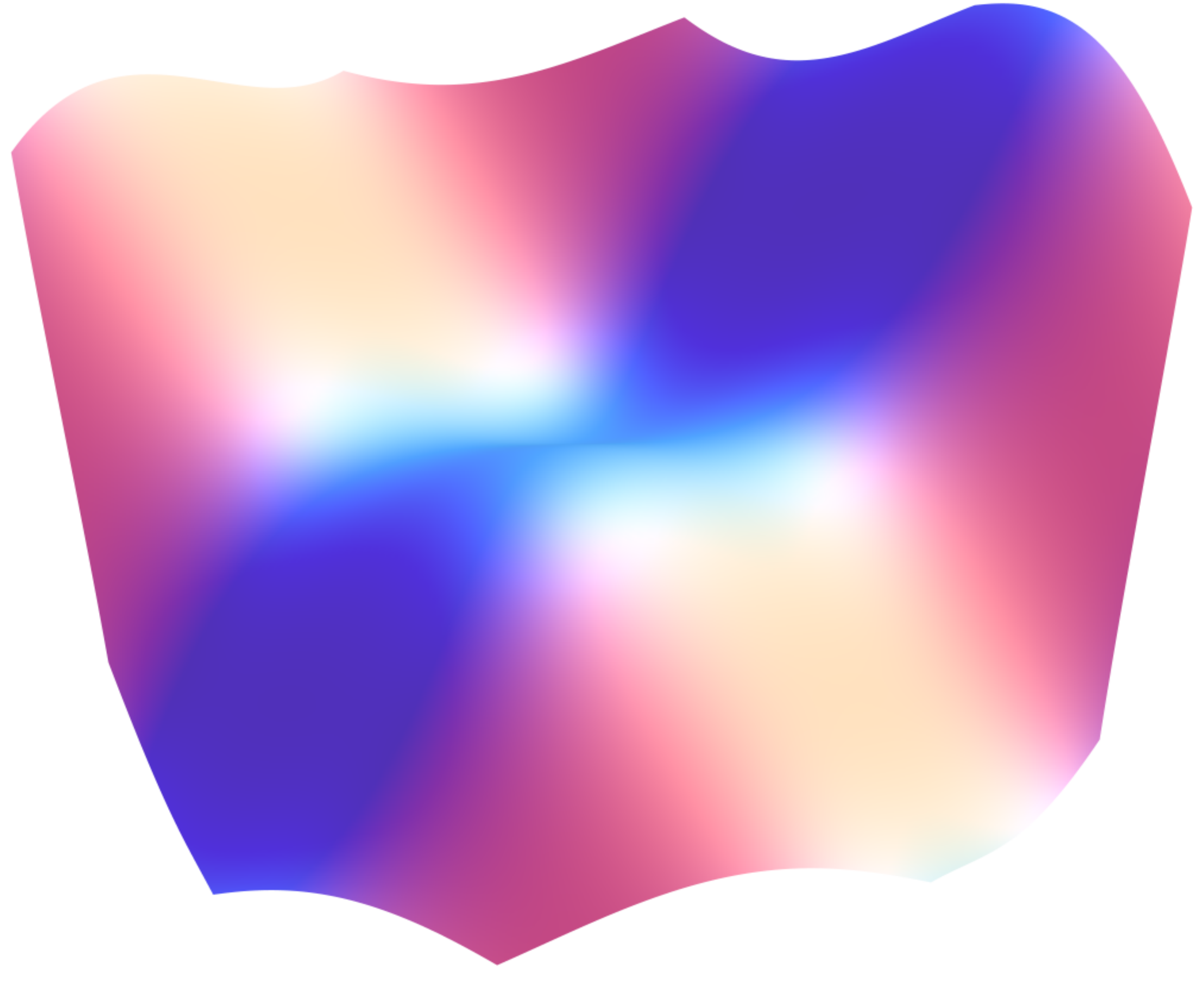} \\
Quadrilateral mesh~$\Mesh$ & Exact solution \\
\includegraphics[width=.45\textwidth,clip]{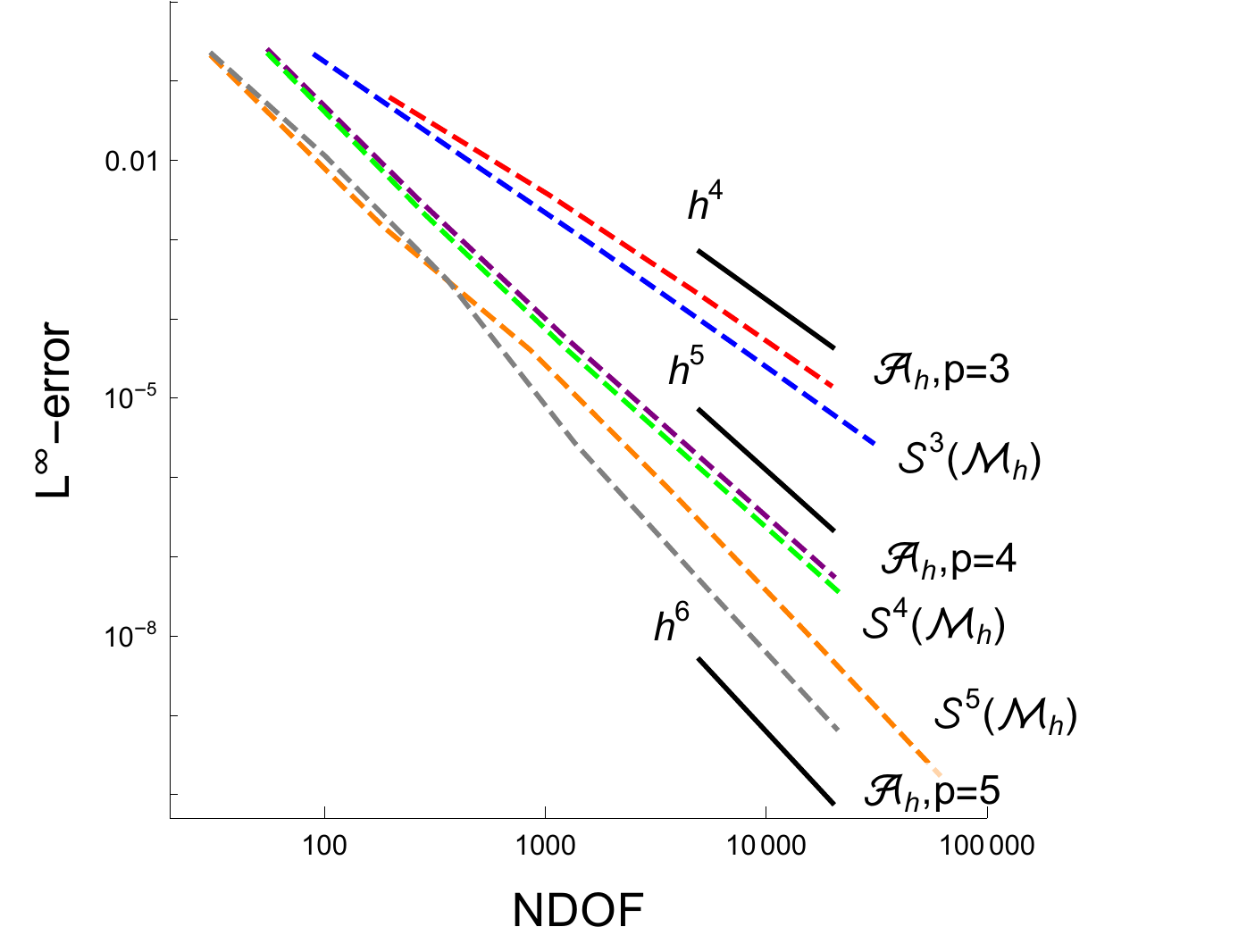} &
\includegraphics[width=.45\textwidth,clip]{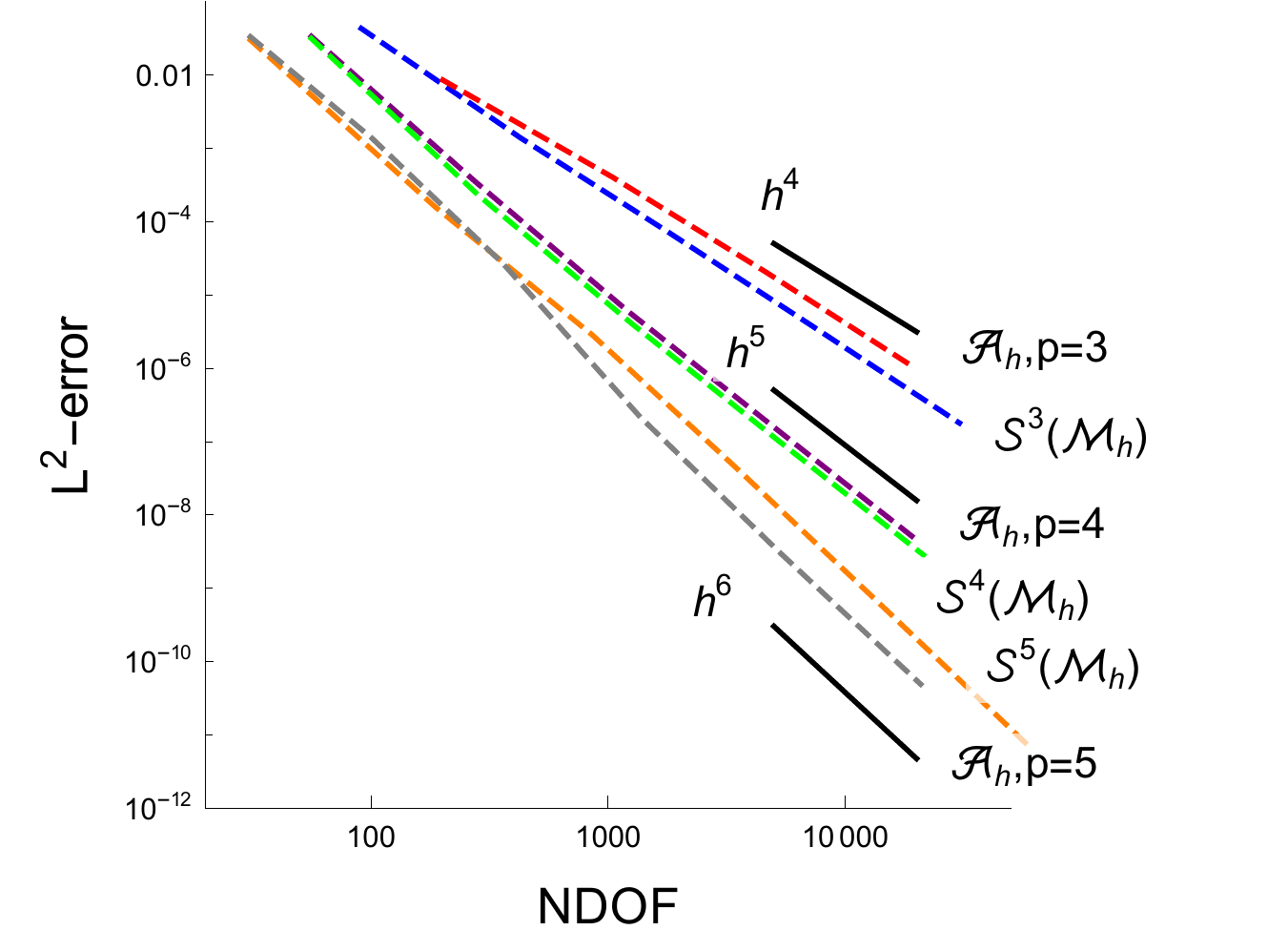} \\
$L^{\infty}$ error & Rel. $L^2$ error \\
\includegraphics[width=.45\textwidth,clip]{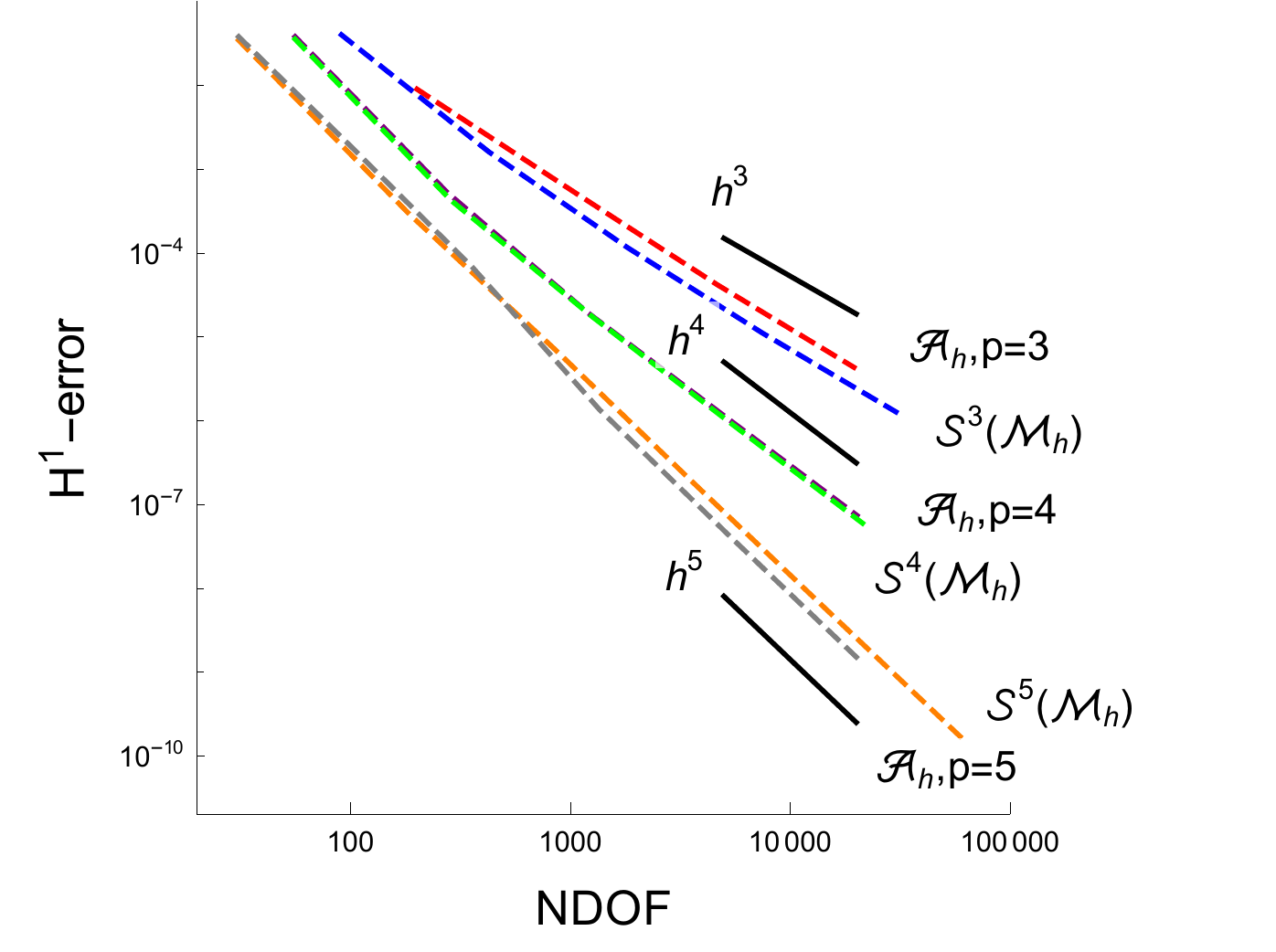} &
\includegraphics[width=.45\textwidth,clip]{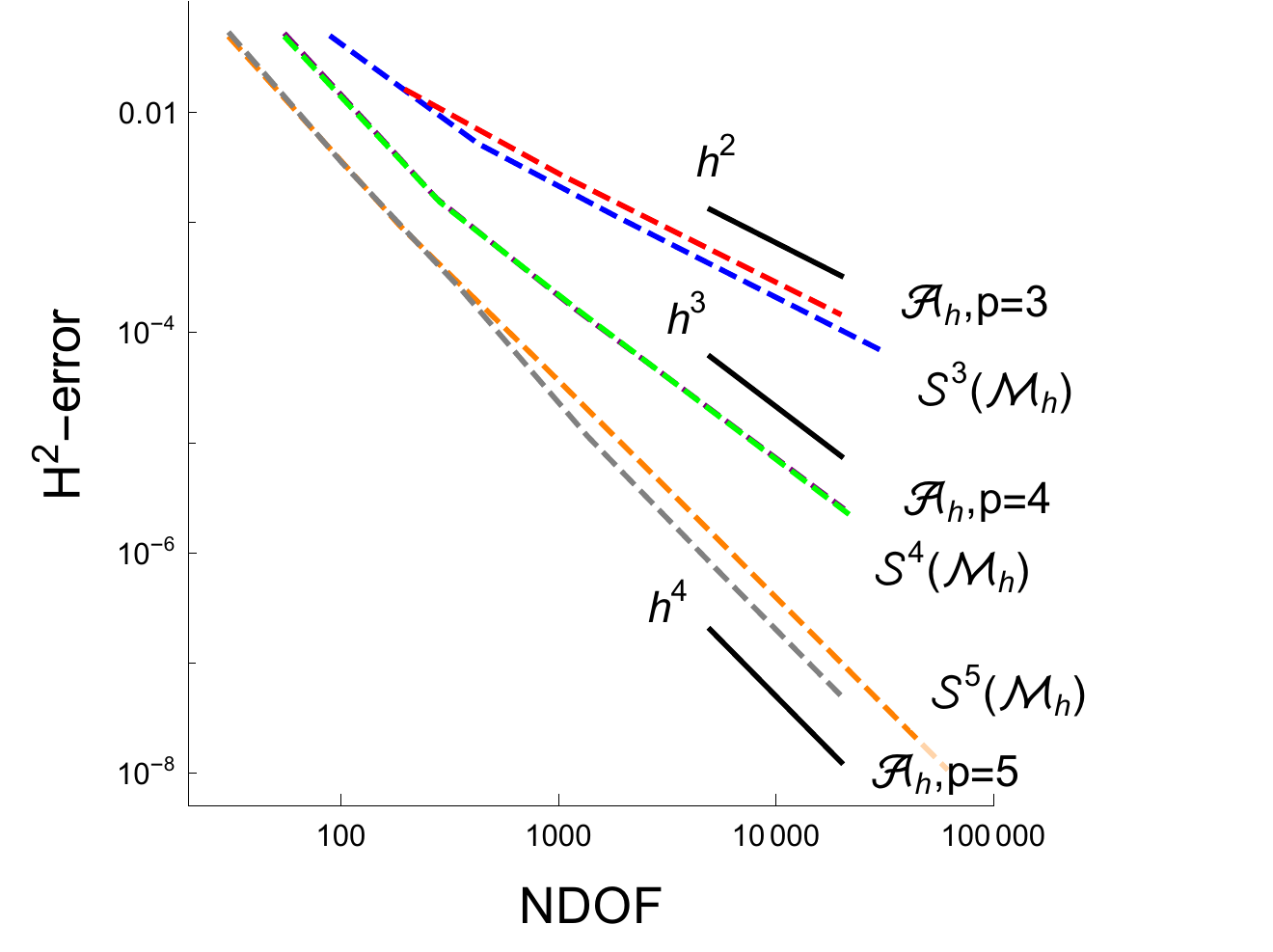} \\
Rel. $H^1$ error & Rel. $H^2$ error
\end{tabular}
\caption{Solving the biharmonic equation~\eqref{eq:problem_biharmonic} on the given quadrilateral mesh~$\Mesh$ (top row, left) for the exact solution~\eqref{eq:exact_comparison} using 
the two $C^1$-smooth spaces~$\globalspace^p(\Mesh_h)$ and $\mathcal{A}_h$ with the resulting $L^{\infty}$ and relative $L^{2}$, $H^1$, $H^2$-errors (middle and bottom row). 
See Example~\ref{ex:example_comparison}.}
\label{fig:example_comparison}
\end{figure}
\end{example}

\begin{example} \label{ex:example_non-nested}
In the previous examples the meshes were always nested, even though the spaces were not. Thus, when refining using a regular split, the elements tend to become closer in shape to parallelograms. This is not necessary for optimal convergence, as can be seen in the present example. Here, we reproduce the mesh refinement presented in~\cite[Figure 1]{ArBoFa02}, for the meshes~$\Mesh_h$ as depicted in Figure~\ref{fig:example_non-nested} (top row), and solve the biharmonic equation for the exact solution
\begin{equation} \label{eq:exact_non-nested}
u(x_1,x_2)= \frac{1}{4}\left(x_1^3 + 5x_2^2 - 10x_2^3 + x_2^4\right)^2
\end{equation}
using the space~$\globalspace^5(\Mesh_h)$ of degree $p=5$. 
In contrast to~\cite{ArBoFa02} we introduce local spaces for every element, hence no uniform space on the parameter domain exists. Thus we are guaranteed to reproduce polynomials of degree $p$, even though not all polynomials of bi-degree $(p,p)$ are present on the parameter domain $\hatquad$. This is different from~\cite{ArBoFa02}, where a fixed space of polynomials on the parameter domain (e.g. for the serendipity element) yields optimal convergence rates on the presented mesh if and only if the space contains all polynomials of bi-degree $(p,p)$.
\begin{figure}
\centering\footnotesize
\begin{tabular}{cc}
\includegraphics[width=.35\textwidth,clip]{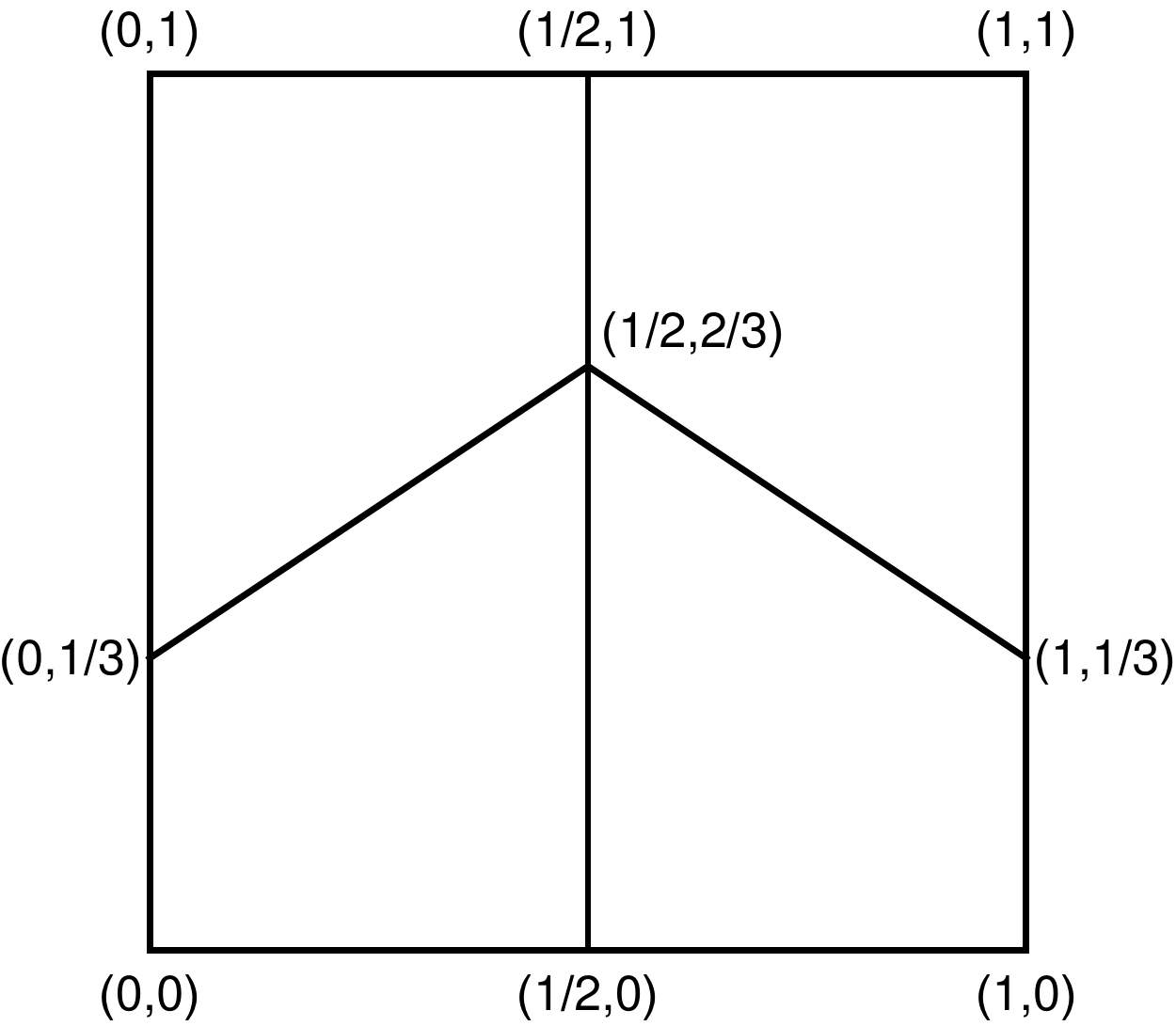} &
\includegraphics[width=.3\textwidth,clip]{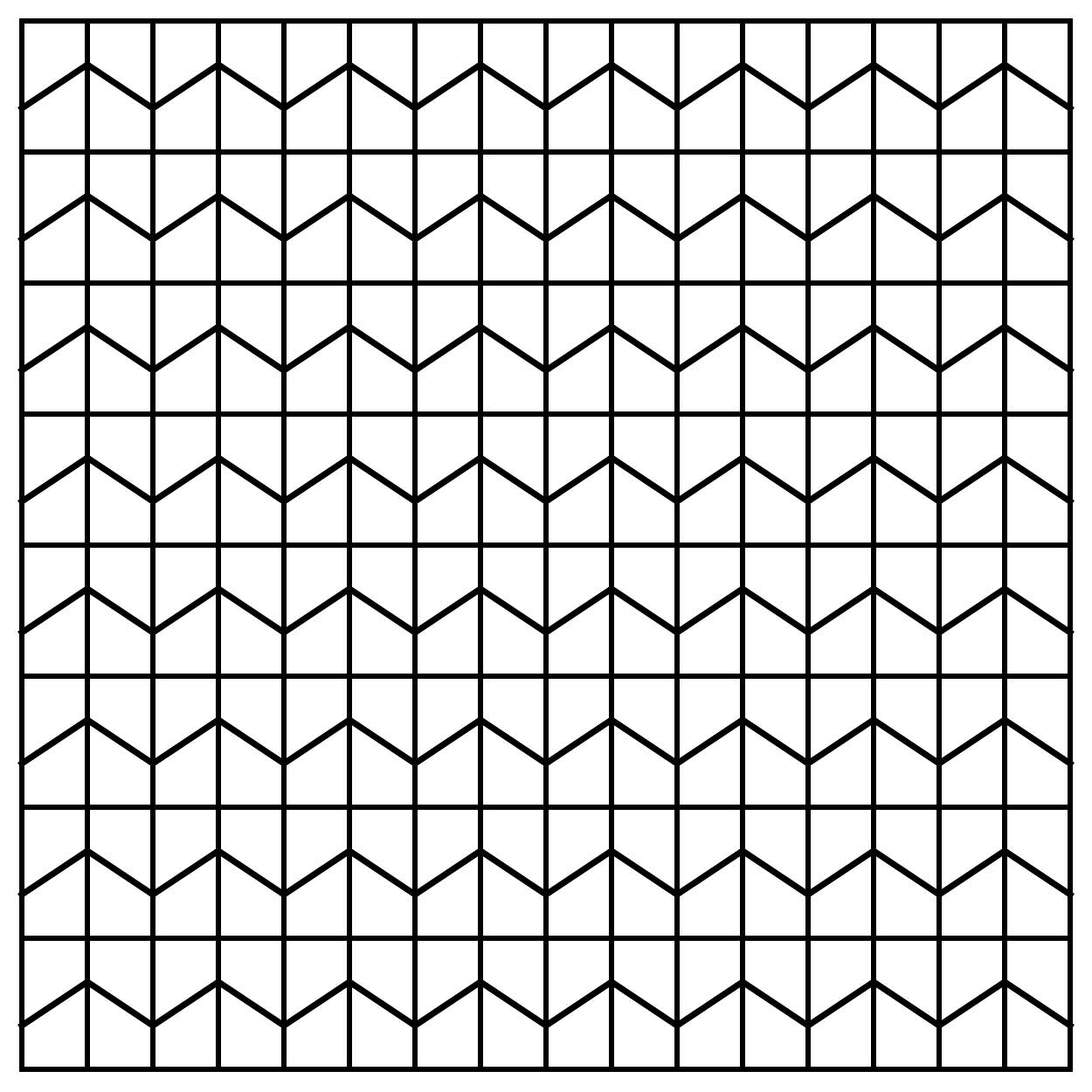} \\
Trapezoid mesh (level $0$) & Trapezoid mesh (level $3$)
\end{tabular}
\includegraphics[width=.45\textwidth,clip]{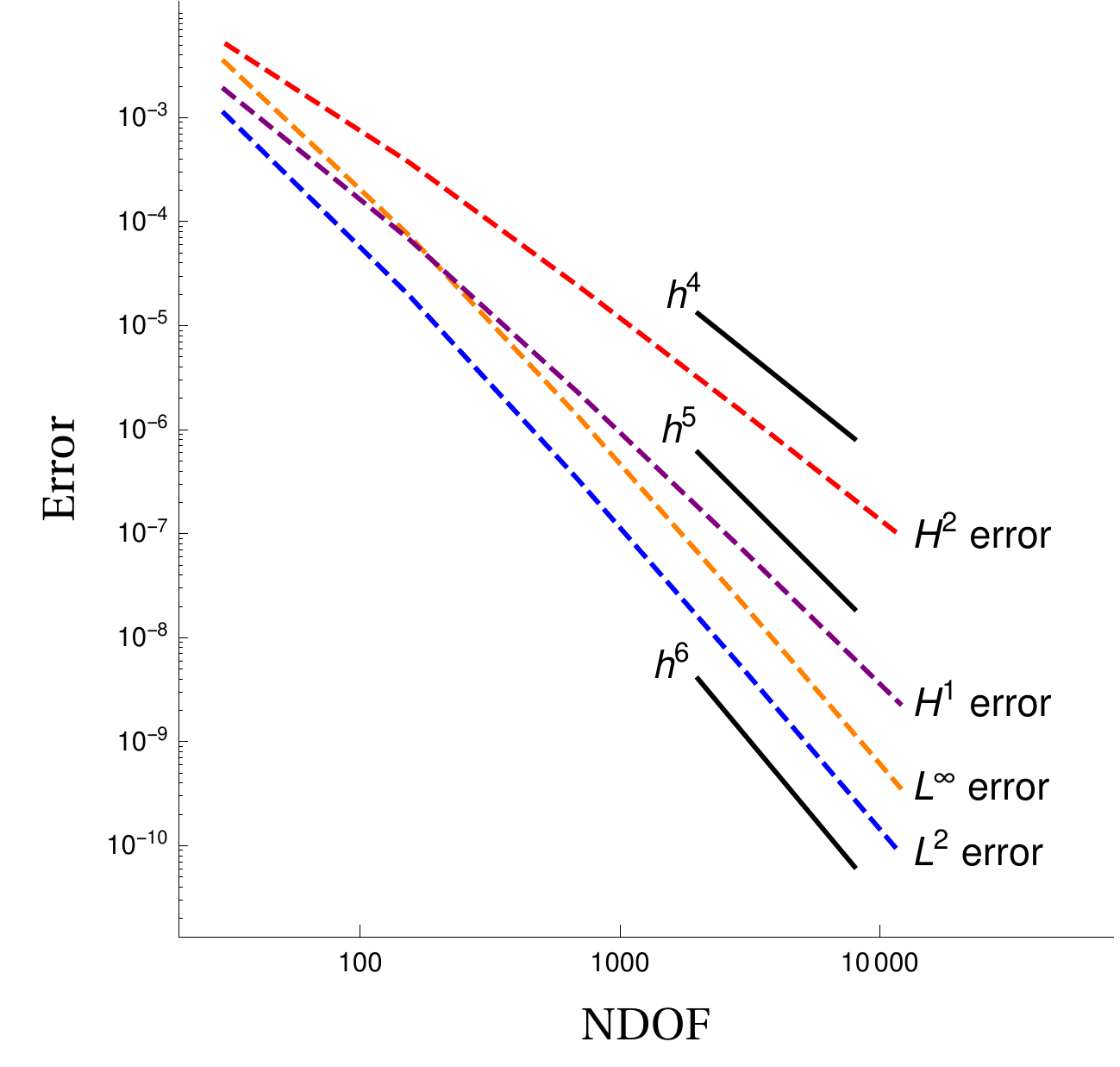} 
\caption{Solving the biharmonic equation~\eqref{eq:problem_biharmonic} on the given, non-nested quadrilateral meshes~$\Mesh_h$ (top row) for the exact solution~\eqref{eq:exact_non-nested} using the $C^1$-smooth space~$\globalspace^5(\Mesh_h)$ with the resulting $L^{\infty}$ and relative $L^{2}$, $H^1$, $H^2$-errors (bottom). 
See Example~\ref{ex:example_non-nested}.}
\label{fig:example_non-nested}
\end{figure}
\end{example}

\begin{example} \label{ex:example_trianglequad}
The goal is to compare the BS  quadrilateral with the Argyris triangle of degree~$p=5$, comparing the spaces~$\globalspace^5(\Mesh_h)$ and~$\globalspace^5(\mathcal{T}_h)$, respectively, where $\mathcal{T}_h$ is the resulting refined triangular mesh obtained via splitting each triangle in a regular way into four sub-triangles. 
For this, we solve the biharmonic equation~\eqref{eq:problem_biharmonic} on two different computational domains, where the corresponding quadrilateral and triangular meshes are 
given in the top rows of Fig.~\ref{fig:example_trianglequad1} and \ref{fig:example_trianglequad2}. In our examples, the quadrilateral and triangular meshes possess in each case the same 
vertices. The considered exact solution is on the one hand 
\begin{equation} \label{eq:exact_trianglequad}
u(x_1,x_2)= 200 \left(x_1 x_2 (1-x_1)(1-x_2)\right)^2
\end{equation}
for the computational domain from Fig.~\ref{fig:example_trianglequad1} (top row, right), and on the other hand
\begin{equation} \label{eq:exact_trianglequad2}
\begin{array}{ll}
u(x_1,x_2)=&\frac{1}{10^7} \left(\left(\frac{13}{5}-x_2\right)\left(\frac{26}{5} + \frac{26 x_1}{15} - x_2\right)\left(\frac{26}{5}+\frac{26x_1}{15}+x_2\right)\right.
\\&
\left.
\left(\frac{13}{5}+x_2\right)\left(\frac{26}{5}-\frac{26x_1}{15}+x_2\right)
\left(\frac{26}{5}-\frac{26x_1}{15}-x_2\right)\right)^2.
\end{array}
\end{equation}
for the computational domain from Fig.~\ref{fig:example_trianglequad2} (top row, right), and fulfills for both cases homogeneous boundary conditions of order~$1$. While in 
Fig.~\ref{fig:example_trianglequad1} the more regular configuration is used for the quadrilateral mesh compared to the triangular one, it is in the opposite way around for the 
meshes in Fig.~\ref{fig:example_trianglequad2}. The numerical results, which are shown in the middle and bottom rows of Fig.~\ref{fig:example_trianglequad1} and 
Fig.~\ref{fig:example_trianglequad2}, and which are compared with respect to the number of degrees of freedom (NDOF), indicate that the BS  quadrilateral 
spaces $\globalspace^5(\Mesh_h)$ perform significantly better than the Argyris triangle spaces~$\globalspace^5(\mathcal{T}_h)$ for the more ``quad-regular'' case 
(cf. Fig.~\ref{fig:example_trianglequad1}) and just slightly worse for the more ``triangle-regular'' case (cf. Fig.~\ref{fig:example_trianglequad2}). However, the rates are not affected, as in all 
considered instances, the resulting $L^{\infty}$-error as well as the relative $L^2$, $H^1$ and $H^2$-errors decrease with optimal order 
of $\mathcal{O}(h^6)$, $\mathcal{O}(h^6)$, $\mathcal{O}(h^5)$ and $\mathcal{O}(h^4)$, respectively.    
 
\begin{figure}
\centering\footnotesize
\begin{tabular}{ccc}
\includegraphics[width=.3\textwidth,clip]{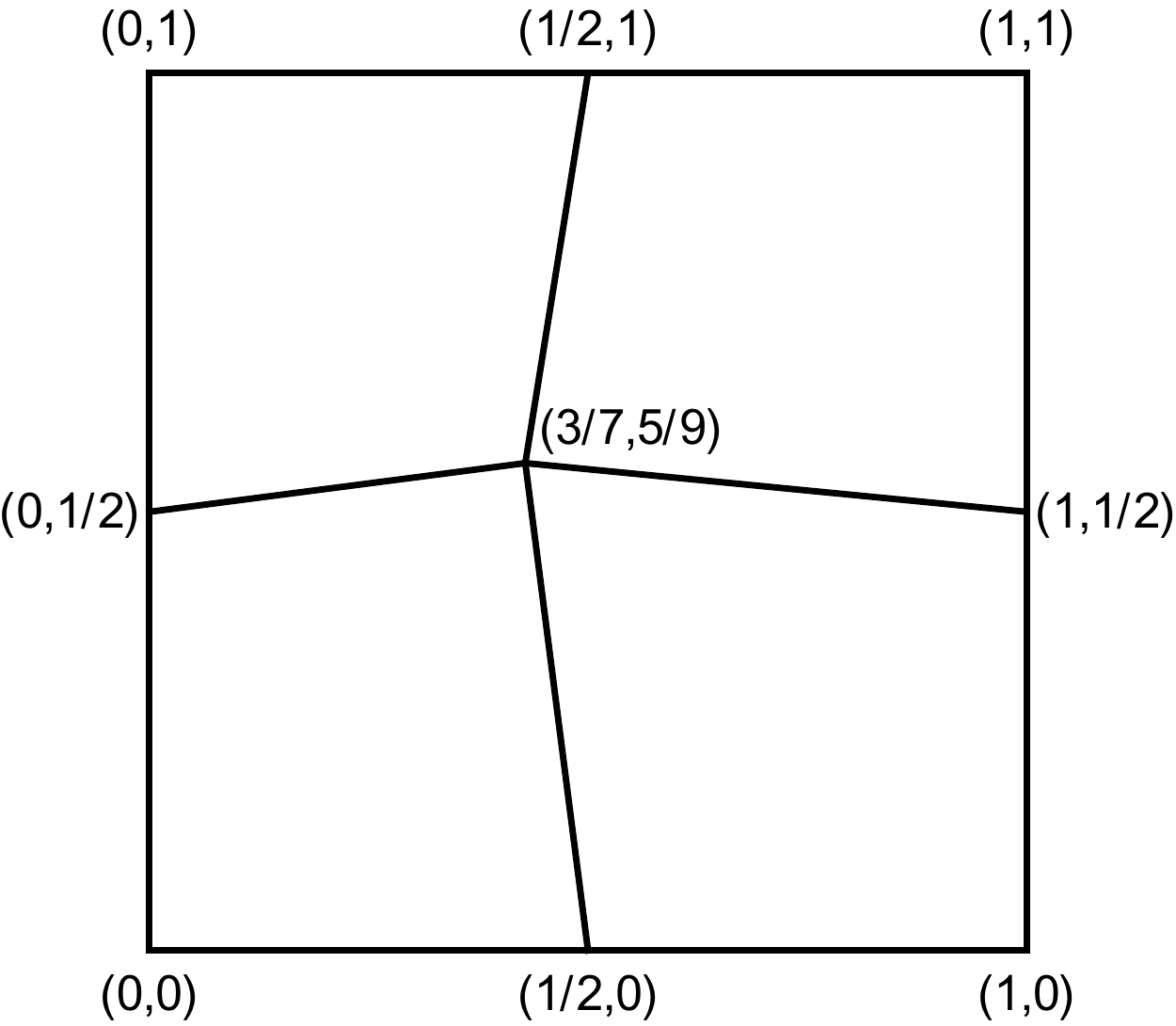} &
\includegraphics[width=.3\textwidth,clip]{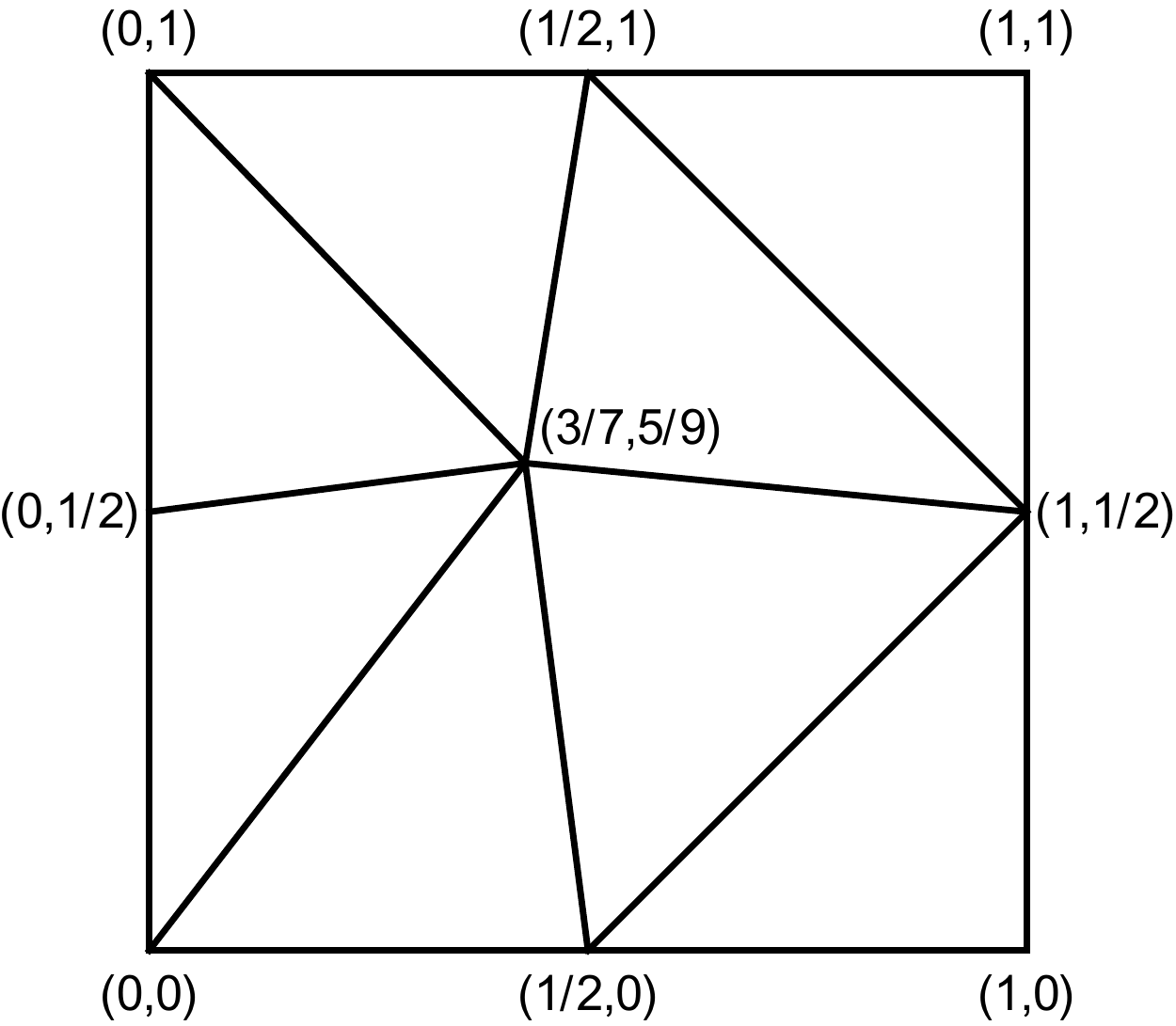} &
\includegraphics[width=.3\textwidth,clip]{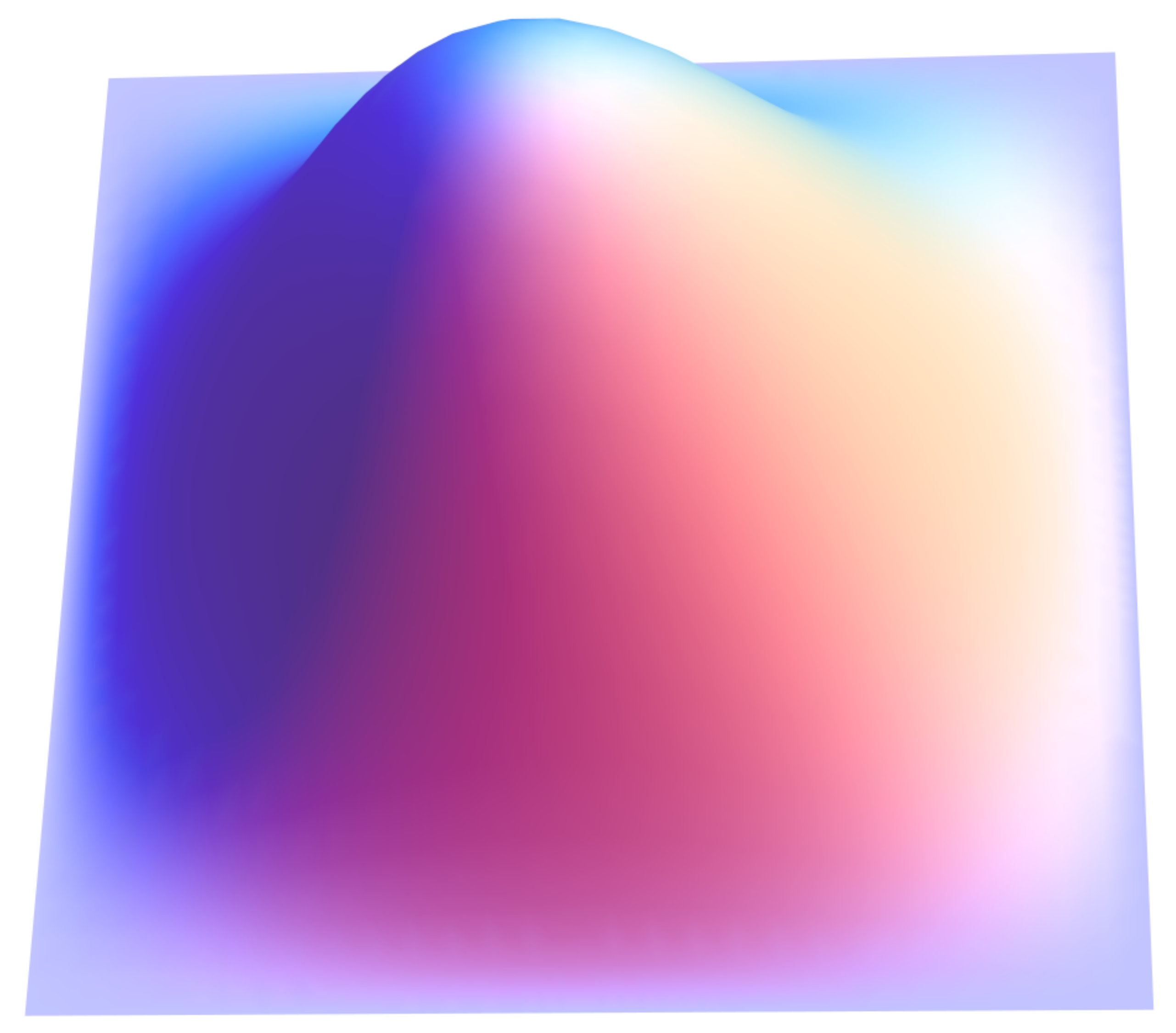} \\
Quadrilateral mesh & Triangular mesh & Exact solution
\end{tabular}
\begin{tabular}{cc}
\includegraphics[width=.45\textwidth,clip]{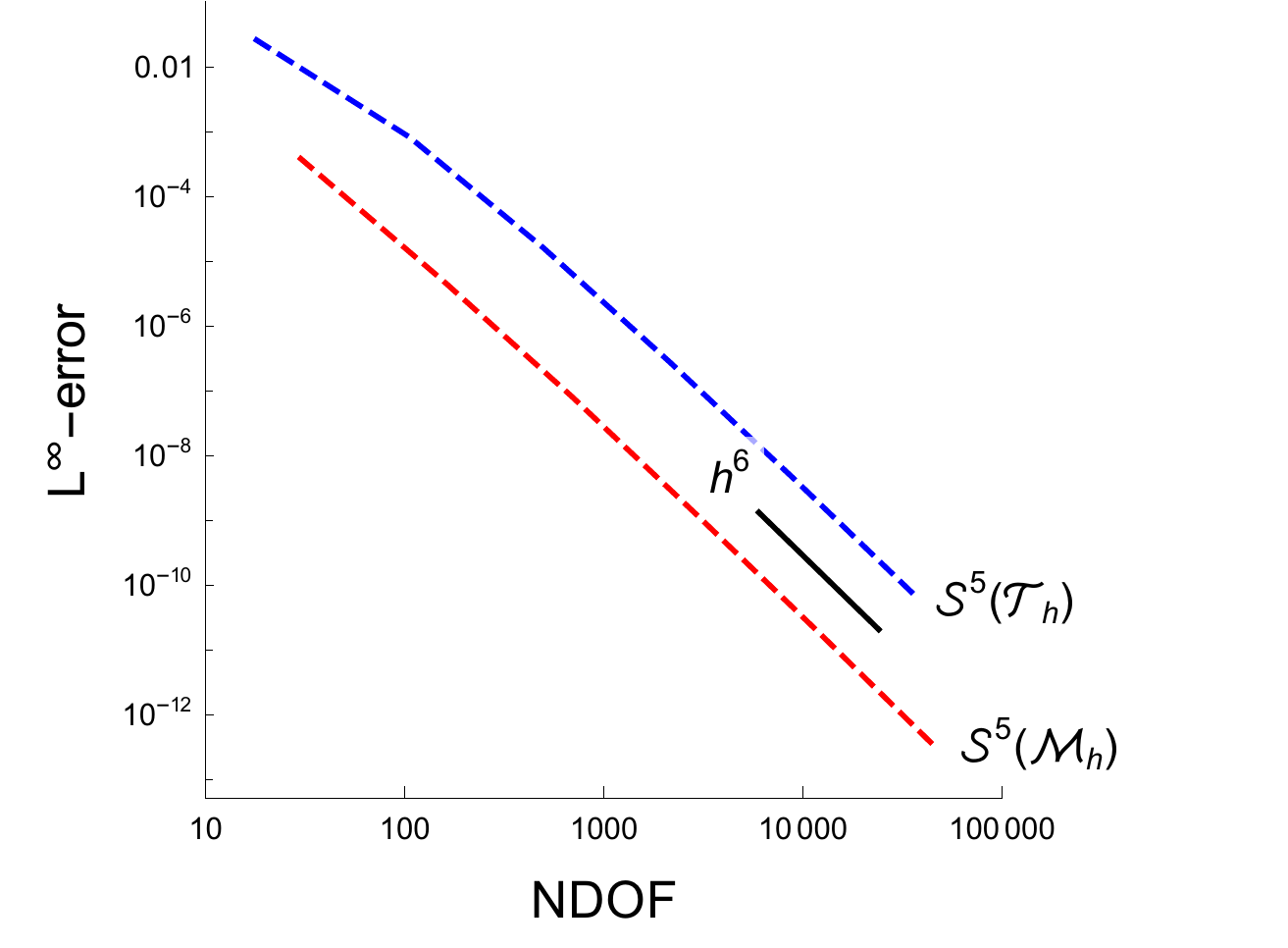} &
\includegraphics[width=.45\textwidth,clip]{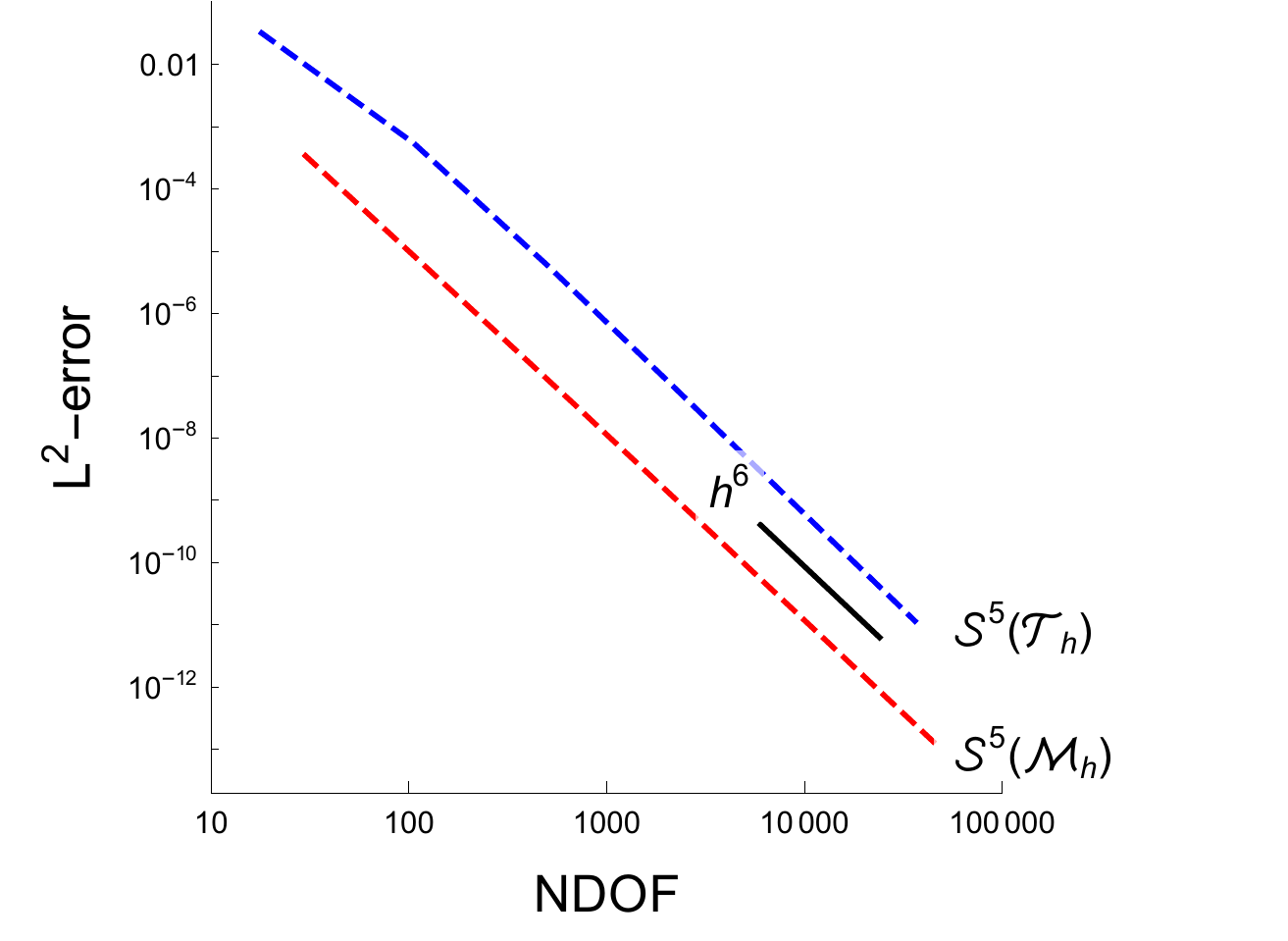} \\
$L^{\infty}$ error & Rel. $L^2$ error \\
\includegraphics[width=.45\textwidth,clip]{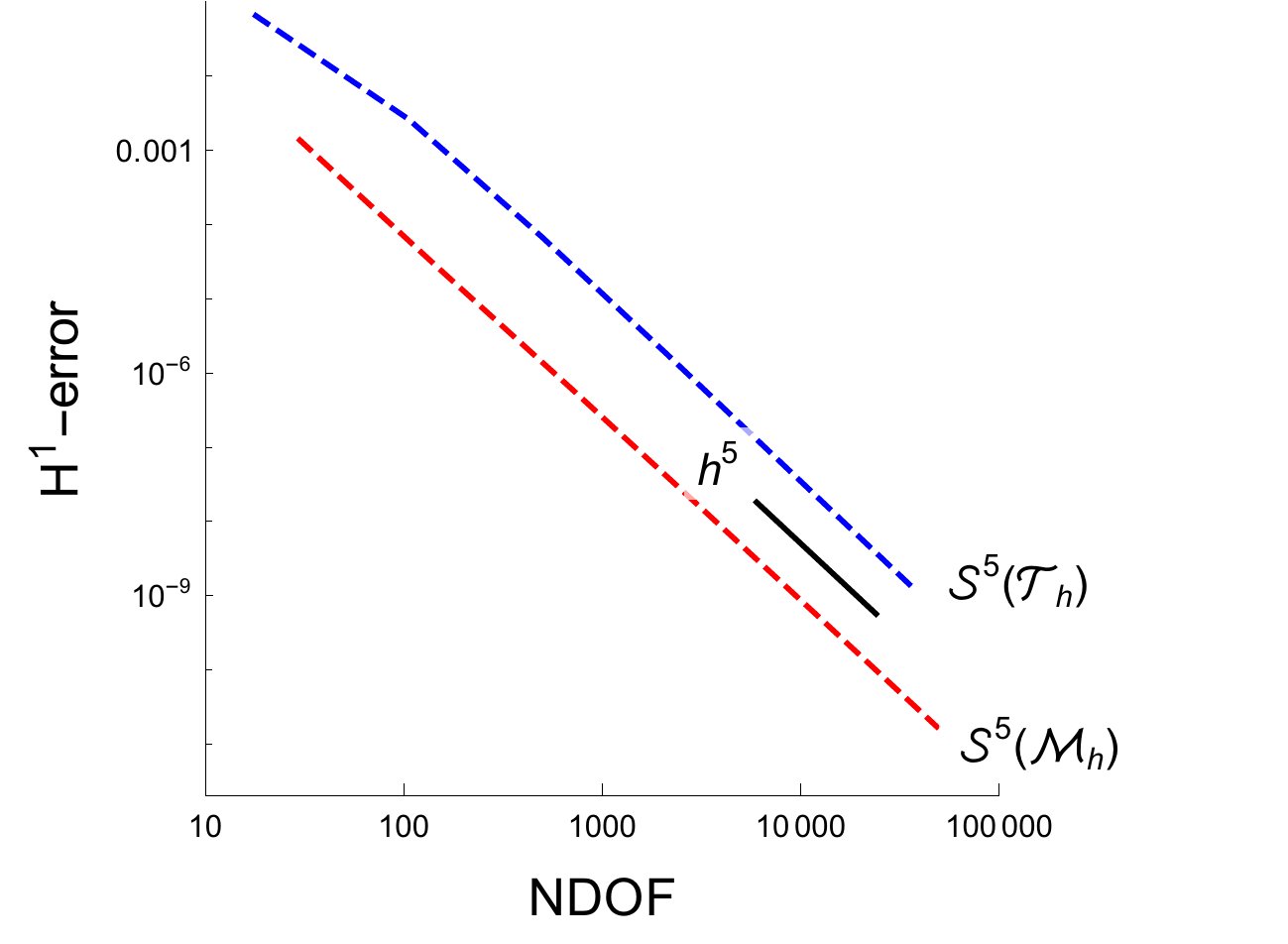} &
\includegraphics[width=.45\textwidth,clip]{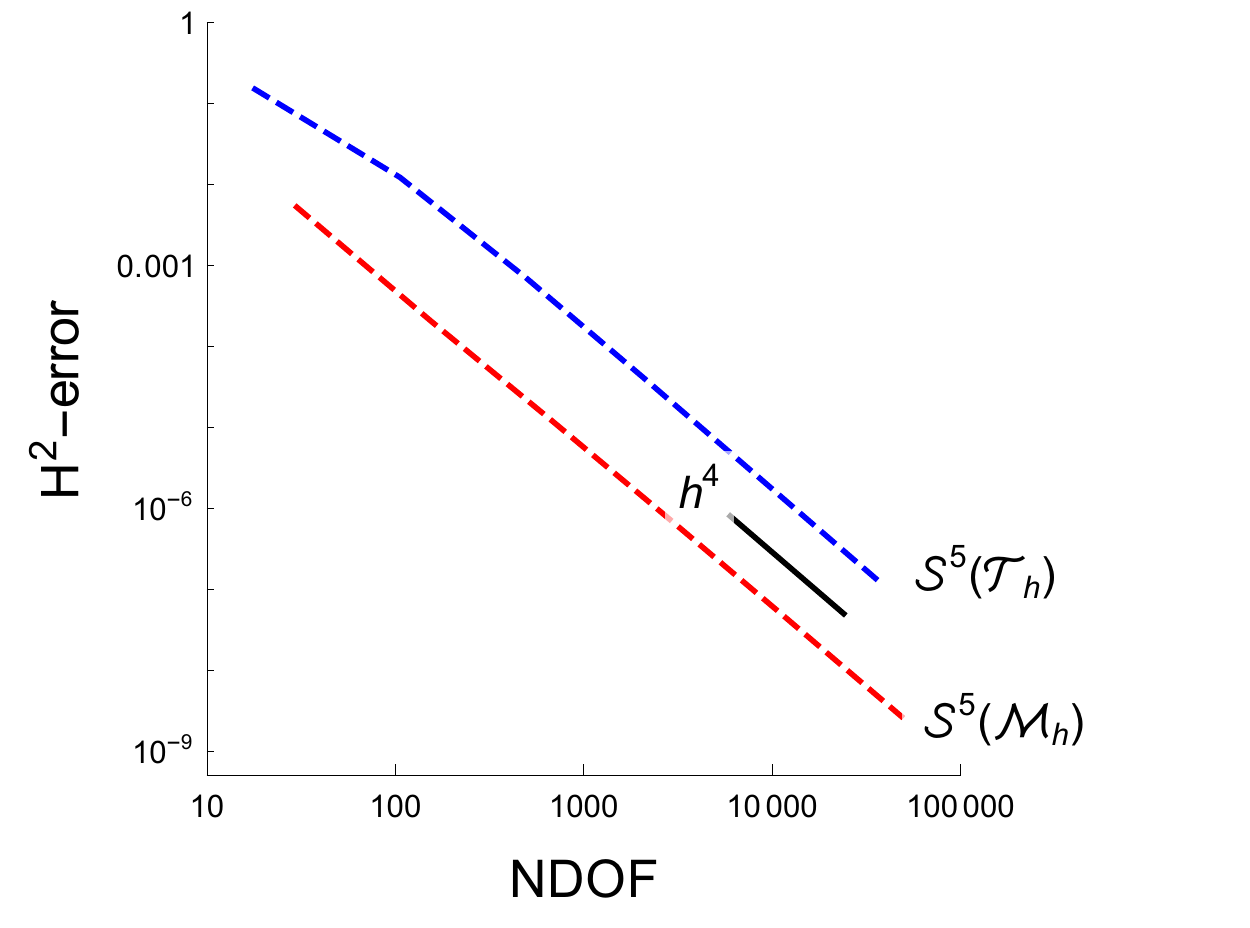} \\
Rel. $H^1$ error & Rel. $H^2$ error
\end{tabular}
\caption{Comparison of using the BS  quadrilateral spaces~$\globalspace^5(\Mesh_h)$ with the Argyris triangle spaces~$\globalspace^5(\mathcal{T}_h)$ 
for solving the biharmonic equation~\eqref{eq:problem_biharmonic} on the same computational domain defined either by a quadrilateral (top row, left) or a triangle mesh~(top row, middle). Exact solution~\eqref{eq:exact_trianglequad} (top row, right) and the resulting $L^{\infty}$ and relative $L^{2}$, $H^1$, $H^2$-errors (middle and bottom row). 
See Example~\ref{ex:example_trianglequad}.}
\label{fig:example_trianglequad1}
\end{figure}

\begin{figure}
\centering\footnotesize
\begin{tabular}{ccc}
\includegraphics[width=.3\textwidth,clip]{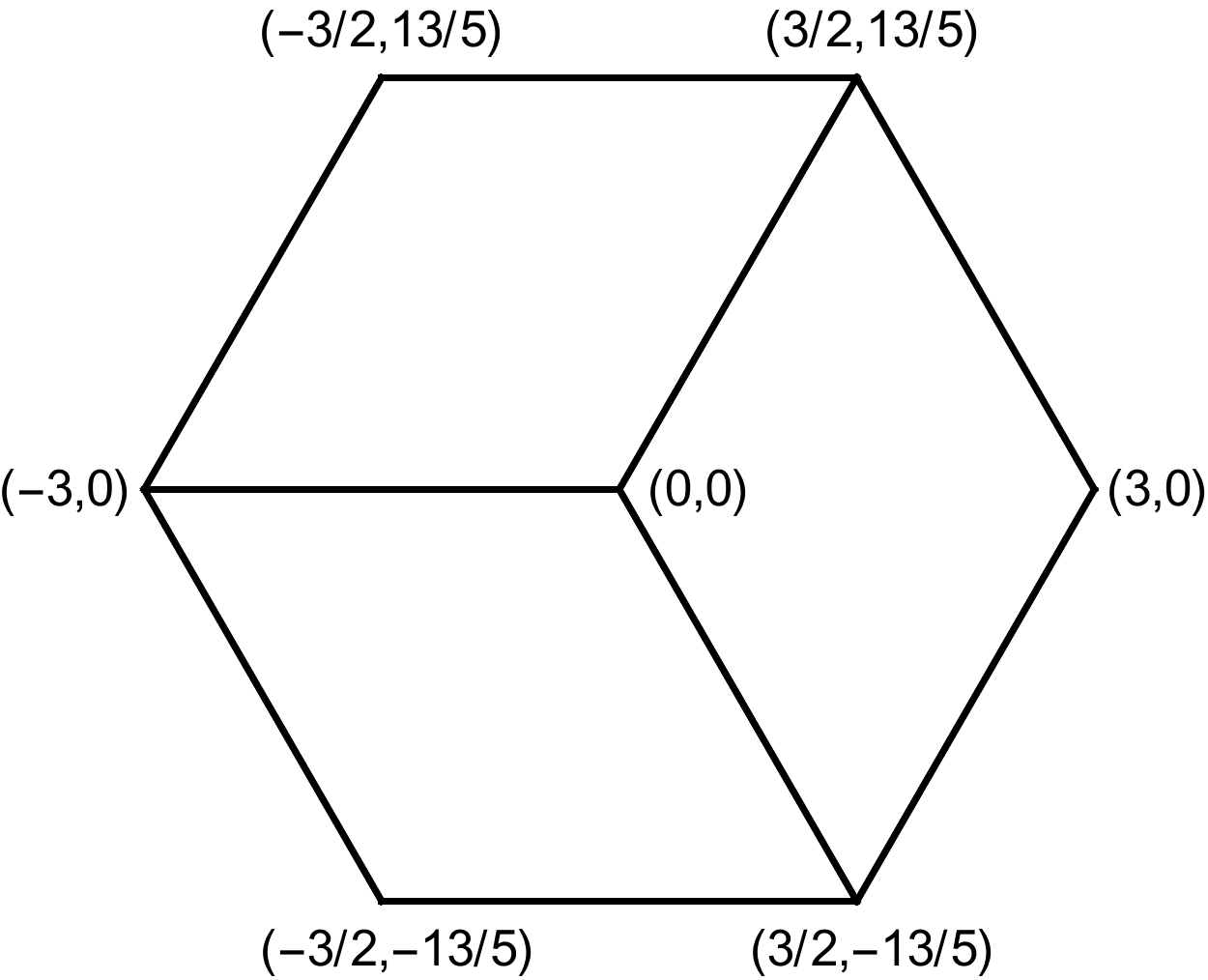} &
\includegraphics[width=.3\textwidth,clip]{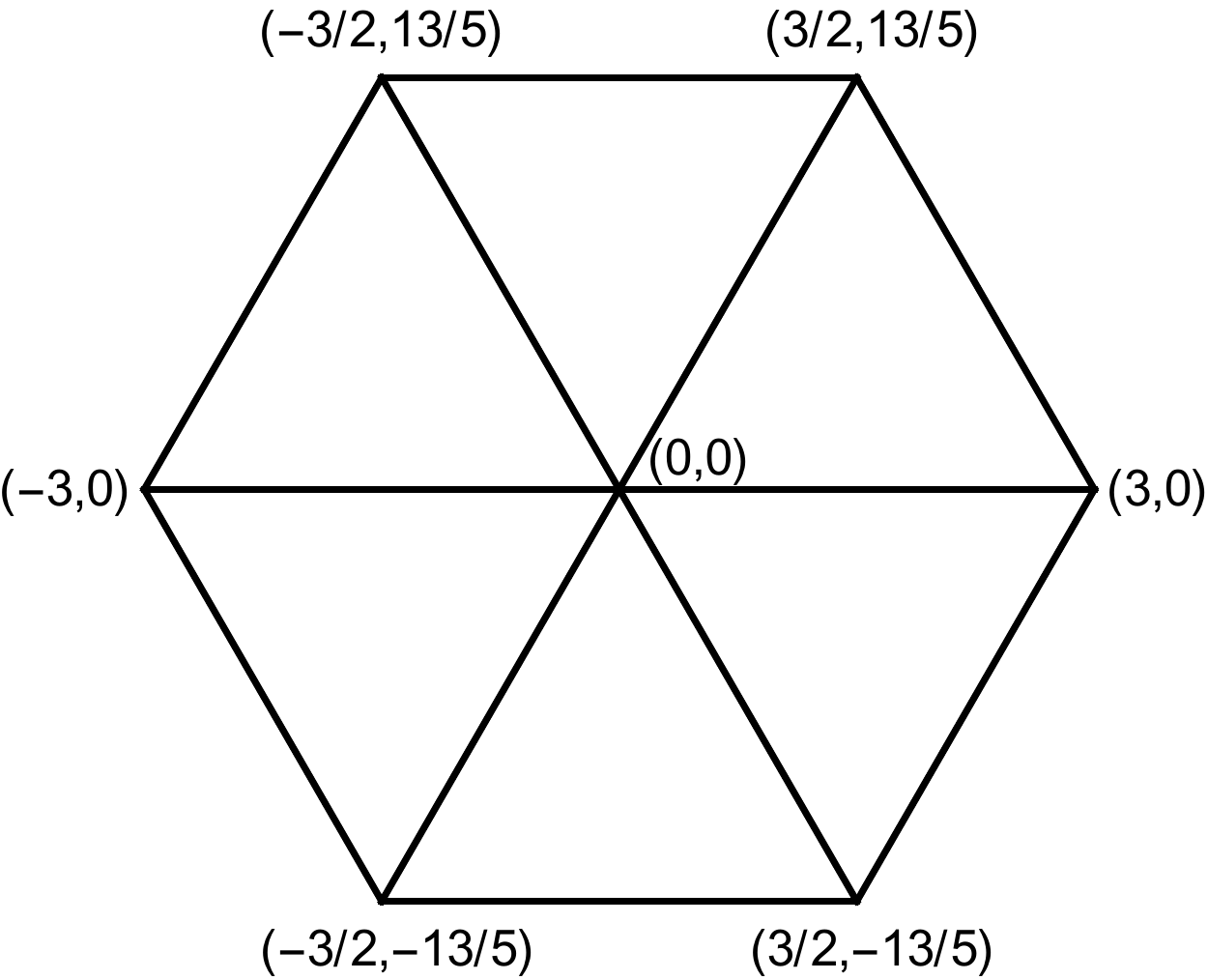} &
\includegraphics[width=.3\textwidth,clip]{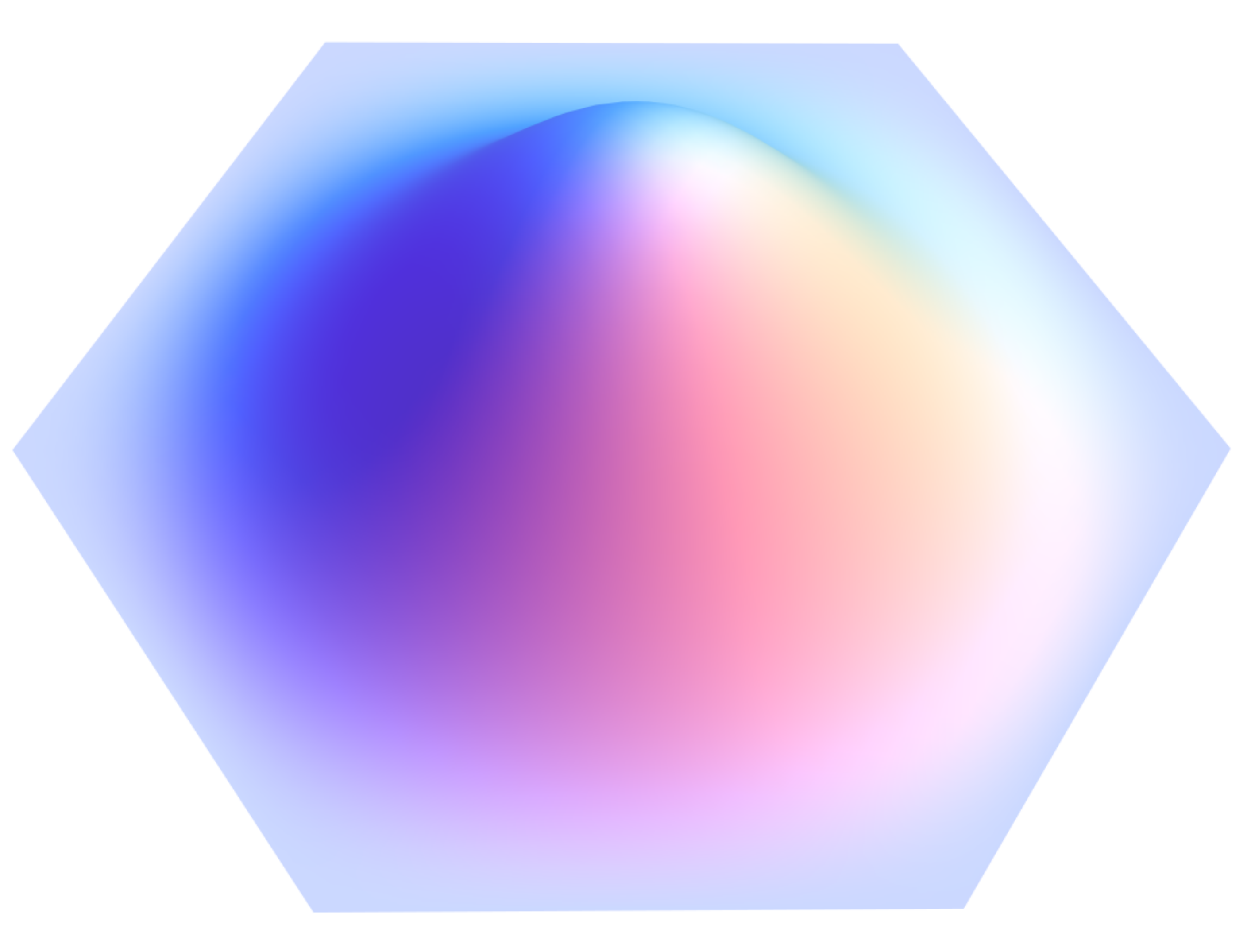} \\
Quadrilateral mesh & Triangular mesh & Exact solution
\end{tabular}
\begin{tabular}{cc}
\includegraphics[width=.45\textwidth,clip]{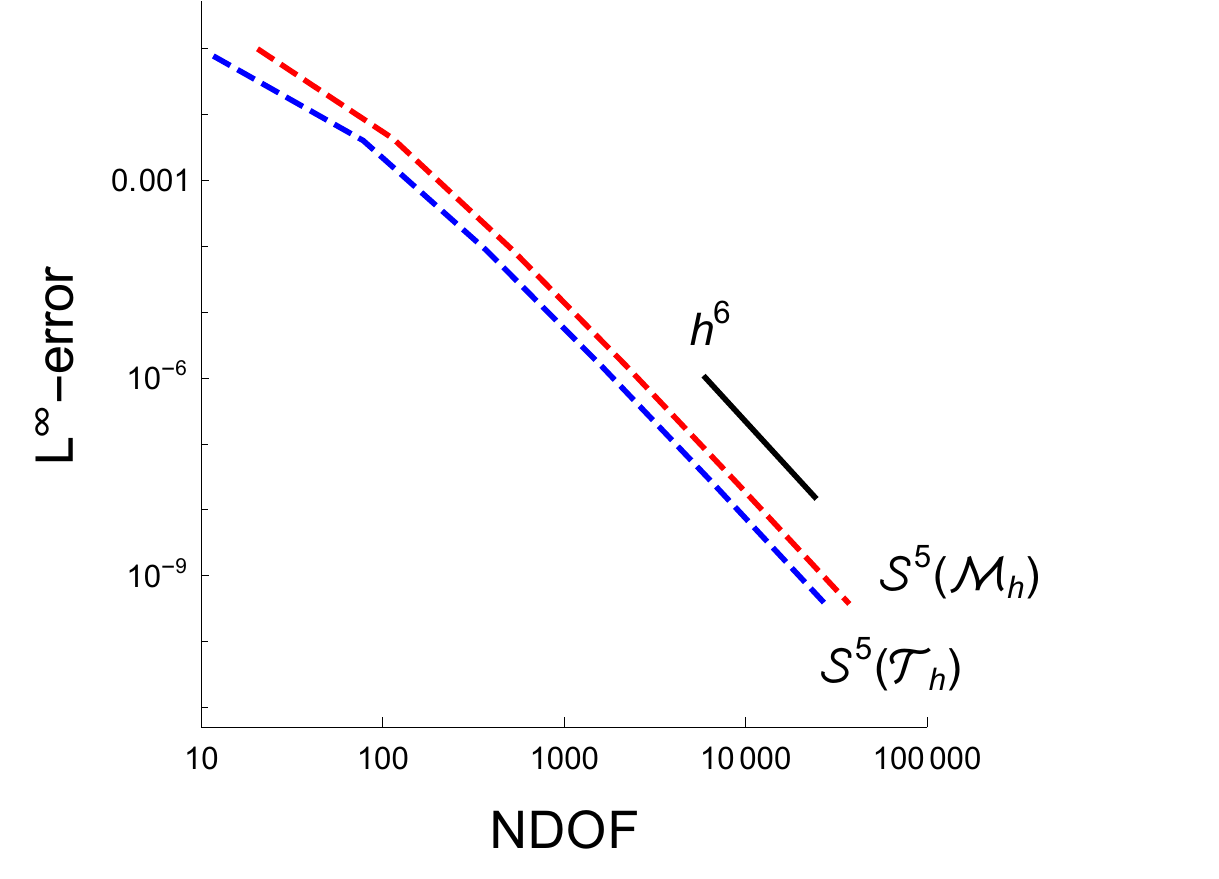} &
\includegraphics[width=.45\textwidth,clip]{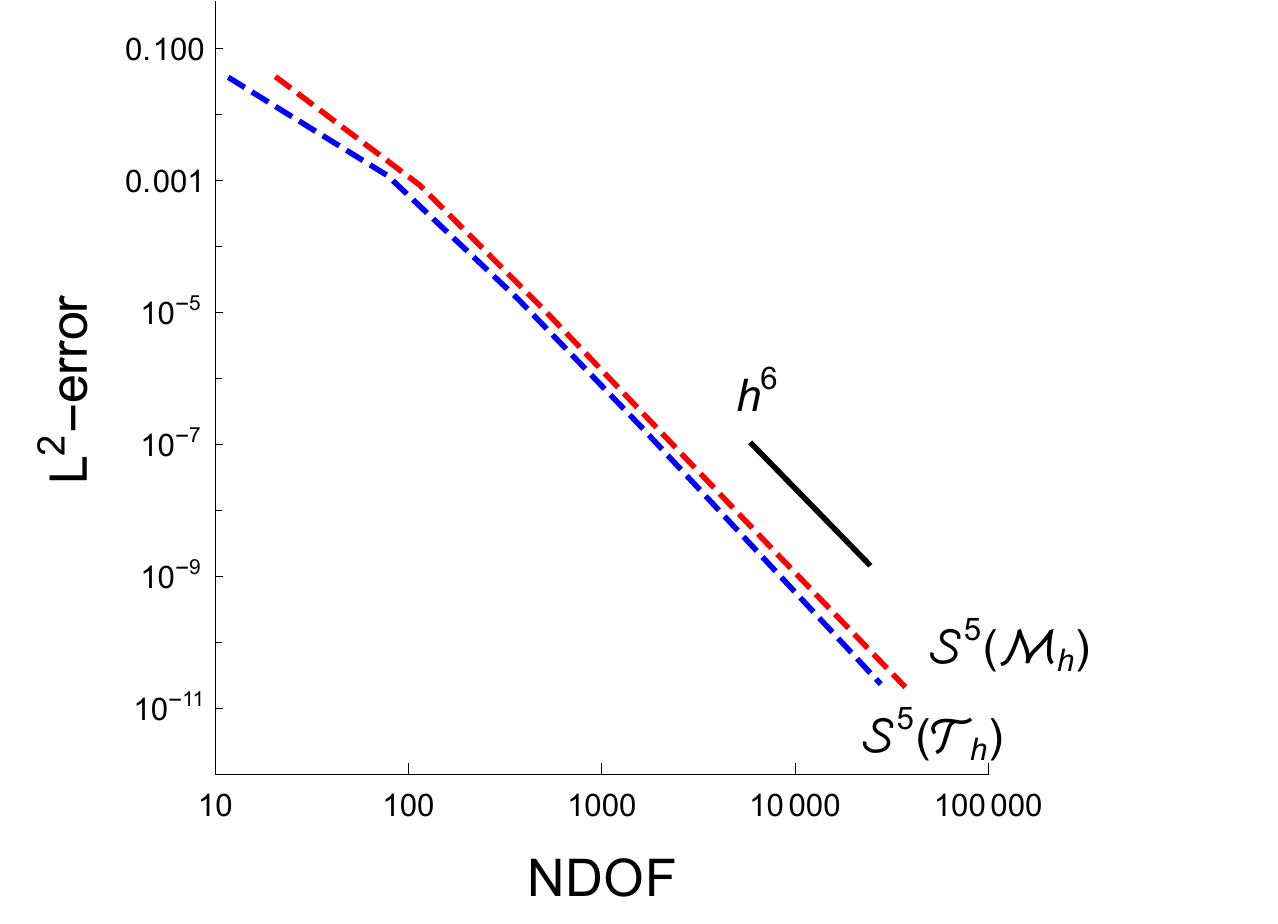} \\
$L^{\infty}$ error & Rel. $L^2$ error \\
\includegraphics[width=.45\textwidth,clip]{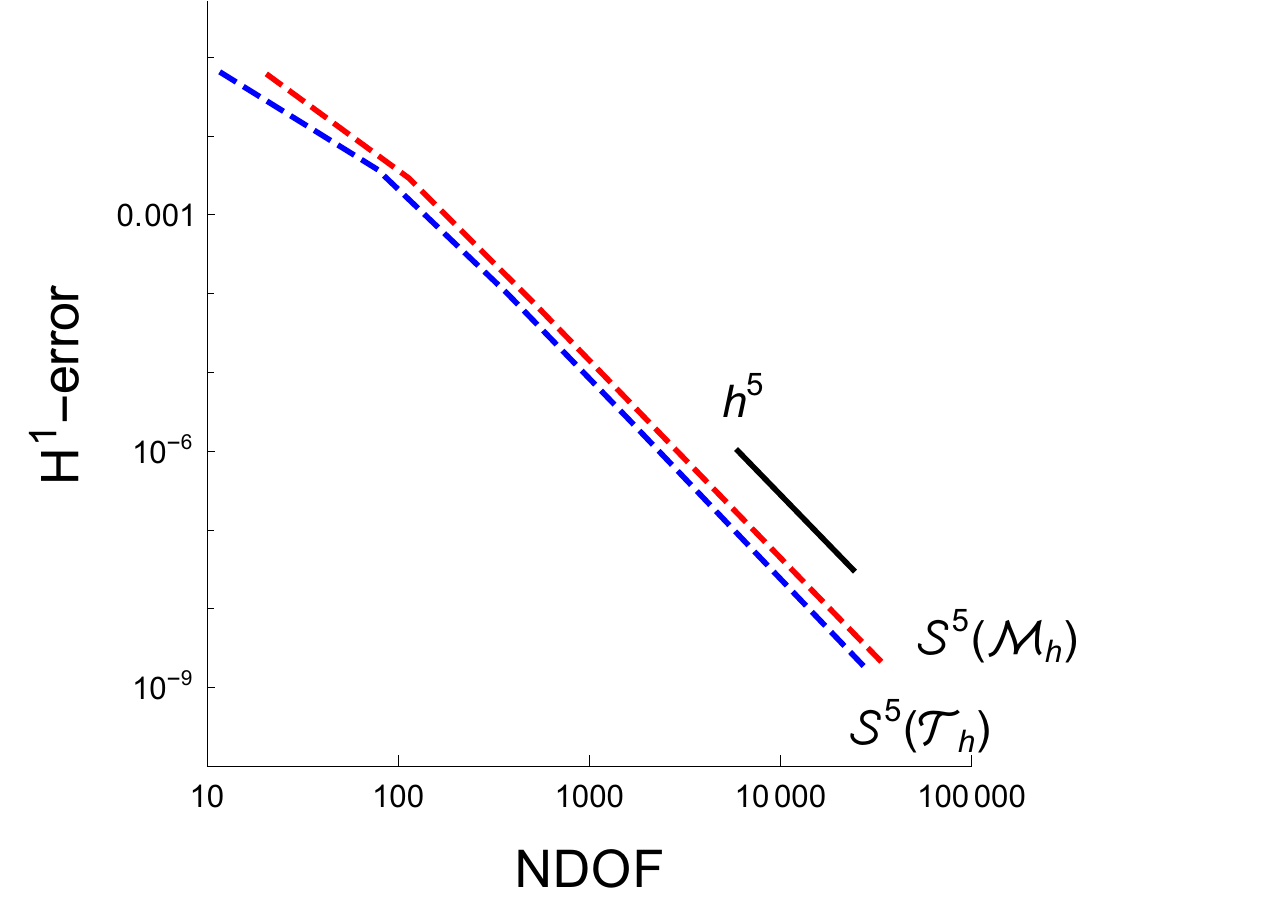} &
\includegraphics[width=.45\textwidth,clip]{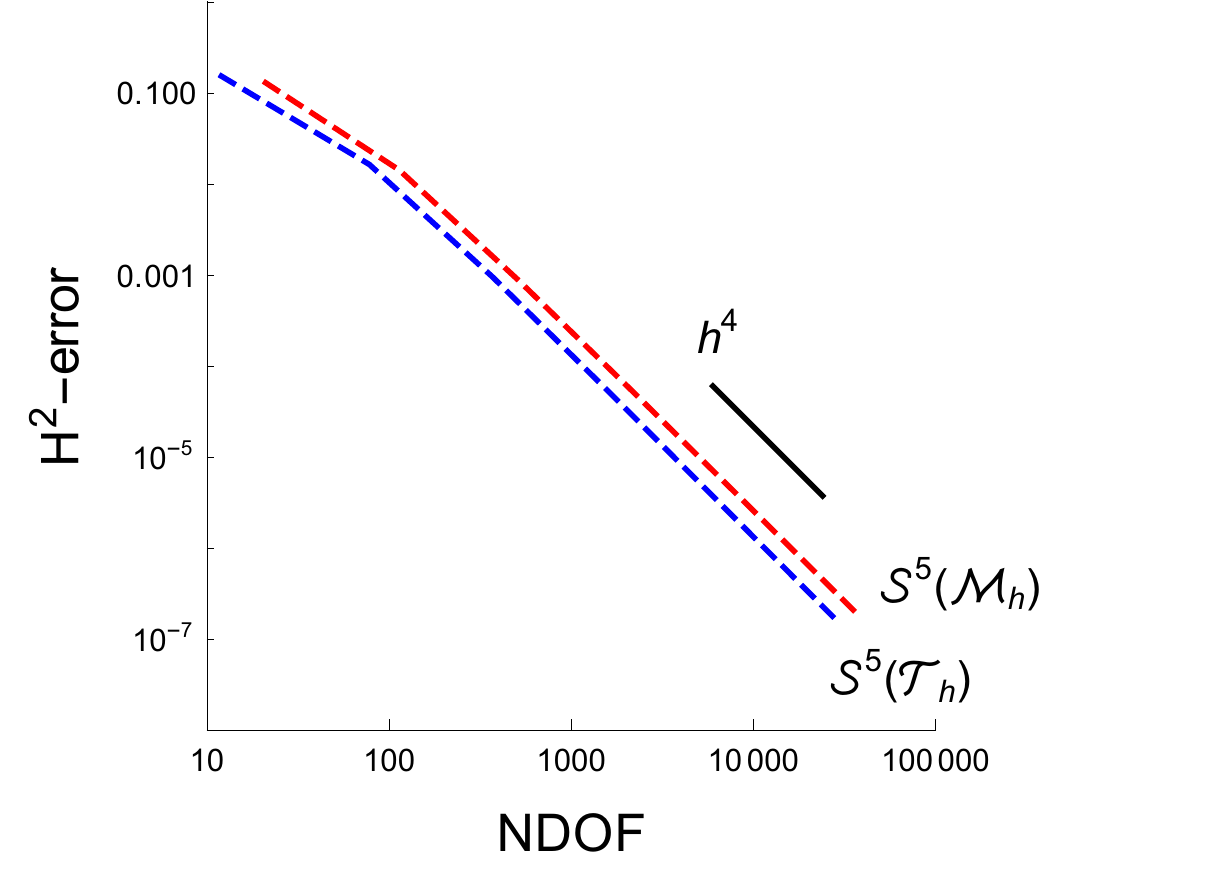} \\
Rel. $H^1$ error & Rel. $H^2$ error
\end{tabular}
\caption{Comparison of using the BS  quadrilateral spaces~$\globalspace^5(\Mesh_h)$ with the Argyris triangle spaces~$\globalspace^5(\mathcal{T}_h)$ 
for solving the biharmonic equation~\eqref{eq:problem_biharmonic} on the same computational domain defined either by a quadrilateral (top row, left) or a triangle mesh~(top row, middle). Exact solution~\eqref{eq:exact_trianglequad2} (top row, right) and the resulting $L^{\infty}$ and relative $L^{2}$, $H^1$, $H^2$-errors (middle and bottom row). 
See Example~\ref{ex:example_trianglequad}.}
\label{fig:example_trianglequad2}
\end{figure}
\end{example}

\newpage
\section{Conclusion} \label{sec:conclusion}

We have described the construction of a novel family of $C^1$ quadrilateral finite elements, extending the BS  quadrilateral construction from~\cite{BrSu05}, possessing similar degrees of freedom as the classical Argyris triangle~\cite{ArFrSc68}. The presented method allows the simple design of polynomial as well as of spline elements. Among others, we have introduced a simple and local basis for the $C^1$ quadrilateral space, and have stated for particular cases explicit formulas for the B\'{e}zier or spline coefficients of the basis functions. We have also studied several properties of the $C^1$ quadrilateral space such as the optimal approximation properties of the space. Furthermore, the  $C^1$ quadrilateral spaces are perfectly suited for solving fourth order PDEs, which has been demonstrated on the basis of several numerical examples solving the biharmonic equation on different quadrilateral meshes. 

Since the classical Argyris triangle space and the BS  quadrilateral space (and variants) presented here possess similar degrees of freedom, we are currently working on an approach to combine the $C^1$ triangle and quadrilateral element to construct a $C^1$ element for a mixed triangle and quadrilateral mesh. Further topics which are worth to study are e.g. the use of the $C^1$ quadrilateral elements for solving other fourth order PDEs such as the Kirchhoff plate problem, the Navier-Stokes-Korteweg equation, problems of strain gradient elasticity, and the Cahn-Hilliard equation, or the extension of our approach to quadrilateral meshes with curved boundaries.

\section*{Acknowledgments}

The research of G. Sangalli is partially supported by the European Research Council through the FP7 Ideas Consolidator Grant \emph{HIGEOM} n.616563, and
by the Italian Ministry of Education, University and Research (MIUR)
through the  ``Dipartimenti di Eccellenza Program (2018-2022) - Dept. of Mathematics, University of Pavia''. The research of M. Kapl is partially supported by the Austrian Science Fund (FWF) through the project P~33023. The research of T. Takacs is partially supported by the Austrian Science Fund (FWF) and the government of Upper Austria through the project P~30926-NBL. This support is gratefully acknowledged. 

\providecommand{\bysame}{\leavevmode\hbox to3em{\hrulefill}\thinspace}
\providecommand{\MR}{\relax\ifhmode\unskip\space\fi MR }
\providecommand{\MRhref}[2]{%
  \href{http://www.ams.org/mathscinet-getitem?mr=#1}{#2}
}
\providecommand{\href}[2]{#2}

\end{document}